\documentclass[tbtags, intlimits]{amsart}

\usepackage{amssymb, enumitem, cancel}
\usepackage{geometry}

\usepackage[foot]{amsaddr}
 
\usepackage{mathtools}
\usepackage{mathrsfs}
\usepackage{hyperref}
\usepackage{cleveref}
\usepackage{xfrac}
\usepackage{marginnote}
\usepackage{dsfont}
\usepackage{changepage}
\usepackage{xcolor}
\usepackage{soul}
\usepackage{extdash}

\newtheorem{thm}{Theorem}[section]
\newtheorem{prop}[thm]{Proposition}
\newtheorem{lem}[thm]{Lemma}

\crefname{thm}{Theorem}{Theorems}
\crefname{prop}{Proposition}{Propositions}
\crefname{lem}{Lemma}{Lemmas}
\crefname{cor}{Corollary}{Corollaries}

\crefname{claim}{Claim}{Claims}
\crefname{step}{Step}{Steps}

\theoremstyle{definition}
\newtheorem{defn}[thm]{Definition}
\newtheorem{ex}[thm]{Example}
\crefname{defn}{Definition}{Definitions}
\crefname{ex}{Example}{Examples}

\theoremstyle{remark}
\newtheorem{rmk}[thm]{Remark}
\crefname{rmk}{Remark}{Remarks}

\numberwithin{equation}{section}

\newcommand{\norm}[1]{\ensuremath\lVert#1\rVert}

\newcommand{\eqdef}{=\vcentcolon}
\newcommand{\defeq}{\vcentcolon=}

\renewcommand{\leq}{\leqslant}
\renewcommand{\geq}{\geqslant}
\renewcommand{\le}{\leqslant}
\renewcommand{\ge}{\geqslant}

\let\temp\phi
\let\phi\varphi
\let\varphi\temp
\let\temp\varepsilon
\let\varepsilon\epsilon
\let\epsilon\temp

\newcommand{\tildee}[1]{\widetilde{#1}}

\DeclareMathOperator{\Lip}{Lip}

\DeclareMathOperator{\tr}{tr}
\DeclareMathOperator{\spt}{spt}

\newcommand{\ssubset}{%
\Subset}

\newcommand{\N}{\mathbb{N}}
\newcommand{\Pc}{\mathcal{P}}

\newcommand{\de}{\partial}
\newcommand{\di}{\mathrm{d}}
\newcommand{\R}{{\mathbb{R}}}

\newcommand{\call}[1]{\ensuremath\mathcal{#1}}
\newcommand{\frk}[1]{\ensuremath\mathfrak{#1}}
\newcommand{\scr}[1]{\ensuremath\mathscr{#1}}
\newcommand{\bb}[1]{\ensuremath\mathbb{#1}}

\renewcommand{\bar}[1]{\ensuremath\overline{#1}}

\newcommand{\trn}{\mathsf{T}}
\newcommand\Sym{\scr S}

\newcommand{\mpush}{{}_\sharp}
\newcommand{\Law}{\call L}

\newcommand{\var}{\mspace{1.8mu}\cdot\mspace{1.8mu}}

\newcommand{\Ds}{$\mathscr{D}\hspace{-0.5pt}$\Hyphdash*}
\newcommand{\LLs}{$\mathscr{L}$\hspace{-0.5pt}\Hyphdash*}
\newcommand{\Dop}{\call D}
\newcommand{\LLop}{\call L}

\newcommand{\QA}{Q}

\newcommand\mres\llcorner

\newcommand\bs\boldsymbol
\newcommand\tbs[1]{\bs{\tilde{#1}}}
\newcommand\msf[1]{\mathsf{#1}}
\newcommand\msfb[1]{\bs{\msf{#1}}}
\newcommand\bbs[1]{\bar{\bs{#1}}}
\newcommand\hbs[1]{\hat{\bs{#1}}}

\makeatletter 
    \def\@makefnmark{\@textsuperscript{\normalfont\texttt{\@thefnmark}}} 
\makeatother

\begin{document}

\date{\today}

\title[Nonsymmetric semimonotone NS: a priori estimates and large populatation limits]{A priori estimates and large population limits for some nonsymmetric Nash systems with semimonotonicity}
\author[M.\ Cirant]{Marco Cirant\textsuperscript{(1)}}
\address{\textsuperscript{(1)}Dipartimento di Matematica ``T.\ Levi-Civita'' \\ Università degli Studi di Padova \\ Via Trieste 63 \\ 35121 Padova, Italy.}
\author[D.\ F.\ Redaelli]{Davide Francesco Redaelli\textsuperscript{(2)}}
\address{\textsuperscript{(2)}Dipartimento di Matematica \\ Università di Roma Tor Vergata \\ Via della Ricerca Scientifica 1\\ 00133 Roma, Italy.}
\email{cirant@math.unipd.it, redaelli@mat.uniroma2.it}


\subjclass[2010]{}

\begin{abstract} We address the problem of regularity of solutions $u^i(t, x^1, \ldots, x^N)$ to a family of semilinear parabolic systems of $N$ equations, which describe closed-loop equilibria of some $N$-player differential games with Lagrangian having quadratic behaviour in the velocity variable, running costs $f^i(x)$ and final costs $g^i(x)$. By global (semi)monotonicity assumptions on the data $f=(f^i)_{1 \leq i \leq N}$ and $g=(g^i)_{1 \leq i \leq N}$, and assuming that derivatives of $f^i, g^i$ in directions $x^j$ are of order $1/N$ for $j \neq i$, we prove that derivatives of $u^i$ enjoy the same property. The estimates are uniform in the number of players $N$. Such a behaviour of the derivatives of $f^i, g^i$ arise in the theory of Mean Field  Games, though here we do not make any symmetry assumption on the data. 

Then, by the estimates obtained we address the convergence problem $N \to \infty$ in a ``heterogeneous'' Mean Field  framework, where players all observe the empirical measure of the whole population, but may react differently from one another. We also discuss some results on the joint $N \to \infty$ and vanishing viscosity limit.
\end{abstract}

\maketitle

\setcounter{tocdepth}{1}
\tableofcontents

\section{Introduction}

Consider the following semilinear (backward) parabolic system for the unknowns $u^i = u_N^i  : [0,T] \times (\R^d)^N \to \R$, $i = 1, \ldots, N$,
\begin{equation} \label{ENS_NS}
\begin{cases}
-\de_t u^i - \tr((\sigma\msfb I + \beta \msfb J)D^2 u^i) + H^i(x^i, D_i u^i) + \displaystyle\sum_{\substack{1 \leq j \leq N \\ j \neq i}} D_pH^j(x^j, D_j u^j) \cdot D_j u^i = f^i \vspace{-0.3cm} \\
 u^i|_{t=T} = g^i
\end{cases}
\end{equation}
stated in $[0,T] \times (\R^d)^N \ni (t,\bs x) = (t,x^1,\dots,x^N)$. The main goal of this paper is to derive estimates on $u^i$ and their space/time derivatives that are \textit{stable} with respect to the number of equations; that is, uniformly in $N$.

The data are the Hamiltonians $H^i \colon \R^d \times \R^d \to \R$, the running and final costs $f^i, g^i \colon (\R^d)^N \to \R$, the diffusion parameters $\sigma > 0$ and $\beta \geq 0$ and the time horizon $T \geq 0$. The symbol $\msfb I = \msf I_{Nd}$ denotes the $Nd$-dimensional identity matrix and $\msfb J = \msf J_N \otimes \msf I_d$, where $\msf J_N = 1_N \otimes 1_N$ is the $N$-dimensional matrix of ones.\footnote{We use $\otimes$ to denote both the Kronecker product between matrices and the outer product between vectors.} $D^2 = D^2_{\bs x}$ denotes the Hessian with respect to $\bs x \in (\R^d)^N$ and $D_j = D_{x^j}$ denotes the gradient with respect to $x^j \in \R^d$. Finally, the symbol $\cdot$ denotes the Euclidean inner product, and $|\cdot|$ denotes the Euclidean norm. Hence, \eqref{ENS_NS} is a compact formulation of
\[
\begin{dcases}
\begin{multlined}
-\de_t u^i(t,\bs x) - \sigma \sum_{1 \leq j \leq N} \Delta_{x^j} u^i(t,\bs x) - \beta \sum_{1 \leq j, k \leq N} {\rm tr} \, D^2_{x^j x^k} u^i(t,\bs x) \\
\quad +H^i\big(x^i, D_{x^i} u^i (t,\bs x)\big) + \sum_{\substack{1 \leq j \leq N \\ j \neq i}} D_pH^j\big(x^j, D_{x^j} u^j (t,\bs x)\big) \cdot D_{x^j} u^i(t,\bs x) = f^i(\bs x) \\
\end{multlined} \\[-20pt]
u^i(T,\bs x) = g^i(\bs x).
\end{dcases}
\]
System~\eqref{ENS_NS} describes Markovian Nash equilibria in $N$-player differential games, in particular it characterises the value function $u^i$ of the $i$-th agent for each $i = 1, \ldots, N$. In our setting, agents control via feedbacks $\alpha^i = \alpha^i(t, \bs x)$ their own states, which are driven by the following $\R^d$-valued SDEs on $[0,T]$:
\begin{equation} \label{ENS_NcSDEs} % -D_i u^i(t,\bs X_t)
\di X_t^i = \alpha^i(t, \bs X_t) \,\di t + \sqrt{2\sigma}\,\di B_t^i + \sqrt{2\beta}\, \di W_t, \quad i \in \{1,\dots,N\},
\end{equation}
where %$u$ is the solution to \eqref{ENS_NS} and 
the $B^i_t$'s and $W_t$ are $d$-dimensional independent Brownian motions. The Brownian motions $B^i_t$ correspond to the individual noises, while $W_t$ is the so-called common noise, as it is the same for all the SDEs. The $i$-th agent aims at mimimising the following cost functional
\[
\alpha^i \mapsto \mathbb E \biggl[\, \int_0^T \Bigl( \,L^i\big(X^i_s, \alpha^i(s, \bs X_s)\big) + f^i(\bs X_s) \Bigr)\, \di s + g^i(\bs X_T) \biggr],
\]
where $L^i$ is the Lagrangian, related to $H^i$ by the identity
\[
H^i(x,p) = \sup_{a \in \R^d} \bigl( - a\cdot p - L^i(x,a) \bigr).
\]
It is well known that the choice $\alpha^{*,i}(t, \bs x) = -D_pH^i\big(x^i, D_i u^i(t, \bs x))$ characterises Nash equilibria, see for instance \cite{cardelar, Friedman72}. Moreover, since one expects uniqueness of solutions to \eqref{ENS_NS} by its (uniformly) parabolic nature, such equilibria are unique. As a model family of Hamiltonians, we have in mind the quadratic one
\[
H^i(x^i, p) = \frac12|p|^2, \qquad \text{for all $1 \le i \le N$}.
\]

For \textit{fixed} $N$, existence and uniqueness results for semilinear systems of the form (or close to) \eqref{ENS_NS} are now classical, see for example \cite{BF90, BF02, Friedman72, LSU}. Nevertheless, the methods that are involved seem to be strongly sensitive to the dimension $N$. The main reason is that one typically wants to employ parabolicity to get existence, uniqueness and regularity, but its smoothing effect may deteriorate as $N$ increases, while the strongly coupling transport term $\sum D_pH^j(D_j u^j) \cdot D_j u^i$ may take over.  A concrete example of this phenomenon can be observed in Linear-Quadratic problems (see for example \cite{DFT}); within this special framework, the study of the Nash system as $N \to \infty$ can be recast into a vanishing viscosity limit for a nonlinear transport PDE (on a fixed Euclidean space). On top of that, the quadratic behavior in the gradients is critical: for some quadratic systems, one may even have the failure of Lipschitz estimates \cite[p.~588]{LSU}. In fact, there are very few results for problems on unbounded domains, like the whole Euclidean space \cite{cardelar2, Jackson}. A byproduct of our a priori estimate is an existence theorem for \eqref{ENS_NS} with bounded data, that will be presented in the Appendix.

As we previously announced, the main goal of this work is to study the stability of some estimates independently on the number of equations $N$. A remarkable feature of the approach presented here is that there is no need to assume the a priori existence of an asymptotic structure, in the $N \to \infty$ limit; that is, our estimates are \textit{non-asymptotic} in nature. This is in the same spirit of \cite{JackLack}, where large population control problems are discussed. 

We developed our strategy in the setting of purely quadratic Hamiltonians $H^i(x^i,p) = \frac12|p|^2$. To make the methodology as clear and light as possible, we will proceed in the introduction under the assumption that we are in this special case. We will come back to more general Hamiltonians at the end of the introduction, focusing on the additional (technical) step that is needed to handle them. The structure of the presentation will be the same in the rest of the paper: in most of the sections the quadratic case will be developed, and one final section will be devoted to the more general situation.

%The main reason behind such a choice is that this setting is the most natural one to introduce the main ideas behind our strategy, as it enjoys a precise correspondence between the piece of information required to obtain the estimates of \Cref{sec_esti} and that needed to ``propagate'' them in the whole time horizon; indeed, such pieces of information are, respectively, the (local in time) \Ds semimonotonicity of $u$ and the propagation (in time) of such a semimonotonicity.

\medskip

We begin our discussion by focusing on the fully-symmetric context of Mean Field Games (MFGs), that was a main motivation for this investigation.

\subsection{Symmetric players, MFGs and monotonicity} If interactions are assumed to be \textit{symmetric} (that is, $f^i(\bs x) = f(x^i, \bs x^{-i})$, $g^i(\bs x) = g(x^i, \bs x^{-i})$ and $f,g$ are symmetric in the variable $\bs x^{-i}$, which is just the vector $\bs x$ with the $i$-th component omitted), one expects, borrowing intuitions from Statistical Mechanics, that a simplified limit model can be derived in the limit $N \to \infty$. In this direction, the theory of Mean Field  Games, originated by the works \cite{hcm, LL07}, has seen during the last fifteen years extensive development in many fields of mathematics and applied sciences. Within this framework, the problem can be embedded into the space of probability measures $\Pc(\R^d )$: one may write
\[
f^i(\bs x) = F(x^i, m_{\bs x^{-i}}),
\]
where $m_{\bs x^{-i}} \defeq \frac1{N-1} \sum_{i \neq j} \delta_{x^j}$ is the empirical measure and $F : \R^d \times \Pc(\R^d ) \to \R$, and the same for $g^i$ (everything that will be said on $f^i$ in this introduction will hold for $g^i$ also). Note that not only is one assuming symmetry, but also that interactions between players are ``small'', since a variation in $x^j$ produces a variation in $F$ (and then in $f^i$) of order $1/N$ if $j \neq i$. Nevertheless, this smallness is compensated by the number $N$ of interactions. Ideally, since each $u^i$ is also symmetric in $\bs x^{-i}$ one may hope in the convergence of $u^i(t, \bs x)$ to $U(t, x^i, m_{\bs x^{-i}})$ as $N \to \infty$, where $U$ is a function over $[0,T] \times \R^d \times \Pc(\R^d )$ that should be characterised by being the solution of the formal \textit{limit} of \eqref{ENS_NS}. This limit object is called the Master Equation, and it is a nonlinear transport equation over the set of probability measures. The function $U(t,x,m)$ can be regarded as the value function of a player with state $x$ at time $t$, observing a population distribution $m$. This approach to the limit of $N$-player symmetric games has been discussed for the first time by P.-L. Lions during his lectures at the Coll\`ege de France \cite{LionsSeminar} (see in particular the lecture given on December 5th, 2008). Since then, the derivation of the estimates needed to complete the program, in particular that the derivatives of $u^i$ behave in the same way as those of $f^i, g^i$, remained an open problem. Note that the related question of large population limit of many agents \textit{optimal control} problems can be addressed as well by deriving analogous estimates on solutions to certain Hamilton-Jacobi equations. Within this context, different results are now available in the literature, see for example \cite{CarSou, DDJ24, GM22, GS15, LMMZ25} and references therein. Nevertheless, we stress that the analysis in our case, that is on \textit{differential games}, is way more involved, as we need to handle a system of strongly coupled Hamilton-Jacobi equations increasing in number, rather than a single equation.

Such a convergence result is of primary importance. Not only does it justify the use of the limit object (broadly speaking, the MFG theory) to approximate Nash equilibria of the $N$-players problem, but it carries fundamental implications. For instance, the optimal control in the limit for each agent reads as $-D_x U(t, x^i, m)$, so it has \textit{decentralised} structure, or, in other words, it is \textit{open-loop} in nature (or, more precisely, ``distributed''). This fact reflects into the possibility of characterising Nash equilibria in the limit by a simple (compared to the Nash system \eqref{ENS_NS}) backward-forward PDE system that goes under the name of MFG system.

The convergence problem has been matter of extensive research, and the fundamental contribution from the PDE perspective has been given by \cite{CDLL}. Interestingly, since the estimates for the Nash system \eqref{ENS_NS} seemed to be out of reach, the authors proposed a different approach: from the stability properties of the MFG system, one may produce smooth solutions $U(t,x,m)$ to the Master Equation. These can be in turn projected over empirical measures; the $v^i(t,\bs x) = U(t, x^i, m_{\bs x^{-i}})$ are then shown to be ``almost'' solutions to the Nash system. Comparing the $v^i$ with the solutions $u^i$, one observes a discrepancy that vanishes as $N\to\infty$.

This program has been carried out under the so-called \textit{Lasry--Lions} monotonicity assumption
\begin{equation}\label{FLLm}
\int_{\R^d} \big(F(x,m_1) - F(x,m_2)\big)(m_1-m_2)(\di x) \ge 0 \qquad \forall m_1,m_2 \in \Pc(\R^d),
\end{equation}
which is crucial to get uniqueness and stability at the level of the MFG system. Later on, other conditions guaranteeing uniqueness of solutions have been formulated \cite{Ahuja16, BM24, CiPo, gangbo, GraMe, MouZhang}.  Among them, we focus our attention now on 
the \textit{displacement} monotonicity assumption
\begin{equation}\label{FDm}
\int_{\R^d \times \R^d} \big(D_x F(x,m_1) - D_x F(y,m_2)\big)\cdot (x-y)\, \mu (\di x, \di y) \ge 0
\end{equation}
for all $m_1, m_2$ and $\mu$ having $m_1$ and $m_2$ as first and second marginal, respectively. This has been used successfully to obtain the well-posedness of the Master Equation in \cite{gangbo}, and later in \cite{MesMou} to deduce uniqueness of solutions to the MFG system. %Note that smoothness of solutions to the Master Equation obtained in \cite{gangbo} could be employed to study the convergence problem
An important observation is that the monotonicity assumptions \eqref{FLLm} and \eqref{FDm} both \textit{propagate} at the level of the (limit) value function $U(t,x,m)$; that is, \eqref{FLLm} or \eqref{FDm} hold for $U(t, \cdot, \cdot)$ in place of $F(\cdot, \cdot)$ for any $t \in [0,T]$.

\subsection{The estimates} To attack the convergence problem \textit{directly} from the Nash system (that is, without employing smooth solutions of the Master Equation), one has to look for estimates of the following form:
\begin{equation}\label{desired_est}
|D_{x_i} u^i |,\, |D_{x_i} (D_{x_i} u^i) |\, \lesssim\, 1, \qquad |D_{x_j} u^i |,\, |D_{x_j} (D_{x_i} u^i) |\, \lesssim\, \frac1N \qquad\quad \forall \, i, j \neq i.
\end{equation}
Roughly speaking, if such bounds are uniform in $N$, a compactness argument allows to define (up to subsequences) a limit value function $U$ and a limit drift $D_x U$, which is the optimal feedback of a typical player. These will be Lipschitz continuous in the $(x, m)$ variables. Our strategy to derive these estimates originated within the displacement monotone setting, and it has been developed first in some simple linear-quadratic case \cite[Section 3]{CR24}. It can be described as follows.

As we already mentioned, under the monotonicity condition \eqref{FDm} one knows that $U(t, \cdot, \cdot)$ also satisfies \eqref{FDm} for any $t \in [0,T]$, see for example \cite[eq.~(4.7)]{MesMou}. Hence, by evaluating such an inequality on empirical measures, and assuming that $u^i(t, \bs x)$ is close to $U(t, x^i, m_{\bs x^{-i}})$, one observes that for all $t \in [0,T]$
\begin{equation}\label{osintro}
\sum_{1 \le i \le N} \big(D_i u^i (t, \bs x) - D_i u^i (t, \bs y)\big) \cdot (x^i - y^i) \ge - c_N |\bs x- \bs y|^2,
\end{equation}
where $c_N$ vanishes as $N \to \infty$ (see Remark \ref{rmk_relationship}). Recall that $\alpha^{*,i}(t, \bs x) = -D_i u^i (t, \bs x)$ is the optimal drift for the $i$-th player. Therefore, the previous inequality reads as a \textit{one-sided Lipschitz condition} (or dissipativity) on the \textit{global} drift vector $\bs \alpha^*$. This property turns out to be crucial in deriving estimates on the derivatives of $u^i$. For example, one can deduce gradient bounds that are independent of the dimension $N$ for $u^i$, which solves in fact
\[
-\de_t u^i - \tr((\sigma\msfb I + \beta \msfb J)D^2 u^i) + \displaystyle\sum_{1 \leq j \leq N} (-\alpha^{*,j}) \cdot D_j u^i = f^i + \dfrac12 |D_i u^i|^2,
\]
by using doubling of variables methods (see Lemma {\ref{ENS_estd1} below). This observation on the structure of $\bs \alpha^*$ suggests to pursue the following program:
\begin{enumerate}
\item assume first that the vector $(D_i u^i)$ satisfies the one-sided Lipschitz condition \eqref{osintro} with some $c_N = M > 0$ (so, uniformly in $N$), and, with this condition in force, prove the desired estimates on the derivatives of $u^i$;
\item employ the estimates (that will depend on the value of $M$, but not on $N$), along with the full structure of the Nash system to show that in fact the vector $(D_i u^i)$ satisfies \eqref{osintro} with $c_N = M/2$, at least for $N$ large enough.
\end{enumerate}
This shows that the Nash system enforces one-sided Lipschitz estimates on the optimal drift, or, in other words, \eqref{osintro} defines (with suitable $c_N = M$) an invariant set for the equilibrium controls. Practically, it allows to show that solutions $u^i$ satisfy the one-sided Lipschitz condition for some constant that does not depend on $N$, implying the desired estimates on $u^i$.

We will first show that this strategy works in a framework which resembles the one of displacement monotonicity in the symmetric case. In particular, we will prove that such an $M$ exists, and it has to be smaller and smaller as $T$ increases. We stress that we are \textit{not} going to assume any symmetry on the data, but only the fact that derivatives of $f^i, g^i$ with respect to $x^j$ behave as in the MF (Mean Field) case. This is why we will use the terminology \textit{Mean-Field-like} costs, see \hyperlink{MF}{{\bfseries{(MF)}}} below. Similarly, we will not assume \eqref{FDm}, but rather formulate the condition 
\[
\sum_{1\leq i \leq N} \,(D_i f^i(\bs x) - D_i f^i(\bs y))\cdot(x^i-y^i) \geq - C |\bs x - \bs y|^2 \quad \forall \, \bs x, \bs y \in (\R^d)^N
\]
which reads like \eqref{FDm} once it is specialized to the MF setting and $C=0$. Since we are going to allow for a negative right-hand side, we will refer to this case as the \textit{\Ds semimonotone} one, see \hyperlink{DS}{{\bfseries{(DS)}}} below.

Let us now come back to the Lasry--Lions monotonicity assumption \eqref{FLLm}, and discuss how the previous program can be adapted to this case. It is convenient to look at the one-sided Lipschitz condition \eqref{osintro} via the pointwise formulation (cf.~Remark \ref{ENS_rmkLLtoDM})
\begin{equation}\label{osintrod}
\sum_{1 \le i, j \le N} D^2_{ji} u^i (t, \bs x) \xi^i \cdot \xi^j \ge -c_N |\bs\xi|^2 \qquad \forall\, \bs\xi \in (\R^d)^N.
\end{equation}
It turns out that, under Lasry--Lions monotonicity, the counterpart of \eqref{osintrod} is (as before by projecting onto empirical measures the Lasry--Lions monotonicity of $U$)
\[
\sum_{\substack{1 \le i, j \le N \\ i \neq j}} D^2_{ji} u^i (t, \bs x) \xi^i \cdot \xi^j = \sum_{1 \le i, j \le N} D^2_{ji} u^i (t, \bs x) \xi^i \cdot \xi^j - \sum_{1 \le i \le N} D^2_{ii} u^i (t, \bs x) \xi^i \cdot \xi^i  \ge -c_N |\bs\xi|^2.
\]
Therefore, we have now an \textit{off-diagonal} one-sided Lipschitz information on the drift. To get the full one-sided Lipschitz control, it is necessary to bound from below the term $\sum_i D^2_{ii} u^i (t, \bs x) \xi^i \cdot \xi^i$, which encodes a sort of ``global'' semiconvexity of the set of value functions $u^i$. This introduces an additional step in our analysis, that is handled as before. Under Lasry--Lions monotonicity, we show that  the Nash system enforces not only the one-sided Lipschitz estimates on the optimal drift $\bs \alpha^*$, but also this kind of global semiconvexity. In other words, in step (1) we assume that both $\sum_{i,j} D_{ij}u^i (t, \bs x) \xi^i \cdot \xi^j$ and $\sum_i D^2_{ii} u^i (t, \bs x) \xi^i \cdot \xi^i$  are bounded below, and in step (2) we verify that these bounds are improved. To do so, as before we do not require symmetry, but the Mean-Field-like assumption \hyperlink{MF}{{\bfseries{(MF)}}} and the global (with respect to $i$) condition
\[
\sum_{1\leq i \leq N} \big( f^i(\bs x) - f^i(\bs x^{-i}, y^i) - f^i(\bs y^{-i}, x^i) + f^i(\bs y) \big) \geq -\kappa | \bs x - \bs y |^2 \quad \forall \, \bs x, \bs y \in (\R^d)^N,
\]
$(\bs x^{-i},y) = (x^1,\dots,x^{i-1},y,x^{i+1},\dots,x^N)$, as the monotonicity condition. This inquality will be called the \textit{\LLs semimonotone} assumption (see \hyperlink{LS}{{\bfseries{(LS)}}} below), and becomes \eqref{FLLm} once it is specialized to the MF setting and $\kappa=0$.

Our main result on the estimates for $u^i$ is stated in Theorem \ref{ENS_mainthm}: assuming Mean-Field-like costs and either \Ds semimonotonicity or \LLs semimonotonicity, the desired (a priori) estimates follow for any fixed $T>0$, sufficiently small semimonotonicity constants and sufficiently large $N \in \N$. Note that we are not exactly obtaining \eqref{desired_est}; regarding second-order derivatives, we get a bound of the form $\sup_i \sum_{j\neq i} |D_j D_i u^i|^2 \lesssim 1/N$. This is still enough, in the symmetric case described below, to produce in the limit $N\to \infty$ a Lipschitz function (with respect to the Wasserstein $W_2$ distance).

\smallskip

A few comments on the two steps (1) and (2) are now in order. Step (2) is achieved, in the \Ds semimonotone case, in Section \ref{sec_Ds}. We evaluate the quantity $\sum_i \big(D_i u^i (t, \bs x) - D_i u^i (t, \bs y)\big) \cdot (x^i - y^i)$ along optimal trajectories, see \Cref{ENS_dscprop}. One observes some nice properties (signs\,\dots) that appear also when operating at the limit, with an additional term, which is proven to be of order $1/\sqrt{N}$ provided that derivatives of $u^i$ are suitably controlled up to the third order. A similar strategy is developed in Section \ref{sec_LLs} for the \LLs monotone case, see in particular \Cref{ENS_LLmprop}.

The estimates under the one-sided Lipschitz condition are the core of step (1) presented above, which are developed in Section \ref{sec_esti}. Mainly two techniques are employed there: the method of doubling variables and the Bernstein method. Both are well suited to exploit the a priori one-sided assumption on $D_i u^i$ (the former uses \eqref{osintro}, while the latter works well in the form \eqref{osintrod}). To reach the third-order derivatives, several steps are necessary: sometimes we address each equation separately (\Cref{ENS_estd1}, \Cref{ENS_uijestp}), sometimes we need to proceed by estimating derivatives of $u^i$ for all $i$ at the same time (\Cref{ENS_d2b}, \Cref{ENS_d2N}, \Cref{ENS_d2N'}, \Cref{ENS_d3b}).

\subsection{The large population limit} With the estimates obtained in the first part of the work, we address the convergence problem $N \to \infty$. First, in the full symmetric case, where all the players react in the \textit{same way} with respect to the empirical measure of the population, we are able to obtain the classical MFG limit (see \Cref{ENS_convfacile}). Note that we require here some slightly weaker assumptions than the ones that are employed in \cite{CDLL} (and more in the direction of \cite{MZ24}, see in particular \Cref{rmkonmfg}). Beyond this setting, the convergence problem in full generality appears to be at this stage a challenging open question, that is an active area of research: sequences of games with non-symmetric interactions may be sometimes described by limit models, and there are now many studies taking advantage of the theory of \textit{Graphons}, see for example \cite{Carmona2, CainesHuang, LackerSoretLabel}. In this setting, cost functionals are player-independent, but they may depend on \textit{weighted} empirical measures; if weights converge in a suitable sense as $N \to \infty$, then an asymptotic description is possible.

 We put ourselves here in a sort of different setting: each player observes the empirical measure as in the standard MFG theory, but players can react differently. More precisely, we assume that the costs $f^i$ are taken from a pool of possible costs; that is,
\[
f^i(\bs x) = F(\lambda^i, x^i, m_{\bs x^{-i}})
\]
where the parameter $\lambda^i \in [0,1]$. If $\lambda^i$ varies in a finite subset of $[0,1]$ for all $N$, one may think of a multi-population MFG as in \cite{BHL18, Cir15, tesisegu}, but $\lambda^i$ can actually vary in the continuum $[0,1]$. This setting, despite the heterogeneity, perfectly fits the assumption \hyperlink{MF}{{\bfseries{(MF)}}} (under suitable regularity of $F$), that merely requires derivatives of $f^i$ to behave as in the MFG case. To obtain a convergence result, we assume Lipschitz regularity of $F$ in the $\lambda$ variable. In Graphon MFG, cost functions are player-independent but they depend on weighted empirical measures that are player-dependent, that is $f^i(\bs x) = F(x^i, N^{-1}\sum_{j} \xi^N_{ij} \delta_{x^j})$, where $(\xi^N)_{N \in \N}$ is a sequence of (weighted) adjacency matrices. Here costs are player-dependent, but they ``see'' the standard empirical measure of the others. Note that we could replace the empirical measure $m_{\bs x^{-i}}$ by $N^{-1}\sum_{j} \xi^N_{ij} \delta_{x^j}$ in our analysis, provided that $|\xi_{ij}^N| \lesssim 1$ (in such case \hyperlink{MF}{{\bfseries{(MF)}}} would still hold), but we prefer to avoid this additional layer of complexity here.

Our convergence argument is in the spirit of \cite{LackerSoretLabel}. If one thinks of any player $i$ in terms of its label-state couple $(\lambda^i, X^i_t)$, it is natural to work in the limit with probabilities $\mu$ over $[0,1] \times \R^d$. More precisely, we are able to define a value function $U(\lambda, t, x,\mu)$ from limits of the Nash system. Furthermore, $U$ generates a solution of a generalised MFG system, that is a continuum of classical MFG systems that are parametrised by $\lambda$. This is the second main result of this work, and it is stated in \Cref{ENS_thmMFsumm}.

For simplicity, we proceed in this part without common noise; that is, $\beta = 0$. A compactness argument allows to define $U$ from limits of $u^i$  (\Cref{ENS_rmkconvdiuil}), and the convergence is strong enough to pass to the limit the notion of Nash equilibrium (via a propagation of chaos argument), see \Cref{ENS_convfacile} and \Cref{ENS_convfinita}; this in turn yields the MFG system. We first show how to get convergence in the usual MFG setting; that is, for $\lambda^i = \lambda^{N,i} = \lambda$ for all $N,i$. In the multi-population (but finite) case, the argument is similar, even though it is technically more involved. Finally, the general case is obtained by approximating via a finite (and increasing) number of populations.

Note that, since compactness arguments are involved, convergence is obtained \textit{up to subsequences}. Clearly, under displacement or Lasry--Lions monotonicity and in the classical homogeneous MFG setting one has that limit objects (MFG equilibria, or solutions to the MFG system, or solutions $U(t,x,m)$ of the Master Equation, \dots) are unique, hence convergence is a posteriori along the full sequence. Nevertheless, we prefer not to stress this point here, as in fact we are requiring the milder assumptions of \textit{semimonotonicity}, and we allow for a certain degree of heterogeneity among agents. While we are not aware of uniqueness theorems in such generality, we believe that they should be true, as consequence of the fact that $U(\lambda, t,x,\mu)$ and $D_xU(\lambda,t,x,\mu)$ are Lipschitz continuous. 

\subsection{The joint large population and vanishing viscosity limit} As a further application of our estimates, we discuss some results on the joint limit $\sigma = \sigma_N \to 0$ and $N \to \infty$, in the fully symmetric Mean Field case. In other words, we show that $u^i$ converges to some limit function $U_0$, defined in the space of measures as in the previous section, and $U_0$ characterises an equilibrium of a \emph{deterministic} MFG. A key assumption here is that $\sigma_N$ is allowed to vanish \textit{slowly enough} as $N \to \infty$, in particular slower than $(\log\log N)^{-1}$.

The main observation that unlocks the proof of such joint limit is that, while showing that the Nash system enforces a propagation of monotonicity (step (2) of the program discussed above), there is some room for letting $\sigma \to 0$, since such a propagation happens if some ``error'' term that is bounded by $C(M, \sigma)/\sqrt{N}$ vanishes as $N \to \infty$. The constant $C(M, \sigma)$ comes exactly from the estimates based on the one-sided Lipschitz assumption (step (1) above); by tracking carefully its dependence on $\sigma$, we can quantify how fast $\sigma$ is allowed to vanish.

We present this application in the \Ds semimonotone case only, being the \LLs semimonotone one way more delicate. This has to do with the fact that the Master Equation is known to have classical solutions in the displacement monotone setting even when $\sigma = \beta = 0$ (see \cite{BMM}), while in the Lasry--Lions monotone case one has to rely on weak solutions \cite{CS22}, since the limit object $U$ is not expected to be even Lipschitz continuous in the measure variable (unless $T$ is small enough, as in \cite{GS15, Mayo}). This delicate issue shows up clearly in our approach as the impossibility of controlling \textit{from below} the quantity $\sum_i D^2_{ii} u^i (t, \bs x) \xi^i \cdot \xi^i$ uniformly as $\sigma \to 0$; this is in fact the further step in the analysis of the Lasry--Lions monotone setting discussed above. We leave the analysis of this more complicated framework for future research.

\subsection{More general Hamiltonians}\label{subsgenH}

Coming back to the more general class of Hamiltonians of the form $H^i(x^i,p)$, a crucial point is that optimal drift of the $i$-th player becomes $\alpha^{*,i}(t, \bs x) = -D_p H^i(x^i, D_i u^i (t, \bs x))$, instead of simply being $\alpha^{*,i}(t, \bs x) = -D_i u^i (t, \bs x)$. Therefore, the estimates on solutions to the Nash system (that is step (1) of the previous strategy) are expected to continue to hold provided that we have a uniform in $N$ one-sided Lipschitz bound on
\[
\bigl(-D_p H^i(x^i,D_iu^i(t,\bs x))\bigr)_{1 \leq i \leq N};
\]
that is,
\begin{equation} \label{generalDmon}
\sum_{1 \leq i \leq N} \bigl( D_p H^i(x^i,D_iu^i(t,\bs x)) - D_pH^i(y^i,D_iu^i(t,\bs y)\bigr) \cdot (x^i-y^i) \geq -\tilde M |\bs x - \bs y|^2,
\end{equation}
for some $\tilde M \geq 0$ independent of $N$, for all $t \in (\tau,T]$ and $\bs x, \bs y \in (\R^d)^N$. On the other hand, the bound we are able to propagate by following step (2) of the previous strategy remains the semimonotonicity of $u$. Note that, for $H$ being $C^2$, \eqref{generalDmon} is equivalent to
\begin{equation} \label{genoslequiv}
\call H_1 + \call H_0 (D^2_{ji} u^i)_{1 \leq i,j \leq N} \geq - \tilde M \msfb I \qquad \text{on} \ (\tau,T] \times (\R^d)^N,
\end{equation}
with
\begin{equation} \label{defnH01}
\begin{gathered}
\call H_0(t,\bs x) \defeq \mathrm{diag}\bigl(D^2_{pp} H^i(x^i,D_iu^i(t,\bs x))\bigr)_{1 \leq i \leq N},
\\
 \call H_1(t,\bs x) \defeq \mathrm{diag}\bigl(D^2_{px} H^i(x^i,D_iu^i(t,\bs x))\bigr)_{1 \leq i \leq N},
 \end{gathered}
\end{equation}
so in general a semimonotonicity bound for $u$ (that is, a lower bound on $(D^2_{ji} u^i)_{1 \leq i,j \leq N}$) does not imply one like \eqref{generalDmon}; not even when, for instance, if $p \mapsto H^i(x^i,p)$ is strongly convex and $D^2_{px} H^i \equiv 0$. As the sole propagation of semimonotonicity  seems to be no longer sufficient to close the estimates, we prove a further control from  \emph{above} of the same matrix: the operator norm of $(D^2_{ji} u^i)_{1 \leq i,j \leq N}$ is in fact strong enough to propagate the one-sided Lipschitz control \eqref{generalDmon} whenever the second derivatives of $H^i$ are bounded uniformly in $N$.

A road map for proving a more general version of the estimates can be then summarized as follows:
\begin{enumerate}
\item prove that \eqref{generalDmon} is sufficient in order to obtain the estimates of \Cref{sec_esti};
\item use the estimates to propagate a \emph{two-sided} Lipschitz bound for $(D_iu^i)_{1 \leq i \leq N}$.
\end{enumerate}
This program will be detailed in the last \Cref{sec_genH}.

\subsection{Further remarks} In this work, we propose a new PDE approach to the quantitative study of (Markovian) Nash equilibria in large population games. The analysis on the regularity of solutions to the Nash system developed here is, as we said, \textit{non-asymptotic}: it does not exploit the presence of a limit model, but it just requires some suitable scaling of the derivatives of the data with respect to $N$. Under the further assumption that certain asymptotic models can be formulated, this analysis leads to a convergence result that embraces in a unified way two scenarios of monotonicity. In addition, by employing the notion of semimonotonicity, we can consider at the same time the large time horizon $T$ case, where one needs almost monotonicity, and the short time $T$ case, where (almost) no monotonicity is needed, as in \cite{CCPjems}.

Besides the aforementioned approach via the Master Equation, let us mention that different probabilistic methods have been developed to tackle the convergence problem in MFG, within the framework of closed-loop equilibria. In \cite{PossamaiTangpi}, propagation of chaos for BSDEs arguments are employed. In \cite{Djete}, by making use of the notion of measure-valued MFG equilibria, convergence is studied for a general class of MFG of controls; remarkably, no uniqueness of MF equilibria is assumed here, as in the important works \cite{Lacker, Lacker2}, where semi-Markov equilibria are obtained by means of compactness arguments. 

By developing further the analytic approach proposed here, we plan to extend our convergence arguments in the following direction, for which, to the best of our knowledge, no results are yet available: the \Ds monotone setting appears to be robust enough to guarantee estimates that are not only independent of $N$, but also on $T$, thus allowing to study the infinite-horizon (or ergodic) problem. % Second, we aim at analysing the vanishing viscosity limit, where the idiosyncratic noise $\sigma = \sigma(N) > 0$ vanishes as $N \to \infty$. Note that the estimates obtained here are stable with respect to the common noise strength $\beta$, but $\sigma$ needs to be bounded away from zero uniformly in $N$ in the present work. Furthermore, we suspect that the two concepts of semimonotonicity explored here might be two special cases of a more general framework.

Finally, what is developed here is restricted to Hamiltonians of the form $H^i(x^i, D^i u^i)$. In our opinion, it would be interesting to explore this method in the framework of more general Hamiltonians $H^i(t,\bs x,D_i u^i)$, possibly with general growth in the $\bs x$ variable, or of problems involving major players. Moreover, in the very recent preprint \cite{CJR}, non-asymptotic estimates on the gap between closed-loop and open-loop equilibria are investigated for rather general differential games, borrowing and extending ideas from the present paper.

\subsection{Acknowledgements} The authors are partially supported by the Gruppo Nazionale per l'Analisi Matematica, la Probabilit\`a e le loro Applicazioni (GNAMPA) of the Istituto Nazionale di Alta Matematica (INdAM), Italy, and by the project funded by the EuropeanUnion--\penalty0\hskip0pt\relax NextGenerationEU under the National Recovery and Resilience Plan (NRRP), Mission 4 Component 2 Investment 1.1 -- Call PRIN 2022 No.\ 104 of February 2, 2022 of Italian Ministry of University and Research; Project 2022W58BJ5 (subject area: PE -- Physical Sciences and Engineering) ``PDEs and optimal control methods in mean field games, population dynamics and multi-agent models''. Cirant partially supported by the King Abdullah University of Science and Technology (KAUST) project CRG2021--4674 ``Mean Field Games: models, theory, and computational aspects''. Redaelli partially supported by Fondazione Cassa di Risparmio di Padova e Rovigo, Italy.

The authors would like to thank the anonymous reviewers whose comments and suggestions helped improve and clarify this manuscript.

\section{Recurring general notation}\label{sec_nota}

Before proceeding, we believe it is useful to fix some notation, a part of which has already been explained in the Introduction.

\subsection{Vectors and matrices}
 Given $m \in \N$, the \emph{all-ones vector} is denoted $1_m$.
The (Euclidean) \emph{inner product} is denoted by $\cdot$\,.
If $\bs x \in (\R^n)^N$ and $\bs y \in (\R^m)^N$, we define
\[
(\bs x : \bs y) \defeq ((x^j,y^j))_{1 \leq j \leq N} \in (\R^n \times \R^m)^N.
\]

 The transpose of a matrix $\msf M$ is denoted by $\msf M^\trn$.
 If $\bs \xi \in \R^m$, ${\rm diag}\,\bs \xi$ is the diagonal $m \times m$ matrix having $\xi^i$ as the $(i,i)$ entry.
Given a real $m \times m$ matrix $\msf M$, we write $\msf M \ge 0$ (even if $\msf M$ is not symmetric) whenever $\sum_{1 \le i,j \le m} \msf M_{ij} \xi^i \xi^j \ge 0$ for all $\bs \xi \in \R^m$.
 The space of all (real) \emph{symmetric $m\times m$ matrices} is $\Sym(m)$. The \emph{identity} matrix therein is $\msf I_m$, while the \emph{matrix of ones} is $\msf J_m$.
 
 The operator $\otimes$ can be both the \emph{Kronecker product} (between matrices) and the \emph{outer product} (between vectors). Then, for instance, $\msf J_m = 1_m \otimes 1_m$ and $\msf I_{Nd} = \msf I_N \otimes \msf I_d$.
An \emph{ad hoc} notation is used for the two matrices appearing in the Nash system: we write $\msfb I \defeq \msf I_{Nd}$ and $\msfb J \defeq \msf J_N \otimes \msf I_d$.

\subsection{Vectors with suppressed or replaced coordinates}
 Given a set $\call S$, $\bs x \in \call S^N$ and $i \in \{1,\dots,N\}$, we denote by $\bs x^{-i} \in \call S^{N-1}$ the vector $\bs x^{-i} = (x^1,\dots,x^{i-1},x^{i+1},\dots, x^N)$.
%$(x^1,\dots,\widehat{x^i}, \dots, x^N) \in \call S^{N-1}$, where $\widehat{x^i}$ means that the coordinate $x^i$ is omitted.
 The symbol $(\bs x^{-i},y)$, with $y \in \call S$, will be understood as the vector $\bs z \in \call S^N$ with $z^j = x^j$ if $j \neq i$ and $z^i = y$; that is, the vector $\bs x$ with the $i$-th coordinate replaced by $y$.
The order will be important in the sense that we will specify first the vector $\bs x$ and then the replacement $y$, so we point out that both the following \emph{not equivalent} notations will occur: given a function $h$ defined on $\call S^N$,
\[
h(\bs x^{-i},y) = h(x^1,\dots,x^{i-1},y,x^{i+1},\dots,x^N),
\]
but
\[
h(y,\bs x^{-i}) = h(y,x^1,\dots,x^{i-1},x^{i+1},\dots,x^N).
\]

\subsection{Measures and spaces of measures}
 Given two measurable spaces $\call X$ and $\call Y$, a measure $\mu$ on $\call X$ and a measurable map $\psi \colon \call X \to \call Y$, the \emph{push-forward} (or \emph{image}) \emph{measure} of $\mu$ via $\psi$ is the measure $\nu$ on $\call Y$ given by $\psi\mpush\mu \defeq \mu \circ \psi^{-1}$.
 
 Given separable Banach space $(\call X,|\var|)$ and $p \in [1,\infty)$, we will denote by $\Pc_p(\call X)$ the space of all probability measures $\nu$ on $\call X$ such that $\int_{\call X} |\var|^p \,\di\nu$ is finite, endowed with the $p$-Wasserstein distance $W_p$. Recall that $W_p$ is given by
\[
W_p(\mu,\nu)^p = \inf_{\varpi \in \Pi(\mu,\nu)} \int_{\call X \times \call X} |x-y|^p \,\varpi(\di x,\di y) \quad \forall\, \mu,\nu \in \Pc_p(\call X),
\]
where $\Pi(\mu,\nu)$ is the set of all couplings of $(\mu,\nu)$; that is, of all measures in $\Pc_p(\call X \times \call X)$ with $\mu$ and $\nu$ as their marginals (that is, if $\pi_j \colon \call X \times \call X \to \call X$ is the projection onto the $j$-th copy of $\call X$, $\pi_1\mpush\varpi = \mu$ and $\pi_2\mpush\varpi = \nu$).

 Given $N \in \N$ and $\bs x = (x^1,\dots,x^N) \in \call X^N$, we will use the notation $m_{\bs x} \defeq \frac1N \sum_{1 \leq i \leq N} \delta_{x^j}$; note that $m_{\bs x} \in \Pc_p(\call X)$ for all $p$ as above.

\subsection{Derivatives} Most functions we deal with are defined on $[0,T] \times (\R^d)^N$, with typical element $(t,\bs x)$, where $\bs x = (x^1,\dots,x^N)$, $x^i \in \R^d$ for all $i \in \{1,\dots,N\}$. The coordinates of $x^i$ will be in turn written as $x^{ij}$ with $j \in \{1,\dots,d\}$. We will refer to $t$ as \emph{time} (temporal coordinate) and on $\bs x$, $x^i$, $x^{ij}$ as \emph{space} (spatial coordinates).
 
 The derivative with respect to time is denoted by $\de_t$. 
 The gradient (or derivative) with respect to a spatial coordinate $y$ is $D_{y}$; in particular, we will use the shorter notations $D \defeq D_{\bs x}$ and $D_j \defeq D_{x^j}$. The same applies to higher order derivatives, so that, for instance, $D^2$ is the Hessian with respect to $\bs x$ and $D^2_{ij} = D^2_{x^ix^j} = D_{x^i}D_{x^j}$.
 
Note that, when seeing $D$ as a vector of $N$ vectors of $d$ differential operators, and $D_j$ as a vector of $d$ differential operators, one has that $D_j$ is the $j$-th coordinate of $D$. Analogously, when seeing $D^2$ as an $N \times N$ matrix of $d \times d$ matrices of differential operators, and $D^2_{ij}$ as a $d \times d$ matrix of differential operators, one has $(D^2)_{ij} = D^2_{ij}$; that is, the entry $(i,j)$ of $D^2$ is $D^2_{ij}$.

 The \emph{multi-index} notation is the usual one: for instance, given $\alpha = (\alpha_1,\dots,\alpha_N) \in \N^N$ and $h$ differentiable on $(\R^d)^N$, we have $D^\alpha h = D^{\alpha_1}_1\cdots D^{\alpha_N}_N h$.
 
 For functions depending on measures, the \emph{flat derivative} and the \emph{intrinsic derivative} with respect to a probability measure (according to the definitions given in~\cite[Section~2.2.1]{CDLL}) are $\frac{\delta}{\delta m}$ and $D_m$, respectively.

\subsection{Norms} Unless otherwise specified: 
the norm $|\var|$ is understood to be the \emph{Euclidean norm} when applied to a vector, the \emph{Frobenius norm} for matrices and, more generally, the \emph{full tensor contraction} for tensors (and according to this, inner products between tensors will be understood as tensor contractions) -- for example,
\[
|D^3_{ijk} u^j|^2 = \sum_{1 \leq h_i,h_j,h_k \leq d} \big|D^3_{x^{ih_i}x^{jh_j}x^{kh_k}} u^j \big|^2;
\]
 given $h : (\R^d)^N \to \R$, the \emph{supremum norms} of $h$ and $Dh$ are given by
\[
\|h\|_\infty = \sup_{\bs x \in (\R^d)^N} |h(\bs x)|, \qquad \|D h\|_\infty = \| |Dh| \|_\infty = \bigg\| \bigg( \sum_{j=1}^N |D_j h|^2 \bigg)^{\frac12} \bigg\|_\infty;
\]
 the supremum norm of a function which depends on more variables (for example, also on time) is understood with respect to all of them.

The \emph{$i$-th weighted norm} is defined by
\[
\norm{\bs x}^i \defeq \Bigl( |x^i|^2 + {\frac1{N}} \sum_{\substack{1\leq j\leq N \\ j\neq i}} |x^j|^2 \Bigr)^{\frac12}, \quad \bs x \in (\R^d)^N.
\]

\subsection{Particular symbols} Given a function $h$, we will write $h\bigl|_{\bs y}^{\bs x} \defeq h(\bs x) - h(\bs y)$. The symbols $\Dop$ and $Q$ are defined in \eqref{ENS_defcallD} and below; for $\LLop$, see \eqref{ENS_defcallL}; for $\bs\triangle$, see \eqref{ENS_defDeltah}.

\subsection{Assumptions} The main assumptions used in this paper are the following: \hyperlink{MF}{{\bfseries{(MF)}}} (see page \pageref{MFass}); \hyperlink{DS}{{\bfseries{(DS)}}} and \hyperlink{LS}{{\bfseries{(LS)}}}  (see page \pageref{DSass}); \hyperlink{H0}{{\bfseries{(H0)}}}, \hyperlink{H1}{{\bfseries{(H1)}}}, and \hyperlink{H1'}{{\bfseries{(H1$\bs'$)}}} (see page \pageref{Hass}); \hyperlink{S}{{\bfseries{(S)}}} and \hyperlink{LP}{{\bfseries{(LP)}}} (see page \pageref{LPass}).

\section{Summary of main results} \label{ENS_SMR}

\subsection{Two notions of (semi)monotonicity}

Along with the hypothesis of \emph{Mean-Field-like} interactions (which will be presented later in \Cref{ENS_sec:EOD}), our main structural assumption on the system will be of \emph{semimonotonicity} of the data, according to either one of the following definitions.

\begin{defn} Let $h \colon (\R^d)^N \to \R^N$, and assume that $D_i h^i$ exists for all $i \in \{1,\dots,N\}$. For $M \geq 0$, we say that $h$ is \emph{$M$-\Ds semimonotone} if
\begin{equation} \label{ENS_Dm}
\sum_{1\leq i \leq N} \,(D_i h^i(\bs x) - D_i h^i(\bs y))\cdot(x^i-y^i) \geq - M |\bs x - \bs y|^2 \quad \forall \, \bs x, \bs y \in (\R^d)^N
\end{equation}
and $M$ is called the \emph{\Ds semimonotonicity constant}. We simply say that $h$ is \textit{\Ds semimonotone} if \eqref{ENS_Dm} holds for some $M\geq0$. If \eqref{ENS_Dm} holds with $M = 0$, we say that $h$ is \textit{\Ds monotone}.
\end{defn}

In order to abridge the notation in what follows we shall write $h|_{\bs y}^{\bs x} = h(\bs x) - h(\bs y)$ and
\begin{equation} \label{ENS_defcallD}
\Dop[h](\bs x, \bs y) \defeq \sum_{1\leq i \leq N} D_i h^i\big|_{\bs y}^{\bs x} \cdot (x^i-y^i),
\end{equation}
so that, inequality~\eqref{ENS_Dm} will appear as $\Dop[h] \geq -M \QA$ on $(\R^d)^{2N} \simeq (\R^d)^N\times (\R^d)^N$, where $\QA = Q_{\msfb A}$ is the quadratic form induced on $(\R^d)^{2N}$ by the matrix $\msfb A \defeq %(2 \msf I_2 - \msf J_2) 
\begin{psmallmatrix}1 & -1\\-1 & 1\end{psmallmatrix} \otimes \msfb I \in \Sym(2Nd)$.

\begin{defn} Let $h \colon (\R^d)^N \to \R^N$. For $\kappa \geq 0$, we say that $h$ is $\kappa$-\textit{\LLs semimonotone} if
\begin{equation} \label{ENS_Lm}
\sum_{1\leq i \leq N} \big( h^i(\bs x) - h^i(\bs x^{-i}, y^i) - h^i(\bs y^{-i}, x^i) + h^i(\bs y) \big) \geq -\kappa | \bs x - \bs y |^2 \quad \forall \, \bs x, \bs y \in (\R^d)^N
\end{equation}
and $\kappa$ is called the \emph{\LLs semimonotonicity constant}. 
We simply say that $h$ is \textit{\LLs semimonotone} if \eqref{ENS_Lm} holds for some $\kappa\geq0$.
If \eqref{ENS_Lm} holds with $\kappa = 0$, we also say that $f$ is \textit{\LLs monotone}.
\end{defn}

For the sake of brevity,  we define
\begin{equation} \label{ENS_defcallL}
\LLop[h](\bs x, \bs y) \defeq \sum_{1\leq i \leq N} \big( h^i(\bs x) - h^i(\bs x^{-i}, y^i) - h^i(\bs y^{-i}, x^i) + h^i(\bs y) \big),
\end{equation}
so that inequality~\eqref{ENS_Lm} will appear as $\LLop[h] \geq -\kappa \QA$ on $(\R^d)^{2N}$.

\begin{rmk}[Pointwise characterisation of semimonotonicity and comparison] \label{ENS_rmkDMeq}\label{ENS_rmkLLeq} \label{ENS_rmkLLtoDM}

If $h$ is of class $C^2$, then exploiting the fundamental theorem of calculus one easily shows that
\begin{equation*} %\label{ENS_D2dm}
\text{$h$ is $M$-\Ds semimonotone (that is, $\Dop[h] \geq -M \QA$ on $(\R^d)^{2N}$)} \quad \iff \quad (D^2_{ij} h^i)_{1\leq i,j \leq N} \geq -M\msfb I.
\end{equation*}

Analogously, applying the fundamental theorem of calculus twice to \eqref{ENS_Lm},
\[
\LLop[h](\bs x, \bs y) = \int_0^1 \int_0^1 \sum_{\substack{1 \leq i,j \leq N \\ j \neq i}} D^2_{ij} h^i(\bs z^{-i}_{s'}, z^i_{s}) (x^j - y^j) \cdot (x^i - y^i) \,\di s\,\di s',
\]
where we have set $\bs z_s \defeq s \bs x + (1-s) \bs y$; thus, one easily shows that
\[
\begin{multlined}[.95\displaywidth] %\label{ENS_D2LL}
\text{$h$ is $\kappa$-\LLs semimonotone (that is, $\LLop[h] \geq -\kappa \QA$ on $(\R^d)^{2N}$)} \\
\iff \quad (D^2_{ij} h^i)_{1 \leq i,j \leq N} - \mathrm{diag}\,(D^2_{ii} h^i)_{1 \leq i \leq N} \geq -\kappa \msfb I.
\end{multlined}
\]

At the level of second-order derivatives, one then observes that the two notions of semimonotonicity differ by the diagonal term
\begin{equation} \label{ENS_defDeltah}
\bs\triangle_h \defeq \mathrm{diag}\,(D_{ii}^2 h^i)_ {1\leq i \leq N}.
\end{equation}
Therefore, under some unilateral control of such a diagonal, they can be somehow compared: 
%Assume now that $f^i$ is twice differentiable for all $i$, and $\sup_{1 \leq i\leq N} \norm{D_{ii}^2 f^i}_\infty \leq L$; hence if $h \in \{f,g\}$ satisfies assumption~{\bfseries(\hyperlink{LS}{LS})}, then by \Cref{ENS_rmkDMeq,ENS_rmkLLeq} it satisfies assumption~{\bfseries(\hyperlink{DS}{DS})} with $M_h = \kappa_h + L_h$, which is again independent of $N$.
%Such a reduction from assumption~{\bfseries(\hyperlink{LS}{LS})} to assumption~{\bfseries(\hyperlink{DS}{DS})} can be performed for any function $h$ 
if $D_{ii}^2 h^i \geq -\gamma \msf I_d$ on $\R^d$, for some constant $\gamma > 0$ and all $i$,  \LLs semimonotonicity implies \Ds semimonotonicity
\[
\begin{cases}
\LLop[h] \geq -\kappa \QA & \text{on} \ (\R^d)^{2N} \\
\bs\triangle_h \geq -\gamma \msfb I & \text{on} \ (\R^d)^N
\end{cases}
\quad \implies \quad
\Dop[h] \geq -(\kappa + \gamma ) \QA \ \ \text{on} \ (\R^d)^{2N}.
\]
On the other hand, 
if $D_{ii}^2 h^i \leq \gamma \msf I_d$ on $(\R^d)^N$,  \Ds semimonotonicity implies \LLs semimonotonicity
\[
\begin{cases}
\Dop[h] \geq -M \QA & \text{on} \ (\R^d)^{2N} \\
\bs\triangle_h \leq \gamma \msfb I & \text{on} \ (\R^d)^N
\end{cases}
\quad \implies \quad
\LLop[h] \geq -(M + \gamma ) \QA \ \ \text{on} \ (\R^d)^{2N}.
\]
This kind of comparison is coherent with the one that is made in the fully symmetric MFG case, see for example \cite{gangbo}.
\end{rmk}

\begin{rmk}
Having a look at the second-order characterisation above, one notes that the notion of \LLs monotonicity is rather strict, in the sense that only functions with a very particular structure can be \LLs monotone. In fact, for $h$ to be \LLs monotone one needs the \emph{zero-diagonal} matrix $(D^2_{ij} h^i)_{1 \leq i,j \leq N} - \triangle_h$ to be positive semidefinite (in the generalised sense that its symmetrisation is positive semidefinite), which happens only if such a matrix is the null one. Hence this forces each $h^i$ to have the form $h^i(x) = h_0^i(x^i) + h_1^i(x^{-i})$, for some $h_0^i \colon \R^d \to \R$ and $h_1^i \colon (\R^d)^{N-1} \to \R$.
\end{rmk}

Nevertheless, for a given $h$, it is clear that the two constants involved in the definitions may differ substantially, hence one should not regard the two notions as comparable at all, especially when $\kappa$ or $M$ are small, that will be likely in our analysis. In the limit case that $f$ be \Ds \textit{monotone}, it will be false in general that $f$ is also \LLs \textit{monotone}, and vice versa.

\begin{rmk}[Relationship with monotonicity in the \textit{Mean Field } setting]\label{rmk_relationship}
Consider the case $f^i(\bs x) = F(x^i, m_{\bs x})$ for all $i \in \{1, \dots, N\}$, where $F : \R^d \times \Pc(\R^d) \to \R$ is %smooth enough for the following computations to be admissible.
such that $D_m F$ is globally $L$-Lipschitz with respect to the ($|\cdot|, W_2, |\cdot|)$ distance.

\noindent $\bullet$ \ Suppose that $F$ is \textit{displacement monotone}; that is,
\begin{equation}\label{FDMr}
\int_{\R^d \times \R^d} \big( D_x F(x, m_1) - D_x F(y, m_2) \big)\cdot(x-y) \,\mu(\di x, \di y) \geq  0
\end{equation}
for every $m_1, m_2 \in  \Pc(\R^d)$ and $\mu \in \Pc(\R^d \times \R^d)$ having $m_1$ and $m_2$ as first and second marginal, respectively. We show that $f$ is ``almost'' \Ds monotone. Indeed, for every $i$,
\[
D_i f^i (\bs x) = D_x F(x^i, m_{\bs x}) + \frac1N \, D_m F(x^i, m_{\bs x}, x^i),
\]
hence
\[
\begin{multlined}[.95\displaywidth]
\sum_{1\leq i \leq N} \,(D_i f^i(\bs x) - D_i f^i(\bs y))\cdot(x^i-y^i) \\
\begin{multlined}[.9\displaywidth]
=
N \int_{(\R^d)^2} (D_x F(x, m_{\bs x}) - D_x F(y, m_{\bs y}) )\cdot(x-y) \, \frac1N\!  \sum_{1\leq i \leq N} \delta_{(x^i, y^i)} (\di x, \di y) \\
+
\frac1N \sum_{1\leq i \leq N} (D_m F(x^i, m_{\bs x}, x^i) - D_m F(y^i, m_{\bs y}, y^i))\cdot(x^i-y^i).
\end{multlined}
\end{multlined}
\]
The former term on the right-hand side is nonnegative by the displacement monotonicity of $F$. The latter one can be estimated in absolute value, assuming that $D_m F$ is globally $L$-Lipschitz with respect to the ($|\cdot|, W_2, |\cdot|)$ distance, by
\[
\frac L N\, W_2(m_{\bs x}, m_{\bs y}) \sum_{1\leq i \leq N}  |x^i-y^i| + \frac{2L}N \sum_{1\leq i \leq N} |x^i-y^i|^2 \leq \frac {3L} N \sum_{1\leq i \leq N}  |x^i-y^i|^2.
\]
Therefore, $f$ is $\frac{3L}{N}$-\Ds semimonotone.
More generally, if the right-hand side of \eqref{FDMr} is not zero, but $-M \int_{\R^d \times \R^d} |x-y|^2 \,\mu(\di x, \di y)$, then $F$ is just displacement semimonotone, as in \cite[Definition 2.7]{gangbo}. In such a case, $f$ becomes $\big(M + \frac{3L}{N} \big)$-\Ds semimonotone.
\smallskip

\noindent $\bullet$ \ Suppose now that $F$ is \textit{monotone} in the Lasry--Lions sense; that is, for all $m_1, m_2 \in  \Pc(\R^d)$,
\begin{equation}\label{FLLmr}
\int_{\R^d} \big( F(z, m_1) - F(z, m_2) \big) (m_1 - m_2)(\di z) \geq 0.
\end{equation}
Then,
\[
\begin{multlined}[.95\displaywidth]
\sum_{1\leq i \leq N} \bigl( f^i(\bs x) - f^i(\bs x^{-i}, y^i) - f^i(\bs y^{-i}, x^i) + f^i(\bs y) \bigr) \\
\begin{multlined}[.9\displaywidth]
 =
N \int_{\R^d }  \big(F(z, m_{\bs x}) - F(z, m_{\bs y})\big) (m_{\bs x} - m_{\bs y})(\di z) \\
+
\sum_{1\leq i \leq N} \biggl( F(y^i, m_{\bs x}) - F\Bigl(y^i, m_{\bs x} + \frac{\delta_{y^i}- \delta_{x^i} }N\Bigr) - F\Bigl(x^i, m_{\bs y} + \frac{\delta_{x^i}- \delta_{y^i} }N\Bigr)+ F(x^i, m_{\bs y}) \biggr).
\end{multlined}
\end{multlined}
\]
While the former term on the right-hand side is nonnegative by the monotonicity of $F$, the latter one can be written as
\[
\begin{multlined}[.95\displaywidth]
\frac1N\sum_{1\leq i \leq N} \int_0^1 \Bigl( \frac{\delta F}{\delta m}(y^i, m_s^i, x^i ) - \frac{\delta F}{\delta m}(y^i, m_s^i, y^i )  - \frac{\delta F}{\delta m}(x^i, m_s^i, x^i )  + \frac{\delta F}{\delta m}(x^i, m_s^i, y^i ) \Bigr)  \di s \\
+ \frac1N\sum_{1\leq i \leq N} \int_0^1 \Bigl( \frac{\delta F}{\delta m}(x^i, m_s^i, x^i ) - \frac{\delta F}{\delta m}(x^i, m_s^i, y^i )   - \frac{\delta F}{\delta m}(x^i, \hat m_s^i, x^i ) + \frac{\delta F}{\delta m}(x^i, \hat m_s^i, y^i ) \Bigr) \di s 
\end{multlined}
\]
where $\frac\delta{\delta m}$ denotes the flat derivative (see, e.g., \cite[Definition~2.2.1]{CDLL}) and we set $m_s^i = m_{\bs x} + \frac{1-s}N (\delta_{y^i}- \delta_{x^i})$ and $\hat m_s^i = m_{\bs y} + \frac{1-s}N (\delta_{x^i}- \delta_{y^i})$. Assuming again that $D_m F$ is globally $L$-Lipschitz with respect to the ($|\cdot|, W_2, |\cdot|)$ distance, using the fundamental theorem of calculus these two sums can be bounded, respectively, by $\frac{L}N |\bs x - \bs y|^2$
and
\[
\begin{split}
&\Biggl\lvert \frac1N \sum_{1\leq i \leq N} \biggl(\, \int_0^1 \int_0^1 D_m F(x^i,\var, s'x^i + (1-s')y^i)\Bigl|^{m_s^i}_{\hat m_s^i} \,\di s' \di s \biggr) \cdot (x^i - y ^i) \Biggr\rvert \\
&\ \leq \frac LN \sum_{1\leq i \leq N} \int_0^1 W_2(m_s^i,\hat m_s^i) \,\di s \, |x^i - y ^i|  
\leq \frac LN\, \Bigl( \frac1N \sum_{1 \leq i \leq N} |x^i-y^i|^2 \Bigr)^{\frac12} \! \sum_{1\leq i \leq N} |x^i - y ^i|  
\leq \frac{L}{N}\, |\bs x - \bs y|^2 \,, 
\end{split}
\]
where the penultimate estimate comes from the fact that $m_{(\bs x:\bs y)} + \frac{1-s}N(\delta_{(y^i,x^i)} - \delta_{(x^i,y^i)})$ is a coupling of $(m^i_s,\hat m^i_s)$.
Therefore, $f$ is $\frac{2L}N$-\LLs semimonotone. More generally, if the right-hand side of \eqref{FLLmr} is not zero, but $-M W_2(m_1, m_2)^2$, then $f$ becomes $\big(M + \frac{2L}{N}\big)$-\LLs semimonotone.
\end{rmk}

\subsection{Estimates on the Nash system} \label{ENS_sec:EOD}

We suppose that all data $f^i$ and $g^i$ are of class $C^2$ and $C^4$, respectively (namely, they are continuous and have continuous derivatives up to the second and fourth order, respectively), with bounded derivatives. The Hamiltonians $H^i$ are of class $C^3$.
%, in the usual sense of, e.g., \cite[Section~3.1]{KryHS}) with bounded derivatives. %This is sufficient in order for \eqref{ENS_NS} to have a unique solution $u = (u^i)_{1\leq i \leq N}$ of class $C^{2,4}$ with bounded derivatives, at least in a short time horizon if the data are not bounded; see \Cref{ENS_thmexbdd,ENS_thmexsl}. 
Solutions to the Nash system are assumed to be classical,  $C^1$ in time and $C^4$ in space, with bounded derivatives. Moreover, we require their second order derivatives in space to be uniformly continuous.
Note that we are not assuming that $f^i$ and $g^i$ are globally bounded here: they may have linear growth.

As prefigured, we work in a regime of \emph{Mean-Field-like interactions}; that is, we assume the following: 

\begin{itemize}[align=left,topsep=7pt]
\item[\hypertarget{MF}{{\bfseries{(MF)}}}]\label{MFass} for $h \in \{f,g\}$, there exists $L_h > 0$, independent of $N$, such that
\begin{gather*}
\sup_{1 \leq  i \leq N} \big(\|D_i h^i\|_\infty + \|D^2_{ii} h^i\|_\infty\big) \leq L_h, \\
\sup_{1 \leq  i \leq N} \Big(\sup_{\substack{1\leq k \leq N \\ k \neq i}} \|D_k h^i\|_\infty +  \sum_{\substack{1\leq k \leq N \\ k \neq i}} \|D^2_{ik} h^k\|^2_\infty\Big) \leq \frac{L_h^2}{N},
\end{gather*}
and in addition
\begin{gather*}
\sup_{1\leq k \leq N}  \sum_{1\leq i\leq N} \norm{D(D_k g^i)}^2_\infty +  \sum_{\substack{1\leq k,i \leq N \\ i\neq k}} \norm{D(D^2_{ki}g^i)}_\infty^2 \leq L^2_g, \\
\sup_{1\leq i \leq N} \sum_{\substack{1\leq k\leq N \\ k \neq i}} \norm{D(D_k g^i)}_\infty^2 \leq \frac{L^2_g}N.
\end{gather*}
\end{itemize}

We stress the fact that these bounds are of Mean Field  type in that they provide the weight with respect to $N$ which one expects in the Mean Field  setting, as described in the remark below. Yet, we are not asking for any symmetry with respect to the variables of $h$; that is, we are essentially dropping the classic indistinguishability (or exchangeability) hypothesis of Mean Field  Games. 

%\begin{itemize}[align=left,topsep=7pt]
%\item[\hypertarget{MF}{{\bfseries{(MF)}}}] there exists $L_h > 0$, independent of $N$, such that $\forall\, \bs x, \bs y \in (\R^d)^N, \quad i \in \{1,\dots,N\}$,
%\begin{enumerate}
%\item $|h^i(\bs x) - h^i(\bs y)| \leq L_h\norm{\bs x - \bs y}^i$.
%\end{enumerate}
%\[
%|D^\alpha h^i(\bs x) - D^\alpha h^i(\bs y)| \leq L_h\norm{\bs x - \bs y}^i \quad \forall\, \bs x, \bs y \in (\R^d)^N, \quad i \in \{1,\dots,N\}, \ |\alpha| \in \{0,1,2\},
%\]
%where the multi-index notation for the derivatives is to be understood as $D^\alpha = D^{\alpha_1}_1\cdots D^{\alpha_N}_N$, the matrix norm we consider is the Frobenius norm and
%\[
%\norm{\bs x}^i \defeq \Big( |x^i|^2 + \ros{\frac1{N}} \sum_{\substack{1\leq j\leq N \\ j\neq i}} |x^j|^2 \Big)^{\frac12}.
%\]
%\end{itemize}

\begin{rmk} \label{ENS_rmkMFMFC}
Assume that $f^i(\bs x) = F^i(x^i, m_{\bs x^{-i}})$; recall that $m_{\bs x^{-i}} = \frac1{N-1} \sum_{j \neq i} \delta_{x^j}$. Suppose that $F^i \colon \R^d \times \Pc_1(\R^d) \to \R$ is Lipschitz continuous with Lipschitz constant $L > 0$. Since for any $v \in \R^d$ and $k \neq i$
\[
\begin{multlined}[.95\displaywidth]
|f^i( x^1, \ldots,  x^{k-1},  x^k + v,  x^{k+1}, \ldots )- f^i(\bs x)| \\
\begin{split}
%&\leq \Big|F^i( x^i, \frac1{N-1} \sum_{j \neq \{i,k\}} \delta_{ x^j} + \frac1{N-1}\delta_{  x^k + v} ) - F^i( x^i, \frac1{N-1} \sum_{j \neq i} \delta_{ x^j})\Big|  \\
&\leq L \, W_1\Big(\frac1{N-1} \sum_{j \neq \{i,k\}} \delta_{ x^j} + \frac1{N-1}\delta_{  x^k + v},  \frac1{N-1} \sum_{j \neq i} \delta_{ x^j} \Big) \leq \frac{L}{N+1}\,|v|,
\end{split}
\end{multlined}
\]
one easily checks that the assumptions on $f^i$ in \hyperlink{MF}{{\bfseries{(MF)}}} hold, independently of $N$, provided that
\begin{equation}\label{FDxFLip}
\text{$F^i$ and $D_x F^i$ are Lipschitz continuous with respect to the $(|\cdot|, W_1)$ distance},
\end{equation}
since in such a case $\|D_k F^i\|_\infty$ and $\|D^2_{ki} F^i\|_\infty$ remain bounded proportionally to $1/N$ uniformly with respect to $k \neq i$ and $N$.

On the other hand, the assumptions on $g^i$ in \hyperlink{MF}{{\bfseries{(MF)}}} are a bit more restrictive. This is because several estimates below exploit the regularisation effect of the diffusion, hence bounds on $D^\alpha u^i$ will depend on $D^{\alpha-1} f^i$ and on $D^{\alpha} g^i$. If $g^i(\bs x) = G^i( x^i, \frac1{N-1} \sum_{j \neq i} \delta_{ x^j})$, where $G^i \in C^1(\R^d \times \Pc_1(\R^d))$, then, whenever $j \neq i$,
\[
D_j g^i(\bs x) = \frac1{N-1} D_m G^i\Big( x^i, \frac1{N-1} \sum_{j \neq i} \delta_{ x^j},  x^j \Big),
\]
see \cite[Remark 6.1.2]{CDLL}.  Therefore, for \hyperlink{MF}{{\bfseries{(MF)}}} to hold one should require \eqref{FDxFLip} with $G$ in lieu of $F$ and that
\begin{equation*}
\text{
$D_m G^i$ and $D_x D_m G^i$ are Lipschitz continuous with respect to the $(|\cdot|, W_1, |\cdot|)$ distance. }
\end{equation*}
\end{rmk}

A priori estimates, useful in studying the limit $N \to \infty$, are derived on the $N$-dimensional Nash system under one of the following two additional structural assumptions on $h \in \{f,g\}$ of \emph{semimonotonicity}: 
\begin{itemize}[align=left, topsep=8pt, itemsep=3pt]
\item[\hypertarget{DS}{{\bfseries{(DS)}}}]\label{DSass} there exists a constant $M_h > 0$, independent of $N$, such that $h$ is $M_h$-\Ds semimonotone;
%\begin{equation} \label{ENS_sa}
%\sum_{1\leq i \leq N} \,(D_i h^i(\bs x) - D_i h^i(\bs y))\cdot(x^i-y^i) \geq - M_h |\bs x - \bs y|^2 \qquad \forall \, \bs x, \bs y \in (\R^d)^N;
%\end{equation}
\item[\hypertarget{LS}{{\bfseries{(LS)}}}] there exists a constant $\kappa_h \geq 0$, independent of $N$, such that $h$ is $\kappa_h$-\LLs semimonotone.
%\begin{equation} \label{ENS_LLm}
%\sum_{1\leq i \leq N} \big( h^i(\bs x) - h^i(\bs x^{-i}, y^i) - h^i(\bs y^{-i}, x^i) + h^i(\bs y) \big) \geq -\kappa_h | \bs x - \bs y |^2 \qquad \forall \, \bs x, \bs y \in (\R^d)^N.
%\end{equation}
%where the notation $(\bs x^{-i}, y^i)$ is used to denote the element $\bs z \in (\R^d)^N$ such that $z^j = x^j$ for all $j \neq i$ and $z^i = y^i$.
\end{itemize}

Finally, the assumptions on the Hamiltonians $H^i$ are as follows:
\begin{itemize}[align=left, topsep=8pt, itemsep=3pt]\label{Hass}
\item[\hypertarget{H0}{{\bfseries{(H0)}}}]\label{H0ass} there exists a constant $C_H$ independent of $N$ such that the second and third derivatives of $H^i$ are bounded by $C_H$, as well as its first derivative with respect to $x$, while
\[
|D_p H^i(x,p)| \leq C_H(1+|p|);
\]
\item[\hypertarget{H1}{{\bfseries{(H1)}}}]\label{H1ass} there exists $\lambda_H>0$ independent of $N$ such that
\[
-D_x H^i\bigr|^{(x,p)}_{(\bar x,\bar p)} \cdot (x - \bar x) + D_p H^i\bigr|^{(x,p)}_{(\bar x, \bar p)} \cdot (p - \bar p) \geq \lambda_H \Bigl| D_p H^i\bigr|^{(x,p)}_{(\bar x, \bar p)} \Bigr|^2 \qquad \forall\, x,p \in \R^d
\]
(in terms of the Lagrangians $L^i=L^i(x,a)$, this means that they are $C^2$ and convex, strongly with respect to $a$, uniformly in $N$);
\item[\hypertarget{H1'}{{\bfseries{(H1$\bs{'}$)}}}]\label{H1'ass} each Hamiltonian $H^i(x,p)$ is strongly convex with respect to $p$, uniformly in $x$ and $N$; that is, $D^2_{pp} H^i \geq \Lambda_H \msf I_d$, for some $\Lambda_H>0$ independent of $N$.
\end{itemize}

\begin{rmk}
We are going to use the two convexity assumptions {\bfseries(\hyperlink{H1}{H1})} and {\bfseries(\hyperlink{H1'}{H1$\bs{'}$})} in the study of point (2) of the road map presented in \Cref{subsgenH}. Note that they are both verified whenever $|D_a L^i|$ is bounded uniformly in $N$; nevertheless, we stated them separately since we are going to use the former in the \Ds semimonotone case (that requires joint convexity in all the variables), and the latter in the \LLs semimotonone case (that requires convexity only in the velocity variable).
\end{rmk} 

Our first main result is the following. It consists in the merging of \Cref{ENS_mainthmDS,ENS_mainthmLL} when the Hamiltonians are purely quadratic; its proof for general $H^i$ is discussed in \Cref{sec_genH}.

\begin{thm}[Estimates on the Nash system with semimonotone data] \label{ENS_mainthm} Let \hyperlink{H0}{{\bfseries{(H0)}}} be in force.
Assume Mean-Field-like and semimonotone interactions; that is, assume \hyperlink{MF}{{\bfseries{(MF)}}}, and either \hyperlink{DS}{{\bfseries{(DS)}}}--\hyperlink{H1}{{\bfseries{(H1)}}} or \hyperlink{LS}{{\bfseries{(LS)}}}--\hyperlink{H1'}{{\bfseries{(H1$\bs'$)}}}. If, given $T > 0$, the semimonotonicity constants are sufficiently small (or vice versa), and $N \in \N$ is sufficiently large, % there exists a unique 
any solution $u$ to the Nash system~\eqref{ENS_NS} on $[0,T] \times (\R^d)^N$ satisfies
\[
\sup_{1 \leq i \leq N} \bigg( \sup_{\substack{1\leq j \leq N \\ j \neq i}} \norm{D_j u^i}_\infty +  \bigg\lVert \sum_{\substack{1\leq j\leq N \\ j \neq i}} \big| D^2_{ij} u^j\big|^2 \bigg\rVert_\infty +  \sum_{\substack{1\leq j \leq N \\ j \neq i}}\norm{D(D_j u^i)}_\infty^2 \bigg) \lesssim \frac1N
\]
and
\[
\sup_{1 \leq i \leq N} \norm{D u^i}_\infty + \sup_{1 \leq i \leq N} \sum_{1\leq j \leq N}\norm{D(D_i u^j)}_\infty^2 + \sum_{\substack{1\leq i,j \leq N \\ j\neq i}} \norm{D(D_{ij} u^j)}_\infty^2 \lesssim 1,
\]
where the implied constants are independent of $N$. In addition, $u$ shares the same type of semimonotonicity of the data.
\end{thm}

To clarify the relationship between $T$ and the semimonotonicity constants, and in particular the ``vice versa'' in the previous statement, note that we are requiring in fact the semimonotonicity constants to be less than some threshold, which depends on $T$ in such a way that it vanishes as $T \to \infty$ and explodes as $T \to 0$ (at suitable rates). In other words, \Cref{ENS_mainthm} holds both for fixed time horizon $T$ and small semimonotonicity constants, and for fixed semimonotonicity constants and small time $T$.

As discussed in the introduction, to prove this result we proceed as follows. First, if the solution $u$ to the Nash system is semimonotone, then the desired estimates on the derivatives hold; then, the semimonotonicity of the data provide semimonotonicity $u$ near the final time $T$ in such a way that, as a consequence of the aforementioned estimates, the semimonotonicity of $u$ ``propagates'' up to time $0$, provided that $N$ is large enough. While this strategy works well in the purely quadratic case, we need an additional technical step when $H^i$ is not purely quadratic.

This is done for \Ds semimonotone data first. Then, making use of the relationship highlighted in \Cref{ENS_rmkLLeq} between the two notions of semimonotonicity, with little additional effort the case of \LLs semimonotone data is covered as well.

\subsection{Nonsymmetric Mean Field  Games} \label{ENS_sumNMFGs} \label{ENS_secADS}

In the second part of the work, we show that the estimates contained in  \Cref{ENS_mainthm} allow to attack the convergence problem in the $N \to \infty$ limit. To be sure that, for any fixed $N$, solutions to the Nash system exist, we assume that $f^i$ and $g^i$ belong to $C^{2+\alpha}((\R^d)^N)$ and $C^{4+\alpha}((\R^d)^N)$ respectively (in the usual sense of, e.g., \cite[Section~3.1]{KryHS}), for every $i$ and $N$. By \Cref{ENS_thmexbdd}, solutions $u^i$ indeed exists and are unique, and belong to $C^{2+\alpha/2, 4+\alpha}([0,T] \times (\R^d)^N)$. We have then enough regularity to apply the a priori estimates obtained in first part.

We work, for the sake of simplicity, without common noise; that is, we set $\beta = 0$ in \eqref{ENS_NcSDEs} and thus in \eqref{ENS_NS}. Then we set $\sigma = 1$, and we address for simplicity only the case of quadratic Hamiltonians $H^i(x^i, p) = \frac12|p|^2$.

Besides these reductions, we let $\Lambda \defeq [0,1]$ and we suppose that, for $h \in \{f,g\}$, the following symmetry assumption is fulfilled: 
\begin{itemize}[align=left,topsep=7pt]
\item[\hypertarget{S}{{\bfseries{(S)}}}] \label{Sass}
in the $N$-dimensional Nash system, $h^i(\bs x) = h_N(\lambda_N^i, x^i, \bs x^{-i})$, for some \emph{bounded} sequence of maps $h_N \colon \Lambda \times \R^d \times (\R^d)^{N-1} \to \R$, $N \in \N$, that are \emph{symmetric} on $(\R^d)^{N-1}$, and for some $\lambda_N^i \in \Lambda$;
\end{itemize}
more explicitly, the assumption of symmetry means that for any $\varsigma \in S_{N-1}$ (that is, $\varsigma$ which is a permutation on $\{1,\dots,N-1\}$) %\footnote{Note that we are using now $\sigma$ to denote both the diffusion coefficient in the Nash system and a permutation. Nevertheless, we believe this should not cause any confusion, also because, as we have said, in this part of the work the former is set to be $1$ (and thus it will not even explicitly appear).}
one has
\[
h_N(\lambda, y, \bs z^{\varsigma}) = h_N(\lambda, y, \bs z), \qquad \bs z^{\varsigma} \defeq (z^{\varsigma(1)},\dots,z^{\varsigma(N-1)}),
\]
for all $\lambda \in \Lambda$, $y \in \R^d$ and $\bs z \in (\R^d)^{N-1}$.

\begin{rmk} \label{whenlconst}
Letting $\lambda_N^i = \lambda \in \Lambda$ for all $N,i$, one obtains standard symmetric MFG interactions, so that $u^i(\bs x) = u^1(x^i,\bs x^{-i})$.
\end{rmk}

\begin{rmk} \label{ENS_rmkfdp1}
A bounded and continuous function $h_N \colon (\R^d)^N \to \R$ that is symmetric on its domain can be seen as a ``finite-dimensional projection'' of the function $\frk h_N \colon \Pc_p(\R^d) \to \R$. Indeed, suppose that $h_N$ has a modulus of continuity $\omega$ with respect to the Euclidean distance (that is, $|h_N(x) - h_N(y)| \leq \omega(|x-y|)$), and define
\[
\frk h_N(m) \defeq \inf_{\bs x \in (\R^d)^N} \{ h_N(\bs x) + \tilde \omega(W_p(m_{\bs x},m)) \}\,,
\]
with $\tilde \omega \defeq \omega(N\var)$. It is known (cf.~\cite[Theorem~2.1]{COTFNT}) that there exists $\varsigma \in S_N$ such that
\[
W_p(m_{\bs x}, m_{\bs y}) = N^{-\frac1p} \bigg( \sum_{1 \leq i \leq N} \big|x^i - y^{\varsigma(i)}\big|^p \bigg)^{\frac1p} = N^{-\frac1p} |\bs x - \bs y^{\varsigma}|_p \quad \forall\, \bs x, \bs y \in (\R^d)^N,
\]
hence $|h_N(x)-h_N(y)| \leq \tilde\omega(W_p(m_x,m_y))$, by the symmetry of $h_N$ and H\"{o}lder's inequality; then one readily checks that $h_N(\bs x) = \frk h_N(m_{\bs x})$.
\end{rmk}

With this additional structure on the data, we require \emph{Lipschitz continuous} dependence on the \emph{parameter} $\lambda$: for $h \in \{f,g\}$, $h = h(\lambda,y,\bs z)$ as above,
\begin{itemize}[align=left,topsep=7pt]
\item[\hypertarget{LP}{{\bfseries{(LP)}}}]\label{LPass} there exist a $L_\Lambda > 0$, independent of $N$, such that
\[
\norm{D_y^k h_N(\lambda,\var) - D_y^k h_N(\lambda',\var)}_\infty \leq L_\Lambda |\lambda-\lambda'|
\qquad \forall\, \lambda,\lambda' \in \Lambda,
\]
for $k=0$ if $h = f$ and $k \in \{0,1\}$ if $h=g$.
\end{itemize}

In order to specify the choice of $\bs\lambda \in \Lambda^N$ with which a Nash system is built, we use the notation $u_{\bs\lambda}$ to denote the solution to~\eqref{ENS_NS} on $[0,T] \times (\R^d)^N$ where $h^i = h_N(\lambda_i, \var)$ for $h \in \{f,g\}$, $i \in \{1,\dots,N\}$. Then, we focus on the map
\[
\frk u_N \colon (\bs\lambda,t,\bs x) \mapsto u_{\bs\lambda}^1(t,\bs x),
\]
which is representative of any solution of an $N$-dimensional Nash system of the form we are considering, as the following remark shows.

\begin{rmk}
A solution $u_{\bs\lambda}$ is recovered by noticing that $u^i_{\bs\lambda}(t,\bs x) = \frk u_N(\bs\lambda^{\varsigma},t,\bs x^{\varsigma})$ for any $\varsigma \in S_N$ such that $\varsigma(1) = i$,
as by construction $\frk u_N$ is symmetric on $(\Lambda \times (\R^d))^{N-1}$, in the sense that it is invariant under permutations of the pairs $(\lambda^j, x^j)$, $j > 1$. Similarly, the optimal controls $D_i u^i$ are recovered from $D_1 \frk u$. Heuristically, this means that switching two players does not affect the value functions, but the two corresponding labels need to be switched correspondingly in the Nash system.
\end{rmk}

As the first coordinate of $\bs x \in (\R^d)^N$ and of $\bs\lambda \in \Lambda^N$ have a distinct role for $\frk u$ with respect to the other coordinates, we use the notations $\bs x = (x,\hbs x) \in \R^d \times (\R^d)^{N-1}$ and $\bs \lambda = (\lambda, \hbs\lambda) \in \Lambda \times \Lambda^{N-1}$. Also, we write $D_x = D_1$ for the derivative with respect to $x \in \R^d$ and $D_j$ for the one with respect to $\hat x^j \in \R^d$.

Our main contribution contained in this part consists in identifying a limit function and showing that it satisfies a certain representation formula.  This in turn implies that a solution to a generalised Mean Field system can be constructed. These results are contained in \Cref{ENS_rmkconvdiuil,ENS_convfinita,ENS_MFGthm} and we summarise them here, omitting the more technical integral characterisation provided by the second theorem.

\begin{thm} \label{ENS_thmMFsumm} Assume that $H^i = \frac12|p|^2$, $\beta = 0$, $\sigma = 1$.
Let assumptions \hyperlink{MF}{{\bfseries{(MF)}}}, \hyperlink{S}{{\bfseries{(S)}}} and \hyperlink{LP}{{\bfseries{(LP)}}} be in force. Assume also that one between \hyperlink{DS}{{\bfseries{(DS)}}} and \hyperlink{LS}{{\bfseries{(LS)}}} holds, with the corresponding semimonotonicity constants being such that, given $T>0$, the thesis of \Cref{ENS_mainthm} holds.

 Then, there exists a Lipschitz continuous map $U \colon \Lambda \times [0,T] \times \R^d \times \call P_2(\Lambda \times \R^d) \to \R$ with bounded derivative $D_x U$ which is Lipschitz continuous on $\Lambda \times \R^d \times \Pc_2(\Lambda \times \R^d)$ and $\frac13$-H\"{o}lder continuous on $[0,T]$, such that, up to a subsequence,
\[
\sup | D_x^k \frk u_N({\bs\lambda},t,\bs x) - D_x^k U(\lambda,t,x,m_{(\hbs\lambda:\hbs x)}) | \,\xrightarrow[N \to \infty]{}\, 0, \quad k \in \{0,1\},
\]
whenever the supremum is taken over a set of the form
\[
\bigl\{ (\bs\lambda,t, \bs x) \in \Lambda^N \times [0,T] \times (\R^d)^N :\ |x| \leq R,\ m_{\hbs x} \in \call K \bigr\}
\]
for some $R > 0$ and some compact set $\call K \subset \Pc_2(\R^d)$.

Furthermore, given $\bs\mu \in \Pc_2(\Lambda \times \R^d)$ with continuous disintegration with respect to the projection onto $\Lambda$, the function $u^\lambda(t,x) \defeq U^\lambda(t,x,\mu_t)$ solves the following (generalised) Mean Field system on $(0,T) \times \Lambda \times \R^d$:
\[
\begin{dcases}
- \de_t u^\lambda - \Delta_x u^\lambda + \frac12 |D_xu^\lambda|^2 = f^\lambda(x,\pi_{\R^d}\mpush\mu_t) \\
\de_t \mu - \Delta_x \mu - \mathrm{div}_x (Du^\lambda \mu) = 0 \\
u^\lambda(T,\cdot) = g^\lambda(\cdot,\pi_{\R^d}\mpush\mu_T), \quad \mu_0 = \bs\mu,
\end{dcases}
\]
where both the Hamilton--Jacobi and the Fokker--Planck equations are satisfied in the classical sense.
\end{thm}

\begin{rmk}
As discussed in more detail in \Cref{ENS_convhN}, the functions $f^\lambda$ and $g^\lambda$ that appear in the above statement are extracted as subsequential limit of the $f_N$ and $g_N$, respectively. Indeed, our standing assumptions guarantee the convergence of $h_N$ to functions defined on probability measures, in the sense that there exists a map $h \in \Lip( \Lambda \times \R^d \times \Pc_1(\R^d) )$ such that, up to a subsequence,
\begin{equation*} %\label{ENS_convh}
\sup_{\lambda \in \Lambda,\, |x| \leq R,\, m_{\hbs x} \in \call K} |  D_x^k h_N(\lambda,\bs x) -  D_x^k h(\lambda,x,m_{\hbs x}) | \,\xrightarrow[N \to \infty]{}\, 0, % \quad k \in \{0,1\},
\end{equation*}
for any $R > 0$ and $\call K \subset \Pc_1(\R^d)$ compact, and for $k$ as in \hyperlink{LP}{{\bfseries{(LP)}}}.
\end{rmk}

\begin{ex} \label{ENS_exmonL}
Let $F, G: \Lambda \times \R^d \times \call P_2(\R^d) \to \R$ be bounded. Consider $f_N(\lambda,y,\bs z) = F(\lambda,y,m_{\bs z})$ and $g_N(\lambda,y,\bs z) = G(\lambda,y,m_{\bs z})$, so that
\[
f^i(\bs x) = F(\lambda_N^i, x^i ,m_{\bs x^{-i}}), \qquad g^i(\bs x) = G(\lambda_N^i, x^i ,m_{\bs x^{-i}});
\]
that is, players' costs are taken from a pool of admissible symmetric costs, according to some label $\lambda_N^i$.

While \hyperlink{S}{{\bfseries{(S)}}} clearly holds and so does \hyperlink{MF}{{\bfseries{(MF)}}} provided that the properties highlighted in \Cref{ENS_rmkMFMFC} are satisfied for any fixed $\lambda$ (uniformly), \hyperlink{LP}{{\bfseries{(LP)}}} requires Lipschitz continuity with respect to $\lambda$ uniformly in the other variables. Summing up, $F$ needs to satisfy
\[
\text{$F$ and $D_x F$ are Lipschitz continuous with respect to the $(|\cdot|, |\cdot|, W_1)$ distance}
\]
and so does $G$, for which we also require that $D_m G(\lambda, \cdot)$ and $D_x D_m G(\lambda, \cdot)$ are Lipschitz continuous with respect to the $(|\cdot|, W_1, |\cdot|)$ distance, uniformly in $\lambda$.

Finally, one can check the semimonotonicity assumption \hyperlink{DS}{{\bfseries{(DS)}}} or \hyperlink{LS}{{\bfseries{(LS)}}} by revising \Cref{rmk_relationship}, as follows. For any fixed $T > 0$, Theorem \ref{ENS_mainthm} gives some smallness condition on the semimonotonicity constant $M_0$.

%Now, for any $T > 0$, Theorem \ref{ENS_mainthm} gives
%and  Though one can simply ask the properties highlighted in \Cref{ENS_rmkMFMFC} to hold for any fixed $\lambda$ (uniformly) in order to have \hyperlink{MF}{\bfseries(MF)} satisfied, the (semi)monotonicity assumptions in \Cref{rmk_relationship} need to be revised, for instance, as follow.

\noindent $\bullet$ \ Suppose that, for any $\nu \in \Pc(\Lambda \times (\R^d)^2)$,
\[
\int_\Lambda \int_{\R^d \times \R^d} \big( D_x F(\lambda,x, \pi^{(1)}_{\R^d}\mpush\nu) - D_x F(\lambda,y, \pi^{(2)}_{\R^d}\mpush\nu) + M(x-y) \big)\cdot(x-y) \,\nu(\di\lambda,\di x, \di y) \geq  0,
\]
where $\pi^{(i)}_{\R^d} \colon \Lambda \times (\R^d)^2 \to \R^d$ is the projection onto the $i$-th copy of $\R^d$ ($i = 1,2$).
If $D_m F(\lambda,\cdot)$ is $L$-Lipschitz with respect to the $(|\var|,W_2,|\var|)$ distance, uniformly in $\lambda$, then $f_N$ is $\bigl(M+O(N^{-1})\bigr)$-\Ds semimonotone. Therefore, if $M < M_0$, Theorem \ref{ENS_thmMFsumm} applies (by requiring also the same inequality on $G$ of course).

\smallskip
\noindent $\bullet$ \ Suppose that, for any $\mu_1, \mu_2 \in \Pc_2(\Lambda \times \R^d)$,
\[
\int_{\Lambda \times \R^d} \bigl( F(\lambda,z,\pi_{\R^d}\mpush\mu_1) - F(\lambda,z,\pi_{\R^d}\mpush\mu_2) \bigr) (\mu_1-\mu_2)(\di\lambda,\di z) \geq - MW_2(\pi_{\R^d}\mpush\mu_1,\pi_{\R^d}\mpush\mu_2)^2.
\]
Then, arguing as in \Cref{rmk_relationship}, $f_N$ is $(M+O(N^{-1}))$-\LLs semimonotone. Hence, for $D_m F$ as above and $M < M_0$, Theorem \ref{ENS_thmMFsumm} applies. 
\end{ex}

\begin{rmk}[The fully symmetric case]\label{rmkonmfg} We elaborate on the consequences of our results in the classical MFG setting, where one has
\[
f^i(\bs x) = F(x^i ,m_{\bs x^{-i}}), \qquad g^i(\bs x) = G(x^i ,m_{\bs x^{-i}}),
\]
that is of course a special case of the previous example. Arguing as before as %\Cref{rmk_relationship} and 
in \Cref{ENS_rmkMFMFC}, we first note that \hyperlink{MF}{{\bfseries{(MF)}}} holds provided that
\[
\text{$F$ and $D_x F$ are Lipschitz continuous with respect to the $(|\cdot|, W_1)$ distance}
\]
and so does $G$, for which we also require that $D_m G$ and $D_x D_m G$ are Lipschitz continuous with respect to the $(|\cdot|, W_1, |\cdot|)$ distance. For any $T>0$, there exists $M_0 > 0$ such that if either
\begin{equation}\label{rmkmon1}
\sum_{1\leq i \leq N} \,(D_x F(x^i ,m_{\bs x^{-i}}) - D_x F(x^i ,m_{\bs y^{-i}}))\cdot(x^i-y^i) \geq - M |\bs x - \bs y|^2
\end{equation}
or
\begin{equation}\label{rmkmon2}
\sum_{1\leq i \leq N} \big( F(x^i ,m_{\bs x^{-i}}) - F(y^i ,m_{\bs x^{-i}}) - F(x^i ,m_{\bs y^{-i}}) + F(y^i ,m_{\bs y^{-i}}) \big) \geq -M | \bs x - \bs y |^2
\end{equation}
for all $\bs x, \bs y \in (\R^d)^N$ and some $0 \le M < M_0$, and likewise for $G$,  then the convergence result stated in \Cref{ENS_thmMFsumm} holds (note that $M_0$ vanishes as $T$ goes to infinity). 

We would like to compare the regularity assumptions on $F$ with the ones that are employed in \cite{CDLL}, where smooth solutions of the Master Equation are obtained. In terms of the (crucial) dependence with respect to the $m$ variable, in \cite{CDLL} one has that $F$ (and similarly $G$) is such that $D_m F$ is Lipschitz continuous with respect to the $W_1$ distance, while here we just require $F$ to be Lipschitz continuous with respect to $W_1$. In other words, our conditions on $F$ (and hence on $f^i$) are not strong enough to expect the Master Equation to have a smooth solution. Note also that \cite{CDLL} considers the flat torus as the state space, which is compact, while here we work on the whole Euclidean space.

In subsequent works \cite{MZ24, bertuccio}, different notions of weak solutions of the Master Equation have been proposed, all of them demanding some Lipschitz continuity of $F, G$ (but not of their derivatives in $m$). It was in fact observed that Lipschitz regularity of solutions (and of their derivatives in $x$) is strong enough to deduce large population limit results. Since the approach presented in this paper also builds upon the Lipschitz regularity of $F$, we would expect, as limits of $u^i$ (in the symmetric case), Lipschitz solutions of the Master Equation in the sense of the aforementioned papers.

Note finally that, to verify \eqref{rmkmon1} or \eqref{rmkmon2} from standard displacement and Lasry--Lions monotonicity respectively, one may need to employ Lipschitz regularity of $D_m F$ with respect to the $W_2$ distance (as in \Cref{rmk_relationship}). In any case, since $W_2$-Lipschitz regularity is weaker than $W_1$-Lipschitz regularity, we cannot have access to those information on derivatives of $u^i$ that would yield $W_1$-Lipschitz regularity of the limit derivative $D_m U$. This is, for example, required to obtain the convergence rate of order $1/N$  from the finite agent problem to the Mean Field limit, as in \cite{CDLL}. We believe that, under stronger assumptions on $f^i$ and $g^i$ (involving ``pointwise'' estimates on higher order derivatives), those could be shown to propagate to $u^i$ by further refining the estimates developed in the next section.
\end{rmk}

\begin{rmk}
The above result says, within the setting of the previous example, that the Mean Field system arises from a sequence of $N$-player games built by choosing certain $\bs\lambda^N$, provided that the $\lambda^N_i$ are ``sufficiently well distributed'', in the sense that their empirical measures weakly converge to some $\rho \in \Pc(\Lambda)$ as $N\to \infty$. Limit properties of those games are then captured by the Mean Field system when one chooses $\bs\mu = \int_\Lambda m^\lambda \rho(\di\lambda)$, being $m^\lambda$ the distribution of $\lambda$-labelled players.

On the other hand, by a theorem of Hedrl\'{i}n (see \cite{Hed61}), for each Mean Field system of that kind there exists a suitable choice of $\bs\lambda^N$ in order for it to arise from the respective sequence of $N$-player games.
\end{rmk}

\subsection{Joint vanishing viscosity and large population limit.} For the last application of our estimate, we assume that $H^i(x^i, p) = \frac12|p|^2$, there is no common noise $\beta = 0$, and that the idiosyncratic noise depends on $N$, that is $\sigma = \sigma_N$. The purpose is to address the convergence of $u^{i}$ as $N \to \infty$ in the \Ds semimonotone MFG setting, where
\[
f^i(\bs x) = F(x^i ,m_{\bs x^{-i}}), \qquad g^i(\bs x) = G(x^i ,m_{\bs x^{-i}}),
\]
in the situation where $\sigma_N \to 0$. The standing assumptions on $F$ and $G$ are those of the previous \Cref{rmkonmfg}, and our result reads as follows.

\begin{thm}\label{ENS_rmkconvdiuilVV} Assume that $F$ and $G$ satisfy the regularity assumptions of Remark \ref{rmkonmfg}, the \Ds semimonotonicity assumption \eqref{rmkmon1} therein (on both $F$ and $G$), and that
\[
\lim_{N \to \infty} \sigma_N = 0, \qquad \theta \defeq \liminf_{N \to \infty} \sigma_N \log\log N \in (0,\infty].
\]
% Under the hypotheses of \Cref{prop_stabestwrts}, let \eqref{sNthreshold} hold with $\theta_*$ as therein. 
There exist $\theta_* >0$ (depending on $T, F, G$) such that, if $\theta \ge \theta_*$ and the semimonotonicity constants are small enough, then
there exists a map $U_0 \colon [0,T] \times \R^d \times \call P_2(\R^d) \to \R$ with bounded derivative $D_x U_0$ such that $U_0$ and $\partial_x U_0$ are Lipschitz continuous on $\R^d \times \Pc_2(\R^d)$ and $\frac13$-H\"{o}lder continuous on $[0,T]$, and, up to a subsequence,
\begin{equation} \label{ENS_convD1VV}
\sup\, \bigl| D_1^k u^1(t,x^1,\bs x^{-1}) - D_x^k U_0(t,x^1,m_{\bs x^{-1}}) \bigr| \,\xrightarrow[N \to \infty]{}\, 0, \qquad k \in \{0,1\},
\end{equation}
whenever the supremum is taken over any set of the form
\[
\bigl\{ (t, \bs x) \in [0,T] \times (\R^d)^N :\ |x^1| \leq R,\ m_{\bs x^{-1}} \in \call K \bigr\}
\]
for some $R > 0$ and some compact set $\call K \subset \Pc_2(\R^d)$.
Moreover, for all $\tau \in [0,T]$ and $\frk m,\bar{\frk m} \in \Pc_2(\R^d)$, one has
\begin{equation} \label{ENS_reprUlVV}
\int_{\R^d} U_0(\tau,\var,\frk m) \,\di \bar{ \frk m } = \int_\tau^T \int_{\R^d} \Big( \,\frac12\big|D_x U_0(s,\var, m_s)\big|^2 + F(\var, m_s) \Big) \di \bar m_s\di s + \int_{\R^d} G(\var,  m_T) \,\di \bar m_T,
\end{equation}
where $m,\bar m$ solve, respectively,
\[
\begin{cases}
\de_t m - \mathrm{div}(D_x U_0(t,x,m)m) = 0 \\
m|_{t=\tau} = \frk m,
\end{cases}
\qquad
\begin{cases}
\de_t \bar m  - \mathrm{div}(D_x U_0(t,x,m)\bar m) = 0  \\
\bar m|_{t=\tau} = \bar{\frk m}
\end{cases}
\]
for $t \in [\tau,T]$ and $x \in \R^d$.
\end{thm}

Note that a further characterisation via a deterministic MFG system, as in the last part of the statement of \Cref{ENS_thmMFsumm} is certainly possible, following the lines of \Cref{ENS_MFGthm}. We leave out the details of this step for the sake of brevity.

\begin{rmk}
We do not expect the threshold identified by $(\log\log N)^{-1}$ to be optimal to perform a joint vanishing viscosity and large population limit, as our a priori estimates have not been developed with the primary purpose of addressing such a problem (even though we have shown that they are in fact versatile enough to do so). We believe that there is room to improve such a threshold on the admissible vanishing viscosity rate, and to address also the Lasry--Lions monotone case.
\end{rmk}

% \part{Estimates on the Nash system}

\section{Estimates under the one-sided Lipschitz condition on the drift}\label{sec_esti}

Throughout this section $H^i(x^i, p) = \frac12|p|^2$. For the following computations, we assume that $u$ is of class $C^1$ in time and of class $C^4$ in space. Throughout the section, we will always assume a one-sided Lipschitz condition on the first-order (drift) term, as in \eqref{ENS_b} below. Such a term will be often the one appearing in \eqref{ENS_NS}, that is $b^i = D_i u^i$, in which case the one-sided condition is equivalent to the \Ds semimonotonicity of $u$.

\subsection{A Lipschitz estimate for linear equations}

The first ingredient is a Lipschitz estimate for solutions of linear equations with semimonotone drifts. Here, we implement the method of doubling variables (also known as Ishii--Lions method, or coupling method in the probabilistic community) as in \cite{PoPri}. The crucial features of the estimate are that the Lipschitz constant is universal with respect to $N$, and its dependence on the Lipschitz norm of the final condition, the supremum norm of the right-hand side and $t$ is explicit. We note that the technique is well-known, but we are aware of very few examples of statements that are dimensional independent (e.g. \cite{AC09}).
Having a nondegenerate diffusion is crucial here, while the size of $\beta$ does not play any role. 

\begin{lem} \label{ENS_estv}
Let $v \colon [0,T] \times (\R^d)^N \to \R$ be a bounded classical solution to
\begin{equation} \label{ENS_sysv}
\begin{cases}
-\de_t v- \tr((\sigma \msfb I + \beta \msfb J)D^2 v) + \displaystyle\sum_{1 \leq j \leq N} b^j \cdot D_j v = F \\
 v|_{t=T} = G,
\end{cases}
\end{equation}
where $F$ and $G$ are continuous and bounded, $G$ having bounded derivative, $b \colon [0,T] \times (\R^d)^N \to (\R^d)^N$ is continuous and bounded and for some $\tau \in [0,T]$ and $M \geq 0$, 
\begin{equation} \label{ENS_b}
\sum_{1\leq i \leq N} \,(b^i(t, \bs x) - b^i(t, \bs y))\cdot(x^i-y^i) \geq - M |\bs x - \bs y|^2 \quad \forall t \in (\tau,T],\ \bs x,\bs y \in (\R^d)^N.
\end{equation}
% For $\tau \in [0,T)$, suppose that $B[u] \geq -M\QA$ on $(\tau,T]\times (\R^d)^{2N}$. 
Then
%\begin{equation*} %\label{ENS_dv1lip}
%|v(t,\bs x) - v(t,\bs y)| \leq c(t) |{\bs x - \bs y}| \quad \forall\, t \in [0,T], \ \bs x, \bs y \in (\R^d)^N,
%\end{equation*}
%where
%\begin{equation*} %\label{ENS_dv1C}
%c(t) = \big(4\norm{F}_\infty (T-t)^{\frac12}e^{M(T-t)} + \norm{DG}_\infty \big)e^{M(T-t)}.
%\end{equation*}
\[
\norm{Dv(t,\var)}_\infty \leq \big(4 \sigma^{-\frac12} \norm{F}_\infty (T-t)^{\frac12}e^{M(T-t)} + \norm{DG}_\infty \big)e^{M(T-t)} \quad \forall\, t \in [\tau,T].
\]
%for some positive constant $C_\sigma$ that depends only on $\sigma$ and such that $C_\sigma \asymp \sigma^{-\frac12}$ as $\sigma \to 0^+$.}
\end{lem}

Note that the rate $\norm{F}_\infty \sqrt{(T-t)/\sigma}$ is sharp and coincides with the one of the heat equation $-(\partial_t + \sigma \Delta)v = F$, that can be obtained, for instance, using Duhamel's formula.

\begin{proof}
Let $w \colon [0,T] \times (\R^d)^N \times (\R^d)^N \to \R$ be given by $w(t,\bs x, \bs y) \defeq v(t,\cdot)|_{\bs y}^{\bs x} - \psi(t, |\bs x-\bs y|) - \norm{DG}_\infty e^{M(T-t)} |\bs x - \bs y|$,
where
\[
\psi(t, r) \defeq 2{\sigma^{-\frac12}} \norm{F}_\infty e^{2M(T-t)}\bigl( r \sqrt{T-t}  + \sigma^{\frac12} (T-t)(1-e^{-\frac{r}{\sqrt{\sigma(T-t)}}}) \bigr),
\]
%with $C_\sigma$ defined by the identity
%\begin{equation} \label{ENS_defCs}
% e^{-1} \min\Bigl\{3(\sigma^{\frac12} \wedge 1), \frac{e-1}{2}\, \sigma^{\frac12} \Bigr\} = 4C_\sigma^{-1} - 1.
%\end{equation}
%We want to use formula~\eqref{ENS_itow} where $h = w$ and $\bs Z = (\bs X,\bs Y)$ is a coupling by reflection of two solutions to \eqref{ENS_NcSDEs}; that is, formula~\eqref{ENS_itore} with $L$ given by \eqref{ENS_Lre}.
Computations show that for $\bs x \neq \bs y$
\[ \begin{split}
D^2_{(\bs x, \bs y)} w(t,\bs x, \bs y) &= \mathrm{diag}\big(D^2 v(t,\bs x), -D^2 v(t,\bs y)\big)
- \psi_{rr}(t,|\bs x- \bs y|) \bigg( \begin{matrix}
\msfb P_{\bs x-\bs y} & - \msfb P_{\bs x-\bs y} \\
-\msfb P_{\bs x-\bs y} & \msfb P_{\bs x-\bs y}
\end{matrix} \bigg)
\\
&\quad\ - \frac1{|\bs x - \bs y|} \big( \psi_r(t,|\bs x - \bs y|) + \norm{DG}_\infty e^{M(T-t)} \big) \bigg( \begin{matrix}
\msfb I - \msfb P_{\bs x-\bs y} & -\msfb I + \msfb P_{\bs x-\bs y} \\
-\msfb I + \msfb P_{\bs x-\bs y} & \msfb I - \msfb P_{\bs x-\bs y}
\end{matrix} \bigg),
\end{split}\]
where we used the notation $\psi_r = \frac\de{\de r}\psi$, and
$\msfb P_{\bs e} \defeq |\bs e|^{-2} \bs e \otimes \bs e$ for $\bs e \in (\R^d)^N \setminus \{\bs 0\}$. We now want to show that $w$ is a subsolution of the parabolic operator $-\partial_t - L$ on $(0,T) \times \big( (\R^d)^N \times (\R^d)^N \setminus \{\bs x = \bs y\} \big)$, where
\[
L =  \tr\!\big((\msf J_2 \otimes (\sigma\msfb I + \beta \msfb J))D^2_{(\bs x, \bs y)}\big) - %\frac{4|\bs x - \bs y|^2}{|\bs x - \bs y|^2_\beta}
4 \sigma \tr\!\big( \msfb P_{\bs x-\bs y} D^2_{\bs x\bs y} \big) - b(t,\bs x) \cdot D_{\bs x} - b(t,\bs y) \cdot D_{\bs y}.
\]
%and $|\,\cdot\,|_\beta$ is the norm given by the metric induced by the matrix $\msfb I + \beta \msfb J$.
Note that the operator is indeed (degenerate) parabolic, since its diffusion matrix
\[
\begin{pmatrix}
\sigma \msfb I + \beta \msfb J & \sigma \msfb I + \beta \msfb J - 2 \sigma \msfb P_{\bs x-\bs y}  \\
\sigma \msfb I + \beta \msfb J - 2 \sigma \msfb  P_{\bs x-\bs y} & \sigma \msfb I + \beta \msfb J
\end{pmatrix}
\]
is nonnegative.
We have
\[
\tr\!\big((\msf J_2 \otimes (\sigma \msfb I + \beta \msfb J)) D^2_{(\bs x,\bs y)} w( t, \bs x,  \bs y) \big)  = \tr((\sigma \msfb I + \beta \msfb J)D^2 v( t, \cdot))\big|_{ \bs y}^{ \bs x}
\]
and
\[
-4 \sigma \tr\!\big( \msfb P_{\bs x-\bs y} D^2_{\bs x\bs y}  w( t, \bs x,  \bs y) \big) = - 4 \sigma \psi_{rr}(t,|\bs x- \bs y|)\,.
\]
Also, we compute
\begin{gather*}
D_{\bs x} w(t,\bs x, \bs y) = Dv( t, \bs x) - \frac1{|\bs x - \bs y|} \big( \psi_{r}(t,|\bs x- \bs y|) + \norm{DG}_\infty e^{M(T- t)} \big)(\bs x - \bs y), \\
D_{\bs y} w(t,\bs x, \bs y) = - Dv( t, \bs y) + \frac1{|\bs x - \bs y|} \big( \psi_{r}(t,|\bs x- \bs y|) + \norm{DG}_\infty e^{M(T- t)} \big)(\bs x - \bs y)
\end{gather*}
and, according to \eqref{ENS_NS},
\[ \begin{split}
\de_t w(t,\bs x,\bs y) =\, &- \tr\!\big((\sigma \msfb I + \beta \msfb J)D^2v( t,\cdot)\big|_{\bs y}^{\bs x}\big) + \sum_{1\leq j\leq N } b^j( t,\cdot) D_j v( t,\cdot)\big|_{\bs y}^{\bs x} - F|_{\bs y}^{\bs x} \\
& - \psi_{t}(t,|\bs x- \bs y|) + M\norm{DG}_\infty e^{M(T- t)} |\bs x-\bs y|,
\end{split}\]
where $\psi_t = \frac{\de}{\de t}\psi$.
It follows that
\[ \begin{split}
(\de_t + L)w(t,\bs x, \bs y) &= -4\sigma \psi_{rr}(t,|\bs x- \bs y|) - F|_{\bs y}^{\bs x} - \psi_{t}(t,|\bs x- \bs y|) + M\norm{DG}_\infty e^{M(T- t)} |\bs x-\bs y|\\
&\quad + \frac1{|\bs x - \bs y|} \big( \psi_{r}(t,|\bs x- \bs y|) + \norm{DG}_\infty e^{M(T- t)} \big) \sum_{1\leq j\leq N } b^j( t,\cdot)\big|_{\bs y}^{\bs x} \cdot (  x^j -  y^j );
\end{split}
 \]
hence %by the fact that $|\,\cdot\,|_\beta \geq |\,\cdot\,|$ and 
by the hypothesis \eqref{ENS_b} on $b$, % $\Dop[u](t) \geq -M\QA$,
\[
(\de_t + L)w(t,\bs x, \bs y) \geq -4\sigma \psi_{rr}(t,|\bs x- \bs y|) - M \psi_{r}(t,|\bs x- \bs y|)|\bs x-\bs y| - \psi_{t}(t,|\bs x- \bs y|) - 2 \norm{F}_\infty.
\]
Now, direct computations show that, with $r=|\bs x-\bs y|$, $s = T-t$ and $y=r/ \sqrt{T-t}$,
\[
\begin{multlined}[.95\displaywidth]
-4\sigma \psi_{rr}(t,r) - M \psi_{r}(t,r)r - \psi_{t}(t,r) \\
\begin{split}
&= 2{\sigma^{-\frac12}} \norm{F}_\infty e^{2Ms}\Bigl(\sigma^{\frac12}  + 3\sigma^{\frac12}e^{-y\sigma^{-\frac12}} + 2Ms \sigma^{\frac12}(1-e^{-y\sigma^{-\frac12}}) + \frac y 2 (1+2Ms)(1-e^{-y\sigma^{-\frac12}}) \Bigr) \\
&> 2\norm{F}_\infty.
\end{split}
\end{multlined}
\]
%where the last inequality is obtained by noticing that $Ms (2(\sigma^{\frac12} \wedge 1)+y)(1-e^{-y}) \geq 0$ and
%\[
%3(\sigma^{\frac12} \wedge 1) e^{-y\sigma^{-\frac12}} + \frac y 2 (1-(\sigma^{\frac12} \wedge 1)e^{-y\sigma^{-\frac12}}) > e^{-1} \min\Bigl\{3(\sigma^{\frac12} \wedge 1), \frac{e-1}{2}\, \sigma^{\frac12} \Bigr\} \quad \text{on $(0,+\infty)$}
%\]
%and then using \eqref{ENS_defCs}.
Since $w(t, \cdot)|_{\bs x = \bs y} = 0$ for all $t$, by the maximum principle\footnote{Here one needs to apply the maximum principle for classical subsolutions of parabolic equations on unbounded domains. The proof of this is rather classical, and it is based on the fact that interior maxima of $w$ cannot exist. Since $w$ may not have a maximum on the unbounded set $(\R^d)^N \times (\R^d)^N$, it should be additively perturbed for instance by $-\epsilon (e^{K(T-t)}(1+|\bs x|^2 + |\bs y|^2) + (T-\tau)^{-1})$, with $K$ large enough. Being now $w$ bounded (as well as $b$), then the conclusion follows by letting $\epsilon \to 0$. We do not provide here further details, but mention that identical computations appear in the proof of the next \Cref{ENS_estd1}.}
\[
w \leq \left(\max_{(\R^d)^N \times (\R^d)^N} w(T, \cdot)\right)_{\!+} = 0 \quad \text{on $[\tau, T] \times (\R^d)^N \times (\R^d)^N$.}
\]
%Then formula~\eqref{ENS_itore} gives
%\[
%w(\bar t, \bbs x, \bbs y) \leq \bb E[w(T \wedge \theta, \bs X_{T \wedge \theta}, \bs Y_{T \wedge \theta})] \leq \big(g|_{\bs y}^{\bs x} - \norm{DG}_\infty |\bs x - \bs y| \big)_+ = 0
%\]
Being the choice of $t$, $\bs x$ and $\bs y$ arbitrary, 
the desired conclusion follows by observing that $\psi(t, r) \leq 4 \sigma^{-\frac12} \norm{F}_\infty e^{2M(T-t)} r \sqrt{T-t}$.
\end{proof}

\begin{rmk} The previous estimate can be proven also for viscosity solutions, using standard methods. Unbounded solutions (with controlled growth) can be also considered, but we are not going to address these generalisations here for the sake of simplicity.
\end{rmk}

The following estimates will be all derived under the assumption that $u$ be \Ds semimonotone on $(\R^d)^{2N}$ uniformly with respect to $t \in (\tau,T]$; that is, according to the introduced notation, $\Dop[u] \geq -M\QA$ on $(\tau,T] \times (\R^d)^{2N}$. Most of them will make use of \Cref{ENS_estv} with $b^j = D_j u^j$. 

\subsection{First-order derivatives of the value functions}

The next result is a Lipschitz estimate for the value function, it is obtained by doubling variables. We show that the Hamilton--Jacobi equations preserve some weighted Lipschitz seminorm on $(\R^d)^N$. Note that, contrarily to the previous lemma, possible regularisation effects from the diffusion are not exploited below.

%\[
%|D^\alpha h^i(\bs x) - D^\alpha h^i(\bs y)| \leq L_h\norm{\bs x - \bs y}^i \quad \forall\, \bs x, \bs y \in (\R^d)^N, \quad i \in \{1,\dots,N\}, \ |\alpha| \in \{0,1,2\},
%\]
%where the multi-index notation for the derivatives is to be understood as $D^\alpha = D^{\alpha_1}_1\cdots D^{\alpha_N}_N$, the matrix norm we consider is the Frobenius norm and
Recall that, for any $\bs x \in (\R^d)^N$, we have the \emph{$i$-th weighted norm} 
\[
\norm{\bs x}^i \defeq \Bigl( |x^i|^2 + {\frac1{N}} \sum_{\substack{1\leq j\leq N \\ j\neq i}} |x^j|^2 \Bigr)^{\frac12}.
\]

\begin{lem}[Weighted Lipschitz continuity of the value functions] \label{ENS_estd1}
Let $\tau \in [0,T)$. Suppose that $\Dop[u] \geq -M\QA$
on $(\tau, T] \times (\R^d)^{2N}$ for some $M > 0$ and that, for $h \in \{f,g\}$, there exist $\widetilde L_h > 0$ such that 
\begin{equation}\label{fgtildeass}
|h^i(\bs x) - h^i(\bs y)| \leq \widetilde L_h\norm{\bs x - \bs y}^i \qquad \forall\, \bs x, \bs y \in (\R^d)^N, \ i \in \{1,\dots,N\}.
\end{equation}
Then
\begin{equation} \label{ENS_d1lip}
|u^i(t,\bs x) - u^i(t,\bs y)| \leq c_1 \norm{\bs x - \bs y}^i \qquad \forall\, t \in [\tau, T], \ \bs x, \bs y \in (\R^d)^N,
\end{equation}
where the constant $c_1$ depends only on $T$, $M$, $\widetilde L_g$ and $\widetilde L_f$.
\end{lem}

\begin{proof}
For $\epsilon> 0$, let $w^i$ be the function $w^i(t,\bs x, \bs y) \defeq u^i(t,\cdot)|_{\bs y}^{\bs x} - \psi(t) \norm{\bs x-\bs y}^i -\epsilon\phi(t,\bs x, \bs y) - \frac\epsilon{t-\tau}$, defined for $t \in (\tau,T]$, $\bs x, \bs y \in (\R^d)^N$,
where $\psi(t) \defeq e^{M(T-t)}\widetilde L_g + \frac{e^{M(T-t)}-1}{M}\,\widetilde L_f$ and $\phi(t,\bs x, \bs y) = e^{K(T-t)}(1+|\bs x|^2 + |\bs y|^2)$
for some constant $K\geq0$ to be determined.
Since $u^i$ is bounded, $w^i$ attains its maximum at some point $(\bar t, \bbs x, \bbs y) \in (\tau, T] \times (\R^d)^N \times (\R^d)^N$. Suppose that $\bar t \neq T$ and $\bbs x \neq \bbs y$.
Computations show that for $\bs x \neq \bs y$ (for simplicity, one may follow the computations below with $\epsilon = 0$, being the corresponding terms merely necessary to guarantee the existence of a maximum for $w$, and therefore of perturbative nature)
\[\begin{split}
D^2_{(\bs x, \bs y)} w^i(t,\bs x, \bs y) &= \biggl( \begin{matrix}
D^2 u^i(t,\bs x) - 2 \epsilon e^{K(T-t)} \msfb I & \msfb 0 \\
\msfb 0 & -D^2 u^i(t,\bs y) - 2 \epsilon e^{K(T-t)} \msfb I
\end{matrix} \,\biggr) \\
&\quad\ - \frac{\psi(t)}{\norm{\bs x-\bs y}^i} \bigg( \begin{matrix}
\msfb I^i - \msfb P^i_{\bs x-\bs y} & -\msfb I^i + \msfb P^i_{\bs x-\bs y} \\
-\msfb I^i + \msfb P^i_{\bs x-\bs y} & \msfb I^i - \msfb P^i_{\bs x-\bs y}
\end{matrix} \bigg),
\end{split}
\]
where $\msfb I^i$ is the block diagonal matrix given by
\[
(\msfb I^i)_{jj} = 
\begin{cases}
\msf I_d & \text{if} \ j = i \\
 N^{-1} \msf I_d & \text{if} \ j \neq i
\end{cases}
\]
and we have set, for $\bs e \in (\R^d)^N \setminus \{\bs 0\}$,
\[
\msfb P^i_{\bs e} \defeq \frac{\msfb I^i\bs e}{\norm{\bs e}^i} \otimes \frac{\msfb I^i\bs e}{\norm{\bs e}^i}%, \qquad (\msfb I^i \bs e)^j = \begin{cases}
%e^i & \text{if} \ j = i \\
%N^{-1} e^j & \text{if}\ j \neq i
%\end{cases}
.
\]
Then, as $\msf J_2 \otimes (\sigma \msfb I + \beta \msfb J) \geq 0$ and $\tr(D^2 w^i(\bar t, \bbs x, \bbs y)) \leq 0$, we have
\begin{equation} \label{ENS_D1tr}
\begin{split}
0 &\geq \tr\!\big((\msf J_2 \otimes (\sigma \msfb I + \beta \msfb J))D^2_{(\bs x,\bs y)} w^i(\bar t, \bbs x, \bbs y) \big) \\
&= \tr((\sigma \msfb I + \beta \msfb J)D^2 u^i(\bar t, \var))\big|_{\bbs y}^{\bbs x} - {4}(\sigma+\beta)Nd \epsilon e^{K(T-\bar t)}.
\end{split}
\end{equation}
Also, since $(D_{\bs x}, D_{\bs y}, \de_t) w^i(\bar t, \bbs x, \bbs y) = 0$, we have
\begin{gather} \label{ENS_Dx0}
0 = Du^i(\bar t, \bbs x) -\frac{\psi(\bar t)}{\norm{\bbs x-\bbs y}^i} \, \msfb I^i(\bbs x - \bbs y) - 2\epsilon e^{K(T-\bar t)} \bbs x, \\
\label{ENS_Dy0}
0 = - Du^i(\bar t, \bbs y) + \frac{\psi(\bar t)}{\norm{\bbs x-\bbs y}^i} \, \msfb I^i(\bbs x - \bbs y) - 2 \epsilon e^{K(T-\bar t)} \bbs y
\end{gather}
and, using \eqref{ENS_NS} when computing $\de_t w^i$,
\begin{equation} \label{ENS_dt=01} \begin{split}
0 & = - \tr\!\big((\sigma \msfb I + \beta \msfb J)D^2u^i(\bar t,\var)\big|_{\bbs y}^{\bbs x}\big) + \sum_{\substack{1\leq j\leq N \\ j \neq i}} D_j u^j(\bar t,\var) D_j u^i(\bar t,\var)\big|_{\bbs y}^{\bbs x} \\ 
&\quad\ + \frac12|D_i u^i(\bar t,\var)|^2\big|_{\bbs y}^{\bbs x} - f^i|_{\bbs y}^{\bbs x} - \psi'(\bar t) \norm{\bbs x-\bbs y}^i + \epsilon K e^{K(T-\bar t)}(1 + |\bbs x|^2 + |\bbs y|^2) + \frac{\epsilon}{(\bar t-\tau)^2}.
\end{split}\end{equation}
Exploiting relations~\eqref{ENS_Dx0} and \eqref{ENS_Dy0} one sees that 
\begin{multline}\label{ENS_sumest-}
\sum_{\substack{1\leq j\leq N \\ j \neq i}} D_j u^j(\bar t,\var) D_j u^i(\bar t,\var)\big|_{\bbs y}^{\bbs x}  + \frac12|D_i u^i(\bar t,\var)|^2\big|_{\bbs y}^{\bbs x} \\
= \frac1{N} \, \frac{\psi(\bar t)}{\norm{\bbs x-\bbs y}^i} \sum_{\substack{1\leq j\leq N \\ j \neq i}} D_j u^j(\bar t,\var)\big|_{\bbs y}^{\bbs x} \cdot (\bar x^j - \bar y^j) + 4\epsilon e^{K(T-\bar t)} \frac{\psi(\bar t)}{\norm{\bbs x-\bbs y}^i} (\bar x^i - \bar y^i) \cdot (\bar x^i + \bar y^i) \\
+ 4\epsilon^2 e^{2K(T-\bar t)} (|\bar x^i|^2 - |\bar y^i|^2) + 2\epsilon e^{K(T-\bar t)} \sum_{\substack{1\leq j\leq N \\ j \neq i}} \big( D_j u^j(\bar t, \bbs x)\cdot \bar x^j + D_j u^j(\bar t, \bbs y) \cdot \bar y^i \big),
\end{multline}
where the terms can be estimated as follows: since $\Dop[u](\bar t) \geq -M\QA$ and $\sqrt{N}\, \norm{\bs x - \bs y}^i \geq |\bs x - \bs y|$, using relations~\eqref{ENS_Dx0} and \eqref{ENS_Dy0} along with the Cauchy--Schwarz inequality,
\begin{multline}\label{ENS_s1-}
\frac1{N} \, \frac{\psi(\bar t)}{\norm{\bbs x-\bbs y}^i} \sum_{\substack{1\leq j\leq N \\ j \neq i}} D_j u^j(\bar t,\var)\big|_{\bbs y}^{\bbs x} \cdot (\bar x^j - \bar y^j) \\
\begin{split}
&= \frac1{N} \, \frac{\psi(\bar t)}{\norm{\bbs x-\bbs y}^i} \sum_{1\leq j\leq N} D_j u^j(\bar t,\var)\big|_{\bbs y}^{\bbs x} \cdot (\bar x^j - \bar y^j) - \frac1{N} \, \frac{\psi(\bar t)}{\norm{\bbs x-\bbs y}^i} D_i u^i(\bar t,\var)\big|_{\bbs y}^{\bbs x} \cdot (\bar x^i - \bar y^i) \\
&\geq -M\psi(\bar t) \norm{ \bbs x - \bbs y}^i - \frac1{N} \, \frac{\psi(\bar t)}{\norm{\bbs x-\bbs y}^i}\, 2 \epsilon e^{K(T-\bar t)} (\bar x^i + \bar y^i) \cdot (\bar x^i - \bar y^i) \\
&\geq -M\psi(\bar t) \norm{ \bbs x - \bbs y}^i - \frac{2\epsilon}{\sqrt N} \, \psi(\bar t) \phi(\bar t, \bbs x, \bbs y);
\end{split}
\end{multline}
since $\norm{\bs x - \bs y}^i \leq |\bs x - \bs y|$, by the Cauchy--Schwarz inequality
\begin{equation} \label{ENS_s2-}
\begin{multlined}[t][.9\displaywidth]
4\epsilon e^{K(T-\bar t)} \frac{\psi(\bar t)}{\norm{\bbs x-\bbs y}^i} (\bar x^i - \bar y^i) \cdot (\bar x^i + \bar y^i) + 4\epsilon^2 e^{2K(T-\bar t)} (|\bar x^i|^2 - |\bar y^i|^2) \\
\begin{split}
&\geq - 4\epsilon e^{K(T-\bar t)} \psi(\bar t)(|\bar x^i| + |\bar y^i|) - 4\epsilon^2 e^{2K(T-\bar t)} (|\bbs x|^2 + |\bbs y|^2) \\
&\geq -4\epsilon \Big( \psi(\bar t) + \epsilon e^{K(T-\bar t)}\Big)\phi(\bar t, \bbs x, \bbs y);
\end{split}
\end{multlined}\end{equation}
%only by the definition of $\phi$,
%\begin{equation} \label{ENS_s3-}
%2\epsilon^2 e^{2K(T-\bar t)} (|\bbs x|^2 - |\bbs y|^2) \geq - 2\epsilon e^{K(T-\bar t)} \phi(\bar t, \bbs x, \bbs y);
%\end{equation}
finally, by the Cauchy--Schwarz and Young's inequalities,
\begin{equation}
\begin{multlined}[t][.9\displaywidth] \label{ENS_s4-} 
2\epsilon e^{K(T-\bar t)} \sum_{\substack{1\leq j\leq N \\ j \neq i}} \big( D_j u^j(\bar t, \bbs x)\cdot \bar x^j + D_j u^j(\bar t, \bbs y) \cdot \bar y^i \big) \\[-15pt]
\geq - \epsilon \Big(2N\sup_{1 \leq j \leq N} \norm{D_j u^j}^2_\infty+1\Big) \phi(\bar t, \bbs x, \bbs y).
\end{multlined}
\end{equation}
Plugging now \eqref{ENS_sumest-} into \eqref{ENS_dt=01} and exploiting inequalities~\eqref{ENS_D1tr}, \eqref{ENS_s1-}, \eqref{ENS_s2-} and \eqref{ENS_s4-} as well as \eqref{fgtildeass}, we get
\begin{equation} \label{ENS_ineqd1f} \begin{split}
0 < \big( &M\psi(\bar t) + \widetilde L_f + \psi'(\bar t) \big) \norm{ \bbs x - \bbs y}^i \\
&+ \epsilon \big(4(\sigma+\beta) N d + 4 \psi(\bar t) + 2\epsilon e^{K(T-\bar t)} + 2NC_N + 1 - K \big) \phi(\bar t, \bbs x, \bbs y),
\end{split} \end{equation}
where $C_N \defeq \sup_{1 \leq j \leq N} \norm{D_j u^j}^2_\infty$. At this point, choose $K = 4(\sigma+\beta) N d + 4 \psi(0) + 2NC_N + 3$;
thus for all $\epsilon < e^{-KT}$ one has $4(\sigma+\beta) N d + 4 \psi(\bar t) + 2\epsilon e^{K(T-\bar t)} + 2NC_N + 1 - K < 0$. On the other hand, by the definition of $\psi$ we have $M\psi + \widetilde L_f + \psi' = 0$, thus the above inequality~\eqref{ENS_ineqd1f} cannot hold. We conclude that one must have $\bar t = T$ or $\bbs x = \bbs y$, for any sufficiently small $\epsilon$. Letting $\epsilon \to 0$, we deduce that $u^i(t,\cdot)|_{\bbs y}^{\bbs x} - \psi(t) \norm{\bbs x-\bbs y}^i \leq (g^i|_{\bbs y}^{\bbs x} - \widetilde L_g \norm{\bbs x - \bbs y}^i )_+ = 0$,
where we have used \eqref{fgtildeass} to obtain the last equality. The desired conclusion follows with $c_1 = \psi(0)$.
\end{proof}

\begin{rmk}
The previous result can be obtained for solutions to Hamilton--Jacobi equations where the transport term $\sum_{1 \leq j \leq N,\, j \neq i} D_j u^j \cdot D_j u^i$ is replaced by $\sum_{1 \leq j \leq N} b^j \cdot D_j u^i$, with $b$ satisfying \eqref{ENS_b}.
\end{rmk}

The previous proposition immediately provides an estimate for the \emph{skew} first-order derivatives; that is, for those derivatives of the $i$-th value functions with respect to $x^j$ with $j \neq i$.

\begin{prop}[Estimate on skew first-order derivatives] \label{ENS_uijest}
Let $\tau \in [0,T)$ and assume that {\bfseries(\hyperlink{MF}{MF})} holds. Suppose that $\Dop[u] \geq -M\QA$ on $(\tau, T]\times (\R^d)^{2N}$. Then
\[
\sup_{1\leq i \leq N} \Big\| \sum_{\substack{1\leq j \leq N \\ j \neq i}} |D_j u^i(t,\var)|^2 \Big\|_\infty \leq \frac{c_1^2}{N} \quad \forall\, t \in [\tau, T],
\]
where $c_1$ is the constant given in \Cref{ENS_estd1}, that applies with $\widetilde L_f = \sqrt 2 L_f$ and $\widetilde L_g = \sqrt 2 L_g$, where $L_f$ and $L_g$ are the constants appearing in  {\bfseries(\hyperlink{MF}{MF})}.
\end{prop}

\begin{proof} Note first that by the fundamental theorem of calculus and  {\bfseries(\hyperlink{MF}{MF})} we have 
\[
|f^i(\bs x) - f^i(\bs y)| \leq \Big\|\Big(|D_i f^i|^2 + N \sum_{j\neq i} |D_j f^i|^2 \Big)^{\frac12}\Big \|_{\infty}\norm{\bs x - \bs y}^i \leq \sqrt 2\, L_f \norm{\bs x - \bs y}^i.
\]
Let now $\bs \xi \in (\R^d)^N$ be such that $|\bs \xi| = 1$ and $\bs \xi^i= 0$. Plug $\bs y = \bs x + h \bs \xi$ into \eqref{ENS_d1lip} and let $h \to 0$ to obtain $|D u^i(t, \bs x) \cdot \bs \xi | \leq c_1 / \sqrt{N}$. Taking now the supremum over $\bs \xi$ yields $\bigl( \sum_{j \neq i} |D_j u^i(t,\bs x)|^2 \bigr)^{\frac12} \leq {c_1}/{\sqrt{N}}$.
\end{proof}

\subsection{Second-order derivatives of the value functions}

Although all constants contained in the following statements will depend also on the dimension $d$, we shall omit to specify it.

We now show that certain sums of second-order derivatives over the indices of the players are controlled uniformly with respect to $N$. This is a preliminary step to achieve more precise control on second-order derivatives.

\begin{prop} \label{ENS_d2b}
Let $\tau \in [0,T)$ and assume that {\bfseries(\hyperlink{MF}{MF})} holds. Suppose that $\Dop[u] \geq -M\QA$ on $(\tau, T]\times (\R^d)^{2N}$. Then
\[
\sup_{1\leq k \leq N} \sum_{1\leq i \leq N}\norm{D(D_k u^i)(t,\var)}_\infty^2 = \sup_{1 \leq k \leq N} \sum_{1\leq i \leq N} \bigg\lVert \sum_{1\leq j\leq N} \big|D^2_{jk} u^i(t,\var)\big|^2 \bigg\rVert_\infty \leq C_2 \quad \forall\, t \in [\tau, T],
\]
where the constant $C_2$ depends only on $\sigma$, $T$, $M$, $L_g$ and $L_f$.
\end{prop}

\begin{proof}
Let $k \in \{1,\dots,N\}$, $\ell \in \{1,\dots,d\}$ and $s \in [0,T-\tau]$. Note that $v = v^{k\ell} = D_{x^{k\ell}} u^i$ solves~\eqref{ENS_sysv} on $[T-s,T] \times (\R^d)^N$ with $b^j = D_j u^j$, $F = D_{x^{k\ell}} f^i - \sum_{1\leq j\leq N,\, j \neq i} D_j(D_{x^{k\ell}} u^j) D_j u^i$ and $G = D_{x^{k\ell}} g^i$.
By \Cref{ENS_estv},
\begin{equation} \label{ENS_1estd2}
\begin{multlined}[t][.9\displaywidth]
\norm{D_{x^{k\ell}} Du^i}_\infty \\
\leq \frac{C s^{\frac12}e^{2Ms}}{\sigma^{\frac12}}  \norm{D_{x^{k\ell}} f^i}_\infty + \norm{D(D_{x^{k\ell}}g^i)}_\infty e^{Ms} + \frac{C s^{\frac12}e^{2Ms}} {\sigma^{\frac12}}   \bigg\lVert \sum_{\substack{1\leq j\leq N \\ j \neq i}} D_j(D_{x^{k\ell}} u^j) D_j u^i \bigg\rVert_\infty, 
\end{multlined}\end{equation}
where the $L^\infty$-norm are understood to be computed on $[T-s,T] \times (\R^d)^N$ or $(\R^d)^N$. By \Cref{ENS_uijest} and the Cauchy--Schwarz inequality we have
\[
\bigg\lVert \sum_{\substack{1\leq j\leq N \\ j \neq i}} D_j(D_{x^{k\ell}} u^j) D_j u^i \bigg\rVert_\infty^2 \leq \frac{c_1^2}{N} \,\bigg\lVert \sum_{1\leq j\leq N} \big|D^2_{jk} u^j\big|^2 \bigg\rVert_\infty;
\]
therefore from \eqref{ENS_1estd2} and Young's inequality we get
\begin{equation} \label{ENS_2estd2}
\begin{multlined}[t][.9\displaywidth]
\norm{D(D_k u^i)}_\infty^2 \\
\leq \frac{3dC^2  e^{4Ms} s}{\sigma} \norm{D_k f^i}_\infty^2  + 3d \norm{D(D_kg^i)}_\infty^2 e^{2Ms} + \frac{3dC^2 c^2_1 e^{4Ms} s }{\sigma N} \bigg\lVert \sum_{1\leq j\leq N} \big|D^2_{jk} u^j\big|^2 \bigg\rVert_\infty.
\end{multlined}
\end{equation}
Since by {\bfseries(\hyperlink{MF}{MF})} we have $\sup_{1\leq k\leq N} \sum_{1\leq i\leq N} \norm{D_k f^i}_\infty^2 \leq 2L_f^2$,
summing \eqref{ENS_2estd2} over $i$ one deduces that
\[
\sum_{1\leq i \leq N}\norm{D(D_{k}u^i)}_\infty^2 \leq \frac{6dC^2 e^{4Ms} s L_f^2}{\sigma} + 3de^{2Ms} \sum_{1\leq i \leq N} \norm{D_kDg^i}_\infty^2 + \frac{3dC^2 c^2_1 e^{4Ms} s }{\sigma} \sum_{1\leq i \leq N}\norm{D(D_{k}u^i)}_\infty^2.
\]
If $\frac{3d C^2 c^2_1 e^{4Ms} s}{\sigma} \leq \frac12$
this yields
\begin{equation} \label{ENS_5estd2}
\sum_{1\leq i \leq N}\norm{D(D_{k}u^i)}_\infty^2 \leq \frac{2L_f^2}{c_1^2} + 6 d e^{2Ms}\sum_{1\leq i \leq N} \norm{D(D_kg^i)}_\infty^2 \leq \frac{2L_f^2}{c_1^2} + K\sum_{1\leq i \leq N} \norm{D(D_kg^i)}_\infty^2,
\end{equation}
where $K = \max\{6 d e^{2MT},2\}$. Considering that $g^i = u^i(T,\cdot)$, by iterating estimate \eqref{ENS_5estd2} on the intervals $[T-\ell s,T-(\ell-1)s]$ with $u^i(T-(\ell-1)s,\cdot)$ in lieu of $g^i$ for any positive integer $\ell \leq \ell^* \defeq \lfloor (T-\tau)/s \rfloor $ and then on the interval $[\tau, T-\ell^*s]$, one can prove by induction that
\[
\begin{split}
\sum_{1\leq i \leq N}\norm{D(D_{k}u^i)}_\infty^2 &\leq \frac{2L_f^2}{c_1^2}\sum_{0 \leq \ell \leq \ell^*} K^\ell + K^{\ell^*+1}\sum_{1\leq i \leq N}  \norm{D(D_k g^i)}_\infty^2 
\\ 
&\leq K^{\ell^*+1}\biggl(\frac{2L_f^2}{c_1^2} + \sum_{1\leq i \leq N}  \norm{D(D_k g^i)}_\infty^2\biggr),
\end{split}
\]
where the $L^\infty$-norms on the left-hand side are understood to be computed on $[\tau,T] \times (\R^d)^N$. We can estimate $s\ell^* \leq T$ and choose $s = (6dC^2c_1^2 \sigma^{-1 }e^{4MT})^{-1} \wedge (T-\tau)$, hence the desired conclusion follows by noticing that {\bfseries(\hyperlink{MF}{MF})} includes
$\sup_{1\leq k \leq N} \sum_{1\leq i\leq N} \norm{D(D_k g^i)}_\infty^2 \leq L_g^2$, with
\begin{equation}\label{C2form}
C_2 = e^{c_2(1+\sigma^{-1})},
\end{equation}
where $c_2$ depends only on $T$, $M$, $L_g$, and $L_f$.
\end{proof}

%As noted in the proof, we actually need
%\begin{equation}\label{realfgass1}
%\sup_{1\leq k \leq N} \sum_{1\leq i\leq N} \norm{D_k f^i}_\infty^2 +  \sum_{1\leq i\leq N} \norm{D(D_k g^i)}_\infty^2 \leq C.
%\end{equation}

Exploiting the estimates deduced so far, Bernstein's method yields the first estimate on the skew second-order derivatives. It provides a control when letting one derivative be ``in the direction'' of the player to which the value function corresponds while the other derivative is skew and then summing the squares of the second-order derivatives over the players. In this sense, as the sum touches all value functions, we will refer to it as a \emph{transversal} estimate.

\begin{prop}[Transversal estimate on skew second-order derivatives] \label{ENS_d2N}
Let $\tau \in [0,T)$ and assume that {\bfseries(\hyperlink{MF}{MF})} holds. Suppose that $\Dop[u] \geq -M\QA$ on $(\tau, T]\times (\R^d)^{2N}$. Then
\[
\sup_{1\leq i \leq N} \bigg\lVert \sum_{\substack{1\leq j\leq N \\ j \neq i}} \big| D^2_{ij} u^j(t,\var)\big|^2 \bigg\rVert_\infty \leq \frac{C_3}{N} \quad \forall\, t \in [\tau, T].
\]
where the constant $C_3$ depends only on $\sigma$, $T$, $M$, $L_g$ and $L_f$.
\end{prop}

\begin{proof}
Let $w^i \defeq \frac12 \big( \sum_{1\leq\ell\leq N,\, \ell \neq i} |D^2_{i\ell} u^\ell|^2 + \frac1N\, |D^2_{ii} u^i|^2 \big)$.
Direct computations exploiting \eqref{ENS_NS} show that
\begin{equation} \label{eqforwi}
\begin{split}
&-\de_t w^i - \tr((\sigma\msfb I + \beta \msfb J)D^2 w^i) + \sum_{1\leq j \leq N} D_j u^j \cdot D_j w^i \\
& + \beta \sum_{\substack{1\leq\ell\leq N \\ \ell \neq i}} \bigg| \sum_{1\leq j\leq N} D^3_{ij\ell} u^\ell \bigg|^2 + \frac{\beta}{N} \bigg| \sum_{1\leq j\leq N} D^3_{iij} u^i \bigg|^2 + \sigma \! \sum_{\substack{1\leq j,\ell\leq N \\ \ell \neq i}} \big|D^3_{ij\ell} u^\ell \big|^2 + \frac{\sigma}N \sum_{1\leq j\leq N} \big|D^3_{iij} u^i\big|^2 \\
& + \sum_{\substack{1\leq j,\ell\leq N \\ \ell \neq i}} D^2_{j\ell} u^\ell D^2_{ij} u^j D^2_{i\ell} u^\ell + \sum_{\substack{1\leq j,\ell\leq N \\ j \neq \ell \neq i}} D^2_{ij} u^\ell D^2_{j\ell} u^j D^2_{i\ell} u^\ell + \sum_{\substack{1\leq j,\ell\leq N \\ j \neq \ell \neq i}} D_j u^\ell D^3_{ij\ell} u^j D^2_{i\ell} u^\ell \\
& + \frac1N \sum_{1\leq j\leq N} D^2_{ij} u^i D^2_{ij} u^j D^2_{ii} u^i + \frac1N \sum_{\substack{1\leq j\leq N \\ j \neq i}} D^2_{ij} u^i D^2_{ij} u^j D^2_{ii} u^i + \frac1N \sum_{\substack{1\leq j\leq N \\ j \neq i}} D_j u^i D^3_{iij} u^j D^2_{ii} u^i \\
&\quad\ = \sum_{\substack{1\leq \ell\leq N \\ \ell \neq i}} D^2_{i\ell} f^\ell D^2_{i\ell} u^\ell + \frac1N\, D^2_{ii} f^i D^2_{ii} u^i,
\end{split}
\end{equation}
where according to the fact that the we are considering the Frobenius inner product on $\R^d$, all products between tensors are to be understood in the sense of tensor contraction over the coordinates of repeated derivative indices; that is, for example,
\begin{gather*}
D^2_{j\ell} u^\ell D^2_{ij} u^j D^2_{i\ell} u^\ell = \sum_{1\leq k_1,k_2,k_3 \leq d} D^2_{x^{jk_2}x^{\ell k_3}} u^\ell D^2_{x^{ik_1}x^{jk_2}} u^j D^2_{x^{ik_1}x^{\ell k_3}} u^\ell = \tr(D^2_{\ell j} u^\ell D^2_{ji} u^j D^2_{i\ell} u^\ell), \\
D_j u^\ell D^3_{ij\ell} u^j D^2_{i\ell} u^\ell = \sum_{1\leq k_1,k_2,k_3 \leq d} D_{x^{jk_2}} u^\ell D^3_{x^{ik_1}x^{jk_2}x^{\ell k_3}} u^j D^2_{x^{ik_1}x^{\ell k_3}} u^\ell, \\
\bigg| \sum_{1\leq j\leq N} D^3_{iij} u^i \bigg|^2 = \sum_{1\leq k_1,k_2,k_3 \leq d} \bigg| \sum_{1\leq j\leq N} D^3_{x^{ik_1}x^{ik_2}x^{jk_3}} u^i \bigg|^2.
\end{gather*}
and analogously for all other products.
We now estimate the sums above: 
since $\Dop[u] \geq -M\QA$ is equivalent to $(D^2_{j \ell} h^\ell)_{1\leq j,\ell \leq N} \geq -M\msfb I$,
\begin{equation} \label{onlylowbound}
\begin{split}
\sum_{\substack{1\leq j,\ell\leq N \\ \ell \neq i}} \!\! D^2_{j\ell} u^\ell D^2_{ij} u^j D^2_{i\ell} u^\ell &= \sum_{\substack{1\leq j,\ell\leq N \\ j,\ell \neq i}} \!\! D^2_{j\ell} u^\ell D^2_{ij} u^j D^2_{i\ell} u^\ell + \sum_{\substack{1\leq \ell\leq N \\ \ell \neq i}} D^2_{i\ell} u^\ell D^2_{ii} u^i D^2_{i\ell} u^\ell \\
&\geq -2M\! \sum_{\substack{1\leq \ell\leq N \\ \ell \neq i}} |D^2_{i\ell} u^\ell|^2 \geq -4M w^i;
\end{split}
\end{equation}
by the Cauchy--Schwarz inequality and \Cref{ENS_d2b},
\[
\sum_{\substack{1\leq j,\ell\leq N \\ j \neq \ell \neq i}} \big| D^2_{ij} u^\ell D^2_{j\ell} u^j D^2_{i\ell} u^\ell \big| \leq \sup_{1\leq i \leq N} \sum_{\substack{1\leq \ell\leq N \\\ell \neq i}} |D^2_{i\ell} u^\ell|^2 \bigg( \sum_{\substack{1\leq j,\ell\leq N \\ j \neq \ell \neq i}} |D^2_{ij} u^\ell|^2 \bigg)^{\frac12} \leq C_2^{\frac12} \sup_{1\leq i \leq N} w^i;
\]
by the Cauchy--Schwarz inequality, \Cref{ENS_uijest} and Young's inequality,
\[\begin{split}
\sum_{\substack{1\leq j,\ell\leq N \\ j \neq \ell \neq i}} \big| D_j u^\ell D^3_{ij\ell} u^j D^2_{i\ell} u^\ell \big| &\leq \bigg( \sup_{1\leq \ell \leq N} \sum_{\substack{1\leq j\leq N \\ j \neq \ell }} |D_j u^\ell|^2 \bigg)^{\frac12} \bigg( \sum_{\substack{1\leq j,\ell\leq N \\ j \neq \ell \neq i}} |D^3_{ij\ell} u^j|^2 \bigg)^{\frac12} \bigg( \sum_{\substack{1\leq \ell\leq N \\ \ell \neq i}} |D^2_{i\ell} u^\ell|^2 \bigg)^{\frac12} \\
&\leq \frac{c_1}{\sqrt{N}} \bigg( \sum_{\substack{1\leq j,\ell\leq N \\ \ell \neq i}} |D^3_{ij\ell} u^\ell|^2 + \sum_{1\leq j \leq N} |D^3_{iij} u^i|^2 \bigg)^{\frac12} \bigg( \sum_{\substack{1\leq \ell\leq N \\ \ell \neq i}} |D^2_{i\ell} u^\ell|^2 \bigg)^{\frac12} \\
&\leq \frac{\sigma}{2N} \sum_{\substack{1\leq j,\ell\leq N \\ \ell \neq i}} |D^3_{ij\ell} u^\ell|^2 + \frac{\sigma}{2N} \sum_{1\leq j \leq N} |D^3_{iij} u^i|^2 + \frac{c_1^2}{\sigma}\, w^i;
\end{split}\]
by the Cauchy--Schwarz inequality, Young's inequality and \Cref{ENS_d2b},
\[\begin{split}
\sum_{1\leq j\leq N} \big| D^2_{ij} u^i D^2_{ij} u^j D^2_{ii} u^i \big| &\leq |D^2_{ii} u^i| \bigg( \sum_{\substack{1\leq j\leq N \\ j \neq i}} |D^2_{ij} u^i|^2 \bigg)^{\frac12} \bigg( \sum_{\substack{1\leq j\leq N \\ j \neq i}} |D^2_{ij} u^j|^2 \bigg)^{\frac12} + |D^2_{ii} u^i|^3 \\
&\leq 2C_2^{\frac32} + C_2^{\frac12} w^i;
\end{split}
\]
by the Cauchy--Schwarz inequality, Young's inequality and \Cref{ENS_uijest},
\[\begin{split}
\sum_{\substack{1\leq j\leq N \\ j \neq i}} \big| D_j u^i D^3_{iij} u^j D^2_{ii} u^i \big| &\leq |D^2_{ii} u^i| \bigg( \sum_{\substack{1\leq j\leq N \\ j \neq i}} |D_{j} u^i|^2 \bigg)^{\frac12} \bigg( \sum_{\substack{1\leq j\leq N \\ j \neq i}} |D^3_{iij} u^j|^2 \bigg)^{\frac12} \\
&\leq \frac{c_1^2}{\sigma}\, w^i + \frac\sigma2 \sum_{\substack{1\leq j\leq N \\ j \neq i}} |D^3_{iij} u^j|^2;
\end{split}\]
by Young's inequality,
\[
\sum_{\substack{1\leq \ell\leq N \\ \ell \neq i}} D^2_{i\ell} f^\ell D^2_{i\ell} u^\ell + \frac1N\, D^2_{ii} f^i D^2_{ii} u^i \leq \frac12 \bigg( \sum_{\substack{1\leq\ell\leq N \\ \ell \neq i}} |D^2_{i\ell} f^\ell|^2 + \frac1N\, |D^2_{ii} f^i|^2 \bigg) + w^i.
\]
Therefore, recalling that $C_2$ has the more precise form \eqref{C2form}, we get
\[
\begin{multlined}[.95\displaywidth]
-\de_t w^i - \tr((\sigma\msfb I + \beta \msfb J)D^2 w^i) + \sum_{1\leq j \leq N} D_j u^j \cdot D_j w^i \\
\leq \Bigl(1 + 4M + \frac{2c_1^2}\sigma + e^{c_2(1+\sigma^{-1})/2}\Bigr) w^i + e^{c_2(1+\sigma^{-1})/2}  \sup_{1 \leq \ell \leq N} w^\ell \\ + \frac{2e^{2c_2(1+\sigma^{-1})/3}}{N} + \frac12 \bigg( \sum_{\substack{1\leq\ell\leq N \\ \ell \neq i}} |D^2_{i\ell} f^\ell|^2 + \frac1N\, |D^2_{ii} f^i|^2 \bigg).
\end{multlined}
\]
By {\bfseries(\hyperlink{MF}{MF})}, $\sum_{1\leq\ell\leq N,\,\ell \neq i} |D^2_{i\ell} f^\ell|^2 + \frac1N |D^2_{ii} f^i|^2 \leq {2L_f^2}/{N}$,
thus
\[
-\de_t w^i - \tr((\sigma\msfb I + \beta \msfb J)D^2 w^i) + \sum_{1\leq j \leq N} D_j u^j \cdot D_j w^i \leq e^{\hat c_2(1+\sigma^{-1})}  \Big( \sup_{1 \leq \ell \leq N} w^\ell + \frac1N \Big),
\]
where the constant $\hat c_2$ depends only on $T$, $M$, $L_g$ and $L_f$. % As $\phi(\bs x) = 1+|\bs x|^2$ is a Lyapunov function for the parabolic operator on the left-hand 
By the maximum principle we get
\[
\norm{w^i(t,\var)}_\infty \leq \norm{w^i(T,\var)}_\infty + (T-t)  e^{\hat c_2(1+\sigma^{-1})} \Big( \sup_{1 \leq \ell \leq N} \norm{w^\ell}_{\infty;[T-s,T]} + \frac1N \Big).
\]
for any $s \in (0,T-\tau]$ and $t \in [T-s, T]$; here we used the notation $\norm{\,\cdot\,}_{\infty;I}$ for the standard norm of $L^\infty(I \times (\R^d)^N)$.
For $s \leq (2e^{\hat c_2(1+\sigma^{-1})} )^{-1}$ we obtain
\begin{equation} \label{ENS_estpm2}
\sup_{1 \leq i \leq N} \norm{w^i}_{\infty;[T-s,T]} \leq 2 \sup_{1\leq i \leq N}\norm{w^i(T,\var)}_\infty + \frac{1}{N};
\end{equation}
letting $s = (2e^{\hat c_2(1+\sigma^{-1})} )^{-1} \wedge (T-\tau)$ and iterating estimate~\eqref{ENS_estpm2} on the intervals $[T-\ell s, T-(\ell-1)s]$ for all positive integers $\ell \leq \ell^* \defeq \lfloor (T-\tau)/s \rfloor$ and then on $[\tau, T-\ell^* s]$, we get by induction
%\[
%\sup_{1 \leq i \leq N} \norm{w^i}_{\infty;[T-\ell s,T-(\ell-1)s]} \leq 2^\ell \norm{w^i(T,\cdot)}_\infty + \frac{2^\ell - 1}{N},
%\]
%\[
%\sup_{1 \leq i \leq N} \norm{w^i}_{\infty;[\tau ,T-\ell^* s]} \leq 2^{\ell^*+1} \norm{w^i(T,\cdot)}_\infty + \frac{2^{\ell^*+1} - 1}{N},
%\]
%and thus
\[
\sup_{1 \leq i \leq N} \norm{w^i}_{\infty;[\tau,T]} \leq 2(2^{\ell^*+1}-1) \Big(
2 \sup_{1\leq i \leq N} \norm{w^i(T,\var)}_\infty + \frac{1}N \Big) - \frac{\ell^*+1}{N}.
\]
The conclusion follows recalling that $u^i(T,\var) = g^i$, thus by {\bfseries(\hyperlink{MF}{MF})} we have $\sup_{1 \leq i \leq N} \norm{w^i(T,\var)} \leq 2L_g^2/N$. In particular we obtain the desired estimate with
\begin{equation}\label{C3form}
C_3 = e^{e^{c_3(1+\sigma^{-1})}},
\end{equation}
where $c_3$ depends only on $T$, $M$, $L_g$, and $L_f$.
%\begin{equation*}%\label{realfgass3}
%\sup_{1\leq i \leq N} \Big\|  \sum_{\substack{1\leq\ell\leq N \\ \ell \neq i}} \big|D^2_{i\ell} g^\ell\big|^2\Big\|_\infty + \frac1N\, \left\| \big|D^2_{ii} g^i\big|^2 \right\|_\infty \leq \frac{2 L_g^2}N. \qedhere
%\end{equation*}
\end{proof}

%On the other hand, the upcoming estimate is the counterpart of the previous one, in the sense that the sum is over both indices of derivatives, one of them making all of them skew with respect to the index of the value function, which is now fixed. In this sense, we will refer to it as a \emph{horizontal} estimate.

In order to proceed, we need further control, for fixed $i$, on second-order derivatives $D_{jk} u^i$ of the value function as $j,k$ varies. We will refer to it as a \emph{horizontal} estimate. Differently from the one obtained in \Cref{ENS_d2b}, we sum over $k \neq i$, thus expecting a decay of order $1/N$. To achieve this, we need more precise control on first-order derivatives: in \Cref{ENS_uijest}, a cumulative information on $D u^i$ was stated; here below, we show \emph{pointwise} bounds, in the sense that no sum over the direction $x^j$ is involved.

\begin{prop}[Pointwise estimate on skew first-order derivatives] \label{ENS_uijestp}
Let $\tau \in [0,T)$. Suppose that $\Dop[u] \geq -M\QA$ on $(\tau, T]\times (\R^d)^{2N}$ and assume that {\bfseries(\hyperlink{MF}{MF})} holds. Then
\[
\sup_{\substack{1\leq i,k \leq N \\ k \neq i}} \|D_k u^i(t,\var) \|_\infty \leq \frac{C_1}{N} \quad \forall\, t \in [\tau, T],
\]
where the constant $C_1$ depends only on $\sigma$, $T$, $M$, $L_g$ and $L_f$.
\end{prop}

\begin{proof} Let $i, k \in \{ 1, \ldots, N\}$, $i \neq k$, and $\ell \in \{ 1, \ldots, d\}$. Note that $v = v^{k\ell} = D_{x^{k\ell}} u^i$ solves the linear problem
\[
\begin{cases}
-\de_t v- \tr((\sigma \msfb I + \beta \msfb J)D^2 v) + \displaystyle\sum_{1 \leq j \leq N} D_j u^j \cdot D_j v =  D_{x^{k\ell}} f^i
- \displaystyle\sum_{\substack{1\leq j \leq N \\ j \neq i}} D_j (D_{x^{k\ell}} u^j) \cdot D_j u^i \vspace{-10pt} \\
 v|_{t=T} = D_{x^{k\ell}} g^i.
\end{cases}
\]
Recall that by {\bfseries(\hyperlink{MF}{MF})}, $\|D_{x^{k\ell}} f^i\|_{\infty} \leq L_f / N$. For any $s \in (0,T-\tau]$ and $t \in [T-s, T]$, we control the  second term of the right-hand side of the previous equation using \Cref{ENS_uijest}, \Cref{ENS_d2N} and \Cref{ENS_d2b} as follows
\[
\begin{multlined}[.95\displaywidth]
\Biggr| \displaystyle\sum_{\substack{1\leq j \leq N \\ k \neq j \neq i}} D_j (D_{x^{k\ell}} u^j) \cdot D_j u^i + D_k (D_{x^{k\ell}} u^k) \cdot D_k u^i \Biggl| \\[-7pt]
\begin{split}
&\leq
\biggl(\, \sum_{ j \neq k}  |D^2_{jk}u^j| ^2  \biggr)^{\frac12} \biggl(\, \sum_{ j \neq i}  |D_{j}u^i| ^2  \biggr)^{\frac12} + |D^2_{kk} u^k| \, |D_k u^i| \\
&\leq \frac{c_1 C_3}N + C_2 \|D_k u^i\|_{\infty; [T-s, T]}.
\end{split}
\end{multlined}
\]
Therefore, by the maximum principle we have
\[
\|v\|_{\infty; [T-s, T]} \leq s\, \frac{c_1 C_3 + L_f}N + s C_2 \|D_k u^i\|_{\infty; [T-s, T]} + \|D_k g^i\|_\infty.
\]
Since the previous estimate holds for all $\ell \in \{ 1, \ldots, d\}$, for $s \leq (2 \sqrt{d}\, C_2)^{-1}$ we conclude that
\[
\|D_k u^i\|_{\infty; [T-s, T]} \leq \frac{c_1 C_3 + L_f}{C_2 N} +  2\sqrt{d}\, \|D_k g^i\|_\infty.
\]
By iterating the previous inequality as in the previous proof, and using {\bfseries(\hyperlink{MF}{MF})} again to control $\|D_{x^{k}} g^i\|_{\infty} \leq L_g / N$ we obtain the assertion on the time interval $[T-\tau, T]$. In particular, by replacing $C_2$ and $C_3$ with the more precise forms \eqref{C2form} and \eqref{C3form}, one obtains the desired estimate with
\begin{equation}\label{C1form}
C_1 = e^{e^{\hat c_1(1+\sigma^{-1})}},
\end{equation}
where $\hat c_1$ depends only on $T$, $M$, $L_g$, and $L_f$.
\end{proof}

%Clearly, in the previous proof we are merely using
%\begin{equation}\label{realfgass25}
%\sup_{k \neq i}\|D_{x^{k}} f^i\|_{\infty} \leq L_f / N , \qquad \sup_{k \neq i} \|D_{x^{k}} g^i\|_{\infty} \leq L_g / N.
%\end{equation}
We can now get the horizontal bounds on skew second-order derivatives.

\begin{prop}[Horizontal estimate on skew second-order derivatives] \label{ENS_d2N'}
Let $\tau \in [0,T)$ and assume that {\bfseries(\hyperlink{MF}{MF})} holds. Suppose that $\Dop[u] \geq -M\QA$ on $(\tau, T]\times (\R^d)^{2N}$. Then
\[
\sup_{1\leq i \leq N} \sum_{\substack{1\leq k \leq N \\ k \neq i}}\norm{D(D_k u^i)(t,\var)}_\infty^2 = \sup_{1 \leq i \leq N} \sum_{\substack{1\leq k \leq N \\ k \neq i}} \bigg\lVert \sum_{1\leq j\leq N} \big|D^2_{jk} u^i(t,\var)\big|^2 \bigg\rVert_\infty \leq \frac{C_4}{N} \quad \forall\, t \in [\tau, T],
\]
where the constant $C_4$ depends only on $\sigma$, $T$, $M$, $L_g$ and $L_f$.
\end{prop}

\begin{proof}
From inequality~\eqref{ENS_1estd2} we have
\begin{equation} \label{ENS_1estd2'}
\begin{multlined}[t][.9\displaywidth]
\norm{D_{x^{k\ell}} Du^i}_\infty
 \leq \frac{C s^{\frac12}e^{2Ms} }{\sigma^{\frac12}}\norm{D_{x^{k\ell}} f^i}_\infty+ \norm{D(D_{x^{k\ell}}g^i)}_\infty e^{Ms}
\\+ \frac{C s^{\frac12}e^{2Ms} }{\sigma^{\frac12}} \bigg\lVert \sum_{\substack{1\leq j\leq N \\ k \neq j \neq i}} D_j(D_{x^{k\ell}} u^j) D_j u^i \bigg\rVert_\infty + \frac{C s^{\frac12}e^{2Ms} }{\sigma^{\frac12}}\norm{D_k(D_{x^{k\ell}}u^k)D_k u^i}_\infty.
\end{multlined}\end{equation}
By the Cauchy--Schwarz inequality, \Cref{ENS_uijest} and \Cref{ENS_d2N},
\[
\bigg\lVert \sum_{\substack{1\leq j\leq N \\ k \neq j \neq i}} D_j(D_{x^{k\ell}} u^j) D_j u^i \bigg\rVert_\infty^2 \leq \frac{c_1^2}{N} \bigg\lVert \sum_{\substack{1\leq j\leq N \\ j \neq k}} \big|D^2_{jk} u^j\big|^2 \bigg\rVert_\infty \leq \frac{c_1^2 C_3}{N^2};
\]
therefore from \eqref{ENS_1estd2'} and Young's inequality we get
\begin{equation} \label{ENS_2estd2'}
\begin{multlined}[t][.9\displaywidth]
\norm{D(D_k u^i)}_\infty^2 \\
\lesssim \frac{C^2 e^{4Ms} s}\sigma \norm{D_k f^i}_\infty^2  + \norm{D(D_kg^i)}_\infty^2 e^{2Ms}
 + \frac{C^2 c^2_1 C_3 e^{4Ms} s }{\sigma N^2} + \frac{C^2 e^{4Ms} s}{\sigma} \norm{D^2_{kk} u^k}_\infty^2 \norm{D_k u^i}_\infty^2,
\end{multlined}
\end{equation}
where the implied constant depends only on $d$. By {\bfseries(\hyperlink{MF}{MF})}, $\sum_{1\leq k\leq N,\, k \neq i} \norm{D_k f^i}_\infty^2 \leq {L_f^2}/{N}$ and $\sum_{1\leq k\leq N,\, k \neq i} \norm{D(D_k g^i)}_\infty^2 \leq {L_g^2}/{N}$;
by \Cref{ENS_d2b} and \Cref{ENS_uijestp}
\begin{equation} \label{ENS_estwithp}
\sum_{\substack{1 \leq k \leq N \\ k \neq i}} \norm{D^2_{kk} u^k}_\infty^2 \norm{D_k u^i}_\infty^2 \leq \frac{C_1C_2}{N}.
\end{equation}
The desired conclusion now follows, summing \eqref{ENS_2estd2'} over $k \neq i$. By the more explicit forms of $C_1$, $C_2$, and $C_3$ from the previous proposition we get in fact the desired estimate with
\begin{equation}\label{C4form}
C_4 = e^{e^{c_4(1+\sigma^{-1})}},
\end{equation}
where $c_4$ depends only on $T$, $M$, $L_g$, and $L_f$.
\end{proof}

\begin{rmk} \label{ENS_d2N'point}
If, in addition, one assumes that
\begin{equation} \label{extraassMF1}
\sup_{\substack{1 \leq i,k \leq N \\ i \neq k}} \norm{D(D_kg^i)}_\infty \le \frac{L_g}{N},
\end{equation}
then 
\[
\sup_{\substack{1 \leq i,k \leq N \\ i \neq k}} \norm{D(D_k u^i)(t,\var)}_\infty^2 = \sup_{\substack{1 \leq i,k \leq N \\ i \neq k}} \bigg\lVert \sum_{1\leq j\leq N} \big|D^2_{jk} u^i(t,\var)\big|^2 \bigg\rVert_\infty \leq \frac{\hat C_4}{N^2} \defeq \frac{e^{e^{\hat c_4(1+\sigma^{-1})}}}{N^2} \qquad \forall\, t \in [\tau, T],
\]
where $\hat c_4$ depends only on $T$, $M$, $L_g$, and $L_f$. Indeed, it suffices to proceed as in the previous proof but without summing over $k \neq i$.
\end{rmk}

\begin{rmk}
Even if \Cref{ENS_uijestp} is clearly stronger than \Cref{ENS_uijest}, we cited them both in the previous proof in order to highlight when the weaker \Cref{ENS_uijest} was sufficient and, consequently, that the stronger \Cref{ENS_uijestp} is only needed to get estimate~\eqref{ENS_estwithp}.
\end{rmk}

\subsection{Third-order derivatives of the value functions}

Finally, we prove the most basic third-order version of the transversal estimate on skew-second order derivatives.

\begin{prop}[Transversal estimate on third-order derivatives] \label{ENS_d3b}
Let $\tau \in [0,T)$ and assume that {\bfseries(\hyperlink{MF}{MF})} holds. Suppose that $\Dop[u] \geq -M\QA$ on $(\tau, T]\times (\R^d)^{2N}$. Then
\[
\sum_{\substack{1\leq i,j \leq N \\ j\neq i}} \norm{D(D_{ij} u^j)(t,\var)}_\infty^2 = \sum_{\substack{1\leq i,j \leq N \\ j\neq i}} \bigg\lVert \sum_{1 \leq k \leq N} \big| D^3_{ijk} u^j (t,\var) \big|^2 \bigg\rVert_\infty \leq C_5 \quad \forall\, t \in [\tau, T],
\]
where the constant $C_5$ depends only on $\sigma$, $T$, $M$, $L_g$ and $L_f$.
\end{prop}

\begin{proof}
Proceed as in the proof of \Cref{ENS_d2b}. Let $h,k \in \{1,\dots,N\}$, $\ell,m \in \{1,\dots,d\}$ and $s \in [0,T-\tau]$. Note that $v = D_{x^{k\ell}x^{hm}} u^i$ solves problem~\eqref{ENS_sysv} on $[T-s,T] \times (\R^d)^N$ with $b^j = D_j u^j$,
\[
\begin{split}
F &= D^2_{x^{k\ell}x^{hm}} f^i - \sum_{\substack{1\leq j\leq N \\ j \neq i}} D_j(D^2_{x^{k\ell}x^{hm}} u^j) D_j u^i - \sum_{\substack{1\leq j\leq N \\ j \neq i}} D_j(D_{x^{k\ell}} u^j) D_j(D_{x^{hm}}u^i) \\
&\quad\ - \sum_{\substack{1\leq j\leq N \\ j \neq i}} D_j(D_{x^{hm}}u^j)  D_j(D_{x^{k\ell}} u^i) - D_i(D_{x^{hm}}u^i)  D_i(D_{x^{k\ell}} u^i)
\end{split}
\]
and $G = D^2_{x^{k\ell}x^{hm}} g^i$.
By \Cref{ENS_estv}, there is $C$ independent of $N$ and $\sigma$ such that
\begin{equation} \label{ENS_1estd3}
\begin{multlined}[t][.85\displaywidth]
\norm{D (D^2_{x^{k\ell}x^{hm}} u^i)}_\infty 
\\
\begin{multlined}[t][.85\displaywidth]
\lesssim\, \frac{C}{\sigma^{\frac12}} \biggl(\, \bigg\lVert \sum_{\substack{1\leq j\leq N \\ j \neq i}} D_j(D^2_{x^{k\ell}x^{hm}} u^j) D_j u^i \bigg\rVert_\infty
+ \bigg\lVert \sum_{\substack{1\leq j\leq N \\ j \neq i}} D_j(D_{x^{k\ell}}u^j)  D_j(D_{x^{hm}} u^i) \bigg\rVert_\infty
\\
+ \bigg\lVert \sum_{\substack{1\leq j\leq N \\ h \neq j \neq i}} D_j(D_{x^{hm}}u^j)  D_j(D_{x^{k\ell}} u^i) \bigg\rVert_\infty + \norm{D_h(D_{x^{hm}}u^h)  D_j(D_{x^{k\ell}} u^i)}_\infty \\[-5pt]
 + \norm{D_i(D_{x^{hm}}u^i)  D_i(D_{x^{k\ell}} u^i)}_\infty
 + \norm{D^2_{x^{k\ell}x^{hm}} f^i}_\infty \biggr) s^{\frac12}e^{2Ms} + \norm{D^2_{x^{k\ell}x^{hm}}Dg^i}_\infty e^{Ms}, 
 \end{multlined}
\end{multlined}\end{equation}
where the implied constant is a number and the $L^\infty$-norms are understood to be computed on $[T-s,T] \times (\R^d)^N$ or $(\R^d)^N$. By \Cref{ENS_uijest} and the Cauchy--Schwarz inequality we have
\[
\bigg\lVert \sum_{\substack{1\leq j\leq N \\ j \neq i}} D_j(D_{x^{k\ell}x^{hm}}^2 u^j) D_j u^i \bigg\rVert_\infty^2 \leq \frac{c_1^2}{N} \bigg\lVert \sum_{\substack{1\leq j\leq N \\ j \neq i}} \big| D^3_{hjk} u^j\big|^2 \bigg\rVert_\infty;
\]
therefore, choosing $k=i$ in \eqref{ENS_1estd3}, and applying Young's inequality we get
\begin{equation} \label{ENS_2estd3}
 \begin{multlined}[t][.9\displaywidth]
\norm{D(D_{hi}^2 u^i)}_\infty^2 \lesssim\, \frac{s e^{4Ms}}{\sigma} \biggl(  \norm{D^2_{hi} f^i}_\infty^2 
 + \bigg\lVert \sum_{\substack{1 \leq j \leq N \\ j \neq i}} \big| D^2_{ji} u^j \big|^2  \sum_{\substack{1 \leq j \leq N \\ j \neq i}} \big| D^2_{hj} u^i \big|^2 \bigg\rVert_\infty
 \\
\qquad + \bigg\lVert \sum_{\substack{1 \leq j \leq N \\ j \neq h}} \big| D^2_{hj} u^j \big|^2  \sum_{\substack{1 \leq j \leq N \\ j \neq i}} \big| D^2_{ji} u^i \big|^2 \bigg\rVert_\infty +  \norm{D_{hi} u^i}_\infty^2 \bigl( \norm{D_{ii} u^i}_\infty^2 +  \norm{D_{hh} u^h}_\infty^2 \bigr)
 \vspace{-5pt} 
\\
+ \frac{c_1^2}{N} \sum_{\substack{1\leq j\leq N \\ j \neq i}}\norm{D^3_{hij}u^j}_\infty^2 \biggr) +\norm{D^2_{hi}Dg^i}_\infty^2 e^{2Ms}
\end{multlined}
\end{equation}
where the implied constant depends only on $C$ and $d$. Our goal now is to control the right-hand side of the previous inequality under the sum over $h$ and $i \neq h$.
First, by assumption~{\bfseries(\hyperlink{MF}{MF})}, $\sum_{h,i: i \neq h} \norm{D^2_{hi} f^i}_\infty^2 \leq L_f^2$;
then by \Cref{ENS_d2N,ENS_d2N'},
\[
\sup_{1 \leq i \leq N} \bigg\lVert \sum_{\substack{1\leq j \leq N \\ j \neq i}} \big| D^2_{ij} u^j\big|^2 \bigg\rVert_\infty \sum_{\substack{1 \leq h,i \leq N \\ h \neq i}} \bigg\lVert \sum_{1\leq j \leq N} \big| D^2_{hj} u^i\big|^2 \bigg\rVert_\infty \leq \frac{C_3C_4}{N}
\]
and 
\[
\sum_{1 \leq h \leq N} \bigg\lVert \sum_{\substack{1 \leq j \leq N \\ j \neq h}} \big| D^2_{hj} u^j \big|^2  \bigg\rVert_\infty \sum_{1 \leq i \leq N} \bigg\lVert \sum_{\substack{1 \leq j \leq N \\ j \neq i}} \big| D^2_{ji} u^i \big|^2 \bigg\rVert_\infty \le C_3 C_4;
\]
Finally, by \Cref{ENS_d2b,ENS_d2N},
\[
 \sum_{\substack{1 \leq h,i \leq N \\ i \neq h}} \norm{D_{hi} u^i}_\infty^2 \bigl(\norm{D_{ii} u^i}_\infty^2 + \norm{D_{hh} u^h}_\infty^2\bigr) \leq 2C_2C_3.
\]
Therefore, summing inequality~\eqref{ENS_2estd3} over $h$ and $i \neq h$ yields
\begin{equation}\label{megasommafinale}
\begin{split}
\sum_{\substack{1\leq h,i \leq N \\ i\neq h}} \norm{D(D^2_{hi} u^i)}_\infty^2 &\lesssim \frac{s e^{4Ms}}\sigma\left(1 + C_3C_4 + C_2 C_3 \right) + e^{2Ms} \sum_{\substack{1\leq h,i \leq N \\ i\neq h}} \norm{D(D^2_{hi}g^i)}_\infty^2 \\ 
&\quad + \frac{1}{\sigma N} \sum_{\substack{1\leq h,i,j\leq N \\ j \neq i \neq h}}\norm{D^3_{hij}u^j}_\infty^2  s e^{4Ms},
\end{split}
\end{equation}
where the implied constant (here as well as below) depends only on $T$, $M$, $L_f$ and $L_g$. Exploiting now the fact that
\[
\sum_{\substack{1\leq h,i,j\leq N \\ j \neq i \neq h}}\norm{D^3_{hij}u^j}_\infty^2 \leq N \sum_{\substack{1\leq i,j \leq N \\ j\neq i}} \bigg\lVert \sum_{1 \leq h \leq N} \big| D^3_{hij} u^j \big|^2 \bigg\rVert_\infty = N \sum_{\substack{1\leq i,j \leq N \\ j\neq i}} \norm{D(D_{ij} u^j)}_\infty^2,
\]
we obtain, using the explicit form of $C_2, C_3, C_4$ and $s \lesssim \sigma$,
\[
\sum_{\substack{1\leq h,i \leq N \\ i\neq h}} \norm{D(D^2_{hi} u^i)}_\infty^2 \lesssim e^{e^{\hat c_4(1+\sigma^{-1})}} + \sum_{\substack{1\leq h,i \leq N \\ i\neq h}} \norm{D(D^2_{hi}g^i)}_\infty^2. %= e^{e^{\hat c_4(1+\sigma^{-1})}} + \sum_{\substack{1\leq h,i \leq N \\ i\neq h}} \norm{D(D^2_{hi}u^i)(T,\cdot)}_\infty^2.
\]
for some $\hat c_4$ depending only on $T$, $M$, $L_g$, and $L_f$.
Using that
\begin{equation} \label{MFG3agg}
\sum_{\substack{1\leq h,i \leq N \\ i\neq h}} \norm{D(D^2_{hi}g^i)}_\infty^2 \leq L_g^2
\end{equation}
due to assumption~{\bfseries(\hyperlink{MF}{MF})}, one concludes by arguing as in the proof of \Cref{ENS_d2b}, with in particular
\begin{equation}\label{C5form}
C_5 = e^{e^{c_5(1+\sigma^{-1})}},
\end{equation}
with $c_5$ depending only on $T$, $M$, $L_g$, and $L_f$.
\end{proof}

%\begin{rmk} \ros{NEW REMARK} Assumption ~{\bfseries(\hyperlink{MF}{MF})} above is sufficient to obtain all the previous estimates. Nevertheless, what is actually used appears in \eqref{ENS_ineqd1f}, \eqref{ENS_ineqd1g}, \eqref{realfgass1}, \eqref{realfgass2},  \eqref{realfgass3}, \eqref{realfgass25}, \eqref{realfgass4}, \eqref{realfgass6}, \eqref{realfgass7} and can be summarized as follows ():
%\begin{align}
%\sup_{i, x, y} \frac{|f^i(\bs x) - f^i(\bs y)|}{\norm{\bs x - \bs y}^i} + \sup_{1\leq k \leq N} \sum_{1\leq i\leq N} \norm{D_k f^i}_\infty + N \sum_{\substack{1\leq\ell\leq N \\ \ell \neq i}} |D^2_{i\ell} f^\ell|^2 +\, |D^2_{ii} f^i|^2  \\
%N \sup_{k \neq i}\|D_{x^{k}} f^i\|_{\infty} + N \sum_{\substack{1\leq k\leq N \\ k \neq i}} \norm{D_k f^i}_\infty^2 + \sum_{\substack{1\leq h,i \leq N \\ i \neq h}} \norm{D^2_{hi} f^i}_\infty^2
%\end{align}
%and
%\begin{align}
%\sup_{i, x, y} \frac{|g^i(\bs x) - g^i(\bs y)|}{\norm{\bs x - \bs y}^i} + \sup_{1\leq k \leq N}  \sum_{1\leq i\leq N} \norm{D(D_k g^i)}^2_\infty + \sup_{1\leq i \leq N} \Big\|  N \sum_{\substack{1\leq\ell\leq N \\ \ell \neq i}} \big|D^2_{i\ell} g^\ell\big|^2\Big\|_\infty + \left\| \big|D^2_{ii} g^i\big|^2 \right\|_\infty \\
%N \sup_{k \neq i}\|D_{x^{k}} g^i\|_{\infty} +  N \sum_{\substack{1\leq k\leq N \\ k \neq i}} \norm{D(D_k g^i)}_\infty^2 + \sum_{\substack{1\leq h,i \leq N \\ i\neq h}} \norm{D(D^2_{hi}g^i)}_\infty^2
%\end{align}
%
%\end{rmk}  

\begin{rmk} \label{ENS_d3N}
If, in addition, one assumes \eqref{extraassMF1} and
\begin{equation} \label{extraassMF2}
\sup_{1 \leq h \leq N} \sum_{\substack{1 \leq i \leq N \\ i \neq h}} \norm{D(D^2_{hi}g^i)}_\infty^2 \leq \frac{L_g^2}{N},
\end{equation}
then one can obtain a stronger estimate, by proceeding as follows. We do not sum over $h$ in the previous proof, but we only sum over $i$, $i \neq h$, thus getting, instead of \eqref{megasommafinale},
\begin{equation*} \begin{split}
\sum_{\substack{1\leq i \leq N \\ i\neq h}} \norm{D(D^2_{hi} u^i)}_\infty^2 &\lesssim \frac{s e^{4Ms}}{\sigma N}\bigl(1 + C_3C_4 + C_2 \hat C_4 \bigr) + e^{2Ms} \sum_{\substack{1\leq i \leq N \\ i\neq h}} \norm{D(D^2_{hi}g^i)}_\infty^2 \\[-3pt]
&\quad + \frac{1}{\sigma N} \sum_{\substack{1\leq i,j \leq N \\ j \neq i \neq h}}\norm{D^3_{hij}u^j}_\infty^2  s e^{4Ms},
\end{split}
\end{equation*}
by using the additional estimate in \Cref{ENS_d2N'point}.
Now we note that
\[\begin{split}
\frac{1}{N} \sum_{\substack{ij \\ j \neq i \neq h}} \norm{D^3_{hij} u^j}_\infty^2 &\leq \frac1N \sum_{\substack{1 \leq i \leq N \\ i \neq h}} \norm{D^3_{ihh} u^h}_\infty^2 + \sum_{\substack{1 \leq j \leq N \\ j \neq h}} \norm{D(D_{hj}^2 u^j)}_\infty^2 
\\
&\leq \frac{C_5}N + \sum_{\substack{1 \leq i \leq N \\ i \neq h}} \norm{D(D_{hi}^2 u^i)}_\infty^2 ,
\end{split}
\]
where we also used \Cref{ENS_d3b} for the second estimate. Therefore, we can then argue as in the proof of \Cref{ENS_d3b} (using \eqref{extraassMF2} in place of \eqref{MFG3agg}) to obtain
\[
\sup_{1 \leq h \leq N} \sum_{\substack{1 \leq i \leq N \\ i \neq h}} \norm{D(D^2_{hi} u^i)(t,\var)}_\infty^2 \leq \frac{\hat C_5}{N} \defeq \frac{e^{e^{\hat c_5(1+\sigma^{-1})}}}{N} \qquad \forall\, t \in [\tau, T],
\]
where $\hat c_5$ depends only on $T$, $M$, $L_g$, and $L_f$.
\end{rmk}

\section{The \texorpdfstring{\Ds semimonotone}{D-semimonotone} case}\label{sec_Ds}

We prove in this section \Cref{ENS_mainthm} in the \Ds semimonotone case, and for quadratic Hamiltonians $H^i(x^i, p) = \frac12|p|^2$. The following result shows the crucial interplay between the previous estimates and the \Ds semimonotonicity of the value functions. It basically shows that if there exists a left temporal open neighbourhood $(\tau, T]$ of $T$ in which $u$ is $M$-\Ds semimonotone for some suitable $M$, then exploiting the estimates on the skew derivatives one can show that such a semimonotonicity actually holds with constant $\frac12M$ on the whole $[\tau, T] \times (\R^d)^N$, provided that $N$ is large.

This implies that if the matrix $(D_{ij} u^j)_{1 \leq i,j \leq N} + M\msfb I$ is positive definite near $T$, then it cannot degenerate at any time in $[0,T]$. In this sense, we can consider this fact as a (backward) \emph{propagation} of the semimonotonicity. Such an argument will in turn imply that all estimates on the derivatives themselves propagate, thus small-time existence of a \Ds semimonotone solution to the Nash system is sufficient to prove global existence on $[0,T]$ of a solution which satisfies all the above estimates (see \Cref{ENS_mainthmDS} below).

Solutions to the Nash system are as before assumed to be classical, locally $C^1$ in time and $C^4$ in space, with bounded derivatives. Moreover, we require second order derivatives in space to be uniformly continuous.

\begin{prop}[Improvement of \Ds semimonotonicity] \label{ENS_dscprop} Assume {\bfseries(\hyperlink{MF}{MF})}.
Then there exist positive constants $M^*$, $M_f^*$ and $M_g^*$ (with $M_g^* < M^*$) depending on $T$, and a natural number $N_*$ depending on $\sigma$, $T$, $L_g$ and $L_f$ such that for any $\tau \in [0,T)$ one has
\[
\begin{rcases}
\Dop[u] \geq -M^*\QA & \text{on}\ (\tau, T] \times (\R^d)^{2N} \\
\Dop[f] \geq -M_f^*\QA & \text{on} \ (\R^d)^{2N} \\
\Dop[g] \geq -M_g^*\QA & \text{on} \ (\R^d)^{2N} \\
N \geq N_*
\end{rcases}
\quad \implies \quad 
\Dop[u(\tau,\var)] \geq - \frac12M^*\QA \ \ \text{on}\ (\R^d)^{2N}.
\]
\end{prop}

Looking at \eqref{stareq} below, one can in fact choose
\[
M_g^* = (12eT)^{-1}, \quad M_f^* = (12eT^2)^{-1}, \quad \text{and} \quad M^* = (2T)^{-1}.
\]

The proof involves again the method of doubling variables. We could proceed with a PDE approach as before, but we prefer to follow a ``dynamic'' approach; that is, we argue along optimal trajectories. This is probably closer in the spirit to what is usually done in the MFG theory.

More precisely, we use a synchronous coupling between two solutions of \eqref{ENS_NcSDEs}, for agents playing optimally. For the benefit of the reader we briefly recall what one means by such a coupling. Let $\bar t \in [0,T)$ and consider the $(\R^d)^{m}$-valued SDE
\begin{equation} \label{ENS_SDEZ}
\di \bs Z_t = \bs b(t, \bs Z_t) \,\di t + \bs\Sigma(t,\bs Z_t) \,\di \tbs B_t, \quad t \in [\bar t, T],
\end{equation}
where $\bs\Sigma$ is an $md \times \ell d$ matrix and $\tbs B_t$ is an $\ell d$-dimensional Brownian motion. Assume that the drift $\bs b$ and the diffusion coefficient $\bs\Sigma$ are locally Lipschitz continuous and have sublinear growth with respect to the state variable, uniformly with respect to time. If $\bs Z$ solves \eqref{ENS_SDEZ} with $\bs Z_{\bar t} = \bbs z \in (\R^d)^m$, then the following \emph{Dynkin's formula} holds for any $h \in C^{1,2}([\bar t,T] \times (\R^d)^m)$ such that $\de_t^\ell D^\gamma h(t,\var)$ has polynomial growth for all $t \in [\bar t,T]$ and $2\ell + |\gamma| \leq 2$:
\begin{equation} \label{ENS_itow}
\bb E\bigl[h(T,\bs Z_T)\bigr] = h(\bar t,\bbs z) + \int_{\bar t}^T \bb E\bigl[\de_t h(t,\bs Z_t) + Lh(t,\bs Z_t)\bigr]\,\di t,
\end{equation}
where $L = \tfrac12 \tr\!\big(\bs \Sigma \bs \Sigma^\trn\! D^2_{\bs z} \big) + \bs b \cdot D_{\bs z}$
is the differential generator of the process $\bs Z$; see, e.g., \cite[Chapter~9]{baldi}.
We are interested in the particular case where $m = 2N$ and $\bs Z$ is the following coupling of two solution of the $(\R^d)^N$-valued SDE
\begin{equation} \label{ENS_SDEcoup}
\di \bs X_t = \bs \alpha(t, \bs X_t) \,\di t + \bs\sigma(t,\bs X_t) \,\di \tbs B_t, \quad t \in [\bar t, T],
\end{equation}
where $\bs\sigma$ is an $Nd \times \ell d$ matrix and $\tbs B_t$ is a $\ell d$-dimensional Brownian motion. Given $\bs X$ and $\bs Y$ which both solve \eqref{ENS_SDEcoup}, with $\bs X_{\bar t} = \bbs x$ and $\bs Y_{\bar t} = \bbs y$, respectively, we say that $\bs Z = (\bs X, \bs Y)^\trn$ with $\bs Z_{\bar t} \sim \mu$ having marginals $\delta_{\bbs x}$ and $\delta_{\bbs y}$ is a \emph{synchronous coupling} of $\bs X$ and $\bs Y$. It solves equation~\eqref{ENS_SDEZ} with
$\bs b(t,\bs x, \bs y) = (\bs\alpha(t,\bs x), \bs\alpha(t,\bs y))^\trn$ and $\bs\Sigma(t,\bs x, \bs y) = (\bs\sigma(t,\bs x), \bs\sigma(t, \bs y))^\trn$,
thus
\[
\bs b \cdot D_{\bs z} = \bs\alpha(t,\bs x) \cdot D_{\bs x} + \bs\alpha(t,\bs y) \cdot D_{\bs y}, \qquad \bs\Sigma\bs\Sigma^\trn = \bigg(\, \begin{matrix}
\bs\sigma(t,\bs x)\bs\sigma(t,\bs x)^\trn & \bs\sigma(t,\bs x)\bs\sigma(t,\bs y)^\trn \\
\bs\sigma(t,\bs y)\bs\sigma(t,\bs x)^\trn & \bs\sigma(t,\bs y)\bs\sigma(t,\bs y)^\trn
\end{matrix}\, \bigg).
\]
In particular, when \eqref{ENS_SDEcoup} coincides with \eqref{ENS_NcSDEs} in the equilibrium, which happens for $\bs\alpha= \bs\alpha^* \defeq (-D_i u^i)_{1 \leq i \leq N}$,  $\tbs B = (\bs B, W)^\trn$ and $\bs\sigma \equiv \sqrt{2}\, \big(\,
\sqrt\sigma\, \msfb I \mid \sqrt{\beta\,} 1_{N} \otimes \msf I_d \,\big)$,
one has $\bs\Sigma\bs\Sigma^\trn %= 2\big( \msf J_2 \otimes \big(\,
%\msfb I \mid \sqrt{\beta\,} 1_{N} \otimes \msf I_d \,\big) \big(\,
%\msfb I \mid \sqrt{\beta\,} 1_{N} \otimes \msf I_d \,\big){}^\trn \big)
= 2(\msf J_2 \otimes (\sigma\msfb I + \beta \msfb J))$
and thus
\begin{equation} \label{ENS_Lsync}
L = \tr\!\big((\msf J_2 \otimes (\sigma\msfb I + \beta \msfb J)) D^2_{(\bs x, \bs y)}\big) + \bs\alpha^*(t,\bs x) \cdot D_{\bs x} + \bs\alpha^*(t,\bs y) \cdot D_{\bs y}.
\end{equation}
Clearly, analogous considerations hold if $\bs Z_{\bar t}$ is any random variable with $\bb E\bigl[|\bs Z_{\bar t}|^k\bigr] < \infty$ for some $k \in \N$; in this case, \eqref{ENS_itow} holds with $h(\bar t, \bbs z)$ replaced by $\bb E\bigl[ h(\bar t, \bs Z_{\bar t}) \bigr]$ and provided that $\de_t^\ell D^\gamma h(t,\var) \lesssim 1 + |\cdot|^k$.

\begin{proof}[Proof of~\Cref{ENS_dscprop}]
Let $w^i(t,\bs x,\bs y) \defeq D_i u^i(t,\var)\big|_{\bs y}^{\bs x}\cdot(x^i-y^i)$. We want to use formula~\eqref{ENS_itow} where $\bar t = \tau$, $h = w^i$ and $\bs Z = (\bs X, \bs Y)$ is a synchronous coupling of two solutions to \eqref{ENS_NcSDEs} at equilibrium; that is, formula~\eqref{ENS_itow} with $L$ given by \eqref{ENS_Lsync}.
We have
\[ \begin{multlined}[.95\displaywidth]
Lw^i(t,\bs x, \bs y) = D_i \tr\!\big((\sigma\msfb I + \beta \msfb J)D^2u^i(t,\var)\big|_{\bs y}^{\bs x}\big) \cdot (x^i - y^i) \\
- \sum_{1 \leq j \leq N} D^2_{ij} u^i(t,\var)D_j u^j(t,\var)\big|_{\bs y}^{\bs x} \cdot (x^i-y^i) - \bigl| D_i u^i(t, \var)\big|_{\bs y}^{\bs x} \bigr|^2,
\end{multlined}\]
and, according to \eqref{ENS_NS},
\[ \begin{multlined}[.95\displaywidth]
\de_t w^i(t,\bs x, \bs y) = - D_i \tr\!\big((\sigma\msfb I + \beta \msfb J) D^2u^i(t,\var)\big|_{\bs y}^{\bs x}\big) \cdot (x^i - y^i) + \sum_{1 \leq j \leq N} D^2_{ij} u^i(t,\var) D_j u^j(t,\var)\big|_{\bs y}^{\bs x} \cdot (x^i - y^i) \\
 + \sum_{\substack{1 \leq j \leq N \\ j \neq i}} D^2_{ij} u^j(t,\var) D_j u^i(t,\var)\big|_{\bs y}^{\bs x} \cdot (x^i - y^i) -  D_i f^i\big|_{\bs y}^{\bs x} \cdot (x^i - y^i).
\end{multlined}\]
Then equality~\eqref{ENS_itow} gives
\[
\bb E[w^i(T,\bs X_T,\bs Y_T)] \leq w^i(\tau,\bbs x, \bbs y) - \int_\tau^T \bb E\Big[ D_i f^i\big|_{\bs Y}^{\bs X} \cdot (X^i - Y^i) \Big] + \int_\tau^T \bb E\, \call E^i_t \,\di t,
\]
where $\call E^i_t \defeq \sum_{1 \leq j \leq N,\, j \neq i}  D^2_{ij} u^j(t,\var) D_j u^i(t,\var)\big|_{\bs Y_t}^{\bs X_t} \cdot (X_t^i - Y_t^i)$.
Letting $w \defeq \sum_{1\leq i \leq N} w^i$, it follows that %and using assumption~{\bfseries(\hyperlink{DS}{DS})}
\begin{equation} \label{ENS_estw1}
w(\tau, \bbs x, \bbs y) \geq -M^*_g\, \bb E\big[|\bs X_T - \bs Y_T|^2\big] - M^*_f \int_\tau^T \bb E\big[|\bs X_t - \bs Y_t|^2\big]\di t - \int_\tau^T \bb E \biggl[\,\sum_{1\leq i\leq N} \call E^i \,\biggr].
\end{equation}
By the fundamental theorem of calculus
\begin{equation*}
 D^2_{ij} u^j(t,\var) D_j u^i(t,\var)\big|_{\bs Y_t}^{\bs X_t} = \bigg( \int_0^1 D \big( D^2_{ij} u^j(t, \bs Z_t(s)) D_j u^i(t, \bs Z_t(s)) \big)\,\di s \bigg)  (\bs X_t - \bs Y_t) ,
\end{equation*}
with $\bs Z_t(s) = s\bs X_t+(1-s)\bs Y_t$; then note that
\begin{multline*}
\sum_{1\leq i\leq N} \call E_t^i = \int_0^1 \sum_{1\leq i,k\leq N} \sum_{\substack{1 \leq j \leq N \\ j \neq i}} \big( D^3_{ijk} u^j(t, \bs Z_t(s)) D_j u^i(t, \bs Z_t(s)) \big) (X^k_t - Y^k_t)  \cdot (X_t^i - Y_t^i) \,\di s\\
+\int_0^1 \sum_{1\leq i,k\leq N} \sum_{\substack{1 \leq j \leq N \\ j \neq i}} \big( D^2_{ij} u^j(t, \bs Z_t(s))) D^2_{jk} u^i(t, \bs Z_t(s)) \big) (X^k_t - Y^k_t)  \cdot (X_t^i - Y_t^i) \,\di s,
\end{multline*}
where
\begin{multline*}
\biggl|\, \int_0^1 \sum_{1\leq i,k\leq N} \sum_{\substack{1 \leq j \leq N \\ j \neq i}} \big( D^3_{ijk} u^j(t, \bs Z_t(s)) D_j u^i(t, \bs Z_t(s)) \big) (X^k_t - Y^k_t)  \cdot (X_t^i - Y_t^i) \,\di s \,\biggr| \\[-7pt]
\begin{split}
&\leq \biggl\| \sum_{1\leq i,k\leq N}\, \biggl| \sum_{\substack{1 \leq j \leq N \\ j \neq i}} D^3_{ijk} u^j D_j u^i \biggr|^2 \biggr\|_\infty^{\frac12} |\bs X_t - \bs Y_t|^2 \\
&\leq \biggl( \sum_{1 \leq i \leq N} \bigg\lVert \sum_{\substack{1 \leq j \leq N \\ j \neq i}} \big| D(D^2_{ij} u^j) \big|^2 \bigg\rVert_\infty \bigg\lVert \sum_{\substack{1 \leq j \leq N \\ j \neq i}}\big| D_j u^i \big|^2 \bigg\rVert_\infty \biggr)^{\frac12} |\bs X_t - \bs Y_t|^2
\end{split}
\end{multline*}
and similarly
\begin{multline*}
\biggl|\, \int_0^1 \sum_{1\leq i,k\leq N} \sum_{\substack{1 \leq j \leq N \\ j \neq i}} \big( D^2_{ij} u^j(t, \bs Z_t(s))) D^2_{jk} u^i(t, \bs Z_t(s)) \big) (X^k_t - Y^k_t)  \cdot (X_t^i - Y_t^i) \,\di s \,\biggr| \\[-7pt]
\leq \biggl( \sum_{1 \leq i \leq N} \bigg\lVert \sum_{\substack{1 \leq j \leq N \\ j \neq i}} \big| D^2_{ij} u^j \big|^2 \bigg\rVert_\infty \bigg\lVert \sum_{\substack{1 \leq j \leq N \\ j \neq i}} \big| D(D_j u^i) \big|^2 \bigg\rVert_\infty \biggr)^{\frac12}|\bs X_t - \bs Y_t|^2.
\end{multline*}
\Cref{ENS_uijest,ENS_d3b} ensure that
\[
\sum_{1 \leq i \leq N} \bigg\lVert \sum_{\substack{1 \leq j \leq N \\ j \neq i}} \big| D(D^2_{ij} u^j) \big|^2 \bigg\rVert_\infty \bigg\lVert \sum_{\substack{1 \leq j \leq N \\ j \neq i}}\big| D_j u^i \big|^2 \bigg\rVert_\infty \leq \frac{c_1^2C_5}{N}
\]
and \Cref{ENS_d2N,ENS_d2N'} give
\[
\sum_{1 \leq i \leq N} \bigg\lVert \sum_{\substack{1 \leq j \leq N \\ j \neq i}} \big| D^2_{ij} u^j \big|^2 \bigg\rVert_\infty \bigg\lVert \sum_{\substack{1 \leq j \leq N \\ j \neq i}} \big| D(D_j u^i) \big|^2 \bigg\rVert_\infty \leq \frac{C_3C_4}{N}.
\] 
Therefore we have
\[
\int_\tau^T \bb E \biggl[\, \sum_{1\leq i\leq N} \call E_t^i\, \biggr] \leq \frac{\overline C}{\sqrt{N}}\,\sup_{t \in [\tau,T]} \bb E\big[ |\bs X_t - \bs Y_t|^2 \big]
\]
for 
\begin{equation}\label{ovCform}
\overline C=c_1\sqrt{C_5} + \sqrt{C_3C_4} \le e^{e^{c_6(1+\sigma^{-1})}},
\end{equation}
where $c_6$ depends only $T$, $M^*$, $L_f$ and $L_g$. On the other hand, by the equation of $\bs X - \bs Y$,
\[
\di |\bs X_t - \bs Y_t|^2 = - 2\sum_{1 \leq i \leq N}( D_i u^i(t,\bs X_t) - D_i u^i(t,\bs Y_t) ) \cdot ( X^i_t - Y^i_t ) \,\di t \leq 2M^* |\bs X_t - \bs Y_t|^2 \,\di t,
\]
where the inequality comes from the \Ds semimonotonicity of $u$, so that Gronwall's lemma yields
\begin{equation} \label{ENS_knowest}
\bb E \big[ |\bs X_t - \bs Y_t |^2 \big] \leq | \bbs x - \bbs y|^2 e^{2M^*(t- \tau)} \quad \forall\, t \in [\tau, T].
\end{equation}
%where also the constant $C_M'$ (continuously) depends only $T$, $M$, $L_f$ and $L_g$ due to our estimates on the derivatives on the value functions. 
Hence, from \eqref{ENS_estw1} we have obtained
\[
w(\tau, \bbs x, \bbs y) \geq -e^{2M^*T}\Bigl( M^*_g + TM^*_f + \frac{\overline C}{\sqrt{N}\,}\Bigr) | \bbs x - \bbs y|^2.
\]
Fix now $M^* = (2T)^{-1}$; choosing $M^*_g$ and $M^*_f$ small enough and $N_*$ large enough so that
\begin{equation}\label{stareq}
M^*_g + TM^*_f + \frac{\overline C}{\sqrt{N_*}\,} \leq \frac{M^*}{2 e^{2M^*T}} = \frac1{4eT}
\end{equation}
%letting $F(M) \defeq M/C'_M$ and choosing
%\[
%M_g^* + TM_f^* \leq \frac14 \max_{M>0} F(M), \quad M^*\in (M_g^*,+\infty) \cap F^{-1}([4M_g^* + 4TM_f^*,+\infty)), \quad N_* \geq \bigg( \frac{4C_{M^*}}{F(M^*)} \bigg)^2,
%\]
we see that $w(\tau,\bbs x, \bbs y) \geq -\frac12M^*| \bbs x - \bbs y|^2$ holds for any %$M_g \leq M_g^*$ and $M_f \leq M_f^*$. 
$N \geq  N_*$. The conclusion follows by the arbitrariness of $\bbs x, \bbs y \in (\R^d)^N$.
\end{proof}

\begin{rmk} \label{rmksusn}
By using the precise form of the constant $\overline C$ in \eqref{ovCform}, we can obtain the same kind of improvement of semimonotonicity also in a setting where $\sigma$ is not fixed, but may depend on $N$, that is $\sigma = \sigma_N$. In particular, taking  $M_g^* = (12eT)^{-1}$, $M_f^* = (12eT^2)^{-1}$ and $M^* = (2T)^{-1}$ as above, the inequality \eqref{stareq} holds provided that
\[
\frac{e^{e^{c_6(1+\sigma_{N_*}^{-1})}}}{\sqrt{N_*}} \leq \frac 1{12e^T}.
\]
As an example, this is satisfied for $N_*$ large enough and $\sigma_N$ satisfying
\[
\sigma_N \wedge 1 \geq \frac1{\delta \log \log N}
\]
for some (small) $\delta$ depending on $T$, $L_g$, and $L_f$. Note in particular that despite the viscosity $\sigma_N$ may vanish as $N\to \infty$, the semimonotonicity bound $M^*$ remains uniform in $N$. For further discussion on a vanishing viscosity limit for $N \to \infty$, see \Cref{sec_vvl}.
\end{rmk}

As anticipated, such a non-degeneration of semimonotonicity of the solution is the key ingredient for the propagation of all derivative estimates, and of semimonotonicity, over the whole time horizon $[0,T]$. This constitutes the first half of \Cref{ENS_mainthm}. 

\begin{thm}[Estimates on the Nash system with \Ds semimonotone data] \label{ENS_mainthmDS}
Assume that $f$ and $g$ satisfy assumptions~{\bfseries(\hyperlink{MF}{MF})} and {\bfseries(\hyperlink{DS}{DS})}. Let $T>0$. There exist positive constants $M_f^*$ and $M_g^*$ depending on $T$ (in such a way that $M_f^*,M_g^* \to 0$ as $T \to +\infty$ and $M_f^*,M_g^* \to +\infty$ as $T \to 0$), and a natural number $N^*$ depending only on $\sigma$, $T$, $L_f$ and $L_g$, such that if $M_g \leq M_g^*$, $M_f \leq M_f^*$ and $N \geq N^*$ then %there exists a unique 
any solution $u$ to~\eqref{ENS_NS} on $[0,T] \times (\R^d)^N$ satisfies
\[
\sup_{1 \leq i \leq N} \bigg( \sup_{\substack{1\leq j \leq N \\ j \neq i}} \norm{D_j u^i}_\infty +  \bigg\lVert \sum_{\substack{1\leq j\leq N \\ j \neq i}} \big| D^2_{ij} u^j\big|^2 \bigg\rVert_\infty +  \sum_{\substack{1\leq j \leq N \\ j \neq i}}\norm{D(D_j u^i)}_\infty^2 \bigg) \lesssim \frac1N
\]
and
\[
\sup_{1 \leq i \leq N} \norm{D u^i}_\infty + \sup_{1 \leq i \leq N} \sum_{1\leq j \leq N}\norm{D(D_i u^j)}_\infty^2 + \sum_{\substack{1\leq i,j \leq N \\ j\neq i}} \norm{D(D_{ij} u^j)}_\infty^2 \lesssim 1,
\]
where the implied constants depend only on $\sigma$, $T$, $L_f$ and $L_g$. In addition, there exists $M^* > 0$ (depending only on $T$) such that $u$ is $M^*$-\Ds semimonotone on $[0,T]$.
\end{thm}

\begin{proof}
% By \Cref{ENS_thmexsl}, we know that in a left neighbourhood of $T$ a unique solution $u$ exists with bounded derivatives. 
Let $M_f^*$, $M_g^*$, $M^*$ and $N^*$ be given by \Cref{ENS_dscprop}. By assumption~{\bfseries(\hyperlink{DS}{DS})}, $\Dop[u(T,\var)] + M^*\QA \geq \epsilon \QA$ on $(\R^d)^{2N}$, with $\epsilon = M^* - M^*_g > 0$;  that is, by \Cref{ENS_rmkDMeq}, $(D_{ij}^2 u^i(T,\var))_{1 \leq i,j \leq N} + M^* \msfb I \geq \epsilon \msfb I$ on $(\R^d)^N$. Since $D^2_{ij} u^i$ are assumed to be uniformly continuous, %(cf.~\Cref{ENS_thmexsl}) by uniform continuity (with respect to $t$, uniformly with respect to $\bs x$) 
there exists $\tau \in [0,T)$ (a priori dependent of $N$) such that $(D_{ij}^2 u^i(t,\var))_{1 \leq i,j \leq N} + M^* \msfb I \geq \frac\epsilon2 \msfb I$ on $(\R^d)^N$ for all $t \in (\tau, T]$. Therefore,
\[
\call T \defeq \big\{ s \in [0,T) :\ \text{$u$ extends to a solution on $(s, T] \times (\R^d)^N$ and $\Dop[u] > -M^* \QA$ thereon} \big\} \neq \emptyset
\]
and $\tau^* \defeq \inf\call T \in [0,\tau]$. Seeking for a contradiction, suppose that $\tau^* > 0$. Then $\Dop[u] \geq - M^*\QA$ on $[\tau^*, T] \times (\R^d)^{2N}$ and by \Cref{ENS_dscprop} we have $\Dop[u(\tau^*,\var)] \geq - \frac12 M^*\QA$ on $(\R^d)^{2N}$. Repeating the same argument as at the beginning of the proof, there exists $\tau' \in \call T$, $\tau' < \tau^*$, thus contradicting the definition of $\tau^*$. This proves the $\mathscr{D}$-semimonotonicity of $u$ on $[0,T]$, hence all estimates on the derivatives are given by \Cref{ENS_d2b,ENS_d2N,ENS_uijestp,ENS_d2N',ENS_d3b}.
\end{proof}

\section{The \texorpdfstring{\LLs semimonotone}{L-semimonotone} case}\label{sec_LLs}

We prove in this section \Cref{ENS_mainthm} in the \LLs semimonotone case, and for quadratic Hamiltonians $H^i(x^i, p) = \frac12|p|^2$. Our observations in \Cref{ENS_rmkLLtoDM} motivate the following result, which comprehends all estimates on derivatives we have made in the previous section, under the assumption of \LLs semimonotonicity and a lower bound on the diagonal of the Jacobian matrix of the vector of optimal controls. As before, solutions to the Nash system are assumed to be classical, locally $C^1$ in time and $C^4$ in space, with bounded derivatives. Moreover, we require second order derivatives in space to be uniformly continuous.

\begin{prop}[Estimates on the derivatives, assuming \LLs semimonotonicity and a diagonal lower bound] \label{ENS_propestLL} 
Let $\tau \in [0,T)$ and assume that {\bfseries(\hyperlink{MF}{MF})} holds. Suppose that $\LLop[u] \geq -\kappa \QA$ on $(\tau, T] \times (\R^d)^{2N}$ and there exists $\gamma > 0$ (independent of $N$), such that 
\begin{equation} \label{ENS_crucialbound}
\bs\triangle_{u(t,\var)} \geq -\gamma \msfb I \ \ \text{on} \ (\R^d)^N \quad \forall\,t \in (\tau,T].
\end{equation}
Then there exist constants $K_i$, $i \in \{1,\dots, 5\}$, depending only on $\sigma$, $T$, $\kappa$, $\gamma$, $L_g$ and $L_f$, such that the following estimates hold for all $t \in [\tau,T]$:
\begingroup
\allowdisplaybreaks
\begin{gather}
\label{ENS_LLd1}
\sup_{\substack{1\leq i,j \leq N \\ j \neq i}} \norm{D_j u^i(t,\var)}_\infty \leq \frac{K_1}{N}, \\
\label{ENS_LLd2}
\sup_{1\leq i \leq N} \sum_{1\leq j \leq N}\norm{D(D_i u^j)(t,\var)}_\infty^2 \leq K_2, \\
\label{ENS_LLd2'}
\sup_{1\leq i \leq N} \bigg\lVert \sum_{\substack{1\leq j\leq N \\ j \neq i}} \big| D^2_{ij} u^j(t,\var)\big|^2 \bigg\rVert_\infty \leq \frac{K_3}{N} , \\
\label{ENS_LLd2''}
\sup_{1\leq i \leq N} \sum_{\substack{1\leq j \leq N \\ j \neq i}}\norm{D(D_j u^i)(t,\var)}_\infty^2 \leq \frac{K_4}{N} , \\
\label{ENS_LLd3}
\sum_{\substack{1\leq i,j \leq N \\ j\neq i}} \norm{D(D_{ij} u^j)(t,\var)}_\infty^2 \leq K_5.
\end{gather}
\endgroup
More precisely, each $K_i$ is given by the constant $C_i$ appearing in \Cref{ENS_d2b,ENS_d2N,ENS_uijestp,ENS_d2N',ENS_d3b} with $M$ replaced by $\kappa + \gamma$.
\end{prop}

\begin{proof}
As noted in \Cref{ENS_rmkLLtoDM}, we have $\Dop[u] \geq -(\kappa+\gamma)\QA$ on $(\tau, T] \times (\R^d)^{2N}$. Then the estimates follows from \Cref{ENS_d2b,ENS_d2N,ENS_uijestp,ENS_d2N',ENS_d3b}.
\end{proof}

With this piece of information one can show that the pair formed by \LLs semimonotonicity and diagonal lower bound~\eqref{ENS_crucialbound} does not degenerate, just like displacement semiconvexity did not (cf.~\Cref{ENS_dscprop}).

\begin{prop}[Improvement of \LLs semimonotonicity and diagonal lower bound] \label{ENS_simLLprop}
Assume that {\bfseries(\hyperlink{MF}{MF})} holds. Then there exist nonnegative constants $\kappa_f^*$, $\kappa_g^*$, $\kappa^*$ and $\gamma^*$ and a natural number $N_*$, all depending only on $\sigma$, $T$, $L_g$ and $L_f$, such that $\kappa_g^* < \kappa^*$, $L_g < \gamma^*$ and, for any $\tau \in [0,T)$, one has
\[
\begin{rcases}
\LLop[u] \geq -\kappa^* \QA & \text{on}\ (\tau, T] \times (\R^d)^{2N} \\
\bs\triangle_u \geq -\gamma^* \msfb I & \text{on}\ (\tau, T] \times (\R^d)^N \\
\LLop[f] \geq -\kappa_f^*\QA & \text{on} \ (\R^d)^{2N} \\
\LLop[g] \geq -\kappa_g^*\QA & \text{on} \ (\R^d)^{2N} \\
N \geq N_*
\end{rcases}
\quad \implies \quad 
\begin{cases}
\LLop[u(\tau,\var)] \geq - \dfrac12\kappa^* \QA & \text{on}\ (\R^d)^{2N} \vspace{5pt}\\
\bs\triangle_u(\tau,\var) \geq  -\dfrac12\gamma^* \msfb I & \text{on}\ (\R^d)^N.
\end{cases}
\]
\end{prop}

We are going to split the proof of this result in two propositions. We will prove that the \mbox{\LLs semimonotonicity} improves, provided that one has the diagonal lower bound (see \Cref{ENS_LLmprop}). Then, symmetrically, that such a lower bound improves, provided that one knows that $u$ is \LLs semimonotone (see \Cref{ENS_cruboprop}).

%\ros{Siamo sicuri che serva davvero il paragrafo seguente?} Finally, one crucial fact will allow to merge this two non-degeneration results: indeed, the former holds for a \LLs semimonotonicity constant $\kappa$ suitably chosen with respect to the lower bound $\gamma$, while the latter allows for a choice of $\gamma$ which is independent of $\kappa$, thus avoiding a circular argument when proving that relaxed \LLs monotonicity and the diagonal lower bound contemporarily do not degenerate. 

\begin{prop}[Improvement of \LLs semimonotonicity] \label{ENS_LLmprop}
Assume that {\bfseries(\hyperlink{MF}{MF})} holds. Let $\gamma > 0$. Then there exist  a nonnegative constant $\kappa^*$ depending on $T$ (but not on $\gamma$), and constants $\kappa_f^*$, $\kappa_g^*$ (with $\kappa_g^* < \kappa$) and a natural number $N_{*}'$, all depending only on $\sigma$, $T$, $\gamma$, $L_g$ and $L_f$, such that for any $\tau \in [0,T)$ one has
\[
\begin{rcases}
\bs\triangle_{u(t,\cdot)} \geq -\gamma \msfb I & \text{on}\ (\tau, T] \times (\R^d)^{N} \\
\LLop[u] \geq -\kappa^* \QA & \text{on}\ (\tau, T] \times (\R^d)^{2N} \\
\LLop[f] \geq -\kappa_f^*\QA & \text{on} \ (\R^d)^{2N} \\
\LLop[g] \geq -\kappa_g^*\QA & \text{on} \ (\R^d)^{2N} \\
N \geq N_{*}'
\end{rcases}
\quad \implies \quad 
\LLop[u(\tau,\var)] \geq - \frac12\kappa^* \QA \ \ \text{on}\ (\R^d)^{2N}.
\]
\end{prop}

Looking at the proof below, one can in fact choose $\kappa_g^* = (12T e^{2\gamma T + 1})^{-1}$, $\kappa_f^* = (12 T^2 e^{2\gamma T + 1})^{-1}$ and $\kappa^* = (2T)^{-1}$, therefore $\kappa^*_h$ is actually independent of $L_g$ and $L_f$.%  and $\kappa^*$ depends on $T$ only.

\begin{proof}[Proof of \Cref{ENS_LLmprop}]
Let $w(t, \bs x, \bs y) \defeq \LLop[u(t,\var)](\bs x, \bs y)$. We want to use formula~\eqref{ENS_itow} where $\bar t = \tau$, $h = w$ and $\bs Z = (\bs X, \bs Y)$ is a synchronous coupling of two solutions to \eqref{ENS_NcSDEs} at equilibrium; that is, formula~\eqref{ENS_itow} with $L$ given by \eqref{ENS_Lsync}.
Omitting the dependence on $t$, we have
\[
Lw(\bs x, \bs y) = \tr\!\big((\sigma\msfb I + \beta \msfb J)\LLop[D^2u](\bs x, \bs y)\big) - \sum_{1 \leq j \leq N} \! \big( D_j u^j(\bs x) \cdot D_{x^j} + D_j u^j(\bs y) \cdot D_{y^j} \big)\LLop[u](\bs x, \bs y),
\]
and, according to \eqref{ENS_NS},
\[
\begin{multlined}[.95\displaywidth]
\de_t w(\bs x, \bs y) = - \tr\!\big((\sigma\msfb I + \beta \msfb J) \LLop[D^2u](\bs x, \bs y)\big) + \frac12 \LLop[(|D_i u^i|^2)_{1 \leq i \leq N}](\bs x, \bs y) \\
+ \sum_{\substack{1 \leq j \leq N \\ j \neq i}} \LLop[D_j u^j \cdot D_j u](\bs x, \bs y) 
- \LLop[f](\bs x, \bs y).
\end{multlined}
\]
Note that
\begin{multline*}
\sum_{1 \leq j \leq N} \big( D_j u^j(\bs x) \cdot D_{x^j} + D_j u^j(\bs y) \cdot D_{y^j} \big)\LLop[u](\bs x, \bs y) \\
= \sum_{1 \leq i,j \leq N} D_j u^j(\bs x) D_j u^i(\bs x) - \sum_{\substack{1 \leq i,j \leq N \\ j \neq i}} D_j u^j(\bs x) D_j u^i(\bs x^{-i}, y^i) - \sum_{1 \leq i \leq N} D_i u^i(\bs x) D_i u^i(\bs y^{-i}, x^i) \\
- \sum_{1 \leq i \leq N} D_i u^i(\bs y) D_i u^i(\bs x^{-i}, y^i) - \sum_{\substack{1 \leq i,j \leq N \\ j \neq i}} D_j u^j(\bs y) D_j u^i(\bs y^{-i}, x^i) + \sum_{1 \leq i,j \leq N} D_j u^j(\bs y) D_j u^i(\bs y) 
\end{multline*}
and
\begin{multline*}
\sum_{\substack{1 \leq j \leq N \\ j \neq i}} \LLop[D_j u^j \cdot D_j u](\bs x, \bs y) \\[-6pt]
= \sum_{\substack{1 \leq i,j \leq N \\ j \neq i}} D_j u^j(\bs x) D_j u^i(\bs x) -\sum_{\substack{1 \leq i,j \leq N \\ j \neq i}} D_j u^j(\bs x^{-i}, y^i) D_j u^i(\bs x^{-i}, y^i) \\
\quad\ - \sum_{\substack{1 \leq i,j \leq N \\ j \neq i}} D_j u^j(\bs y^{-i}, x^i) D_j u^i(\bs y^{-i}, x^i) + \sum_{\substack{1 \leq i,j \leq N \\ j \neq i}} D_j u^j(\bs y) D_j u^i(\bs y),
\end{multline*}
where all products between vectors are understood to be Euclidean inner products.
Then easy computations show that
\begin{equation} \label{eqwnnterms}
\begin{multlined}[.9\displaywidth]
(\de_t + L)w(\bs x, \bs y) \vspace{3pt}
= \sum_{\substack{1 \leq i,j \leq N \\ j \neq i}} D_j u^j\big|_{(\bs x^{-i}, y^i)}^{\bs x} \cdot D_j u^i(\bs x^{-i}, y^i) + \sum_{\substack{1 \leq i,j \leq N \\ j \neq i}} D_j u^j\big|_{(\bs y^{-i}, x^i)}^{\bs y} \cdot D_j u^i(\bs y^{-i}, x^i) \\
- \sum_{1 \leq i \leq N} \big| D_i u^i(\bs x) - D_i u^i(\bs y^{-i}, x^i) \big|^2 - \sum_{1 \leq i \leq N} \big| D_i u^i(\bs y) - D_i u^i(\bs x^{-i}, y^i) \big|^2 - \LLop[f](\bs x, \bs y).
\end{multlined}
\end{equation}
Exploiting the fundamental theorem of calculus,
\[ \begin{multlined}[.95\displaywidth]
\sum_{\substack{1 \leq i,j \leq N \\ j \neq i}} D_j u^j\big|_{(\bs x^{-i}, y^i)}^{\bs x} \cdot D_j u^i(\bs x^{-i}, y^i) + \sum_{\substack{1 \leq i,j \leq N \\ j \neq i}} D_j u^j\big|_{(\bs y^{-i}, x^i)}^{\bs y} \cdot D_j u^i(\bs y^{-i}, x^i) \\
\begin{split}
&= \sum_{\substack{1 \leq i,j \leq N \\ j \neq i}} \big( D_j u^j\big|_{(\bs x^{-i}, y^i)}^{\bs x} + D_j u^j\big|_{(\bs y^{-i}, x^i)}^{\bs y} \big) \cdot D_j u^i(\bs x^{-i}, y^i) + \sum_{\substack{1 \leq i,j \leq N \\ j \neq i}} D_j u^j\big|_{(\bs y^{-i}, x^i)}^{\bs y} \cdot D_j u^i\big|_{(\bs x^{-i}, y^i)}^{(\bs y^{-i}, x^i)} \\
&
\begin{multlined}[.95\displaywidth]
= \sum_{\substack{1 \leq i,k \leq N \\ k \neq i}} \int_0^1 \int_0^1 \sum_{\substack{1 \leq j \leq N \\ j \neq i}} D^3_{ijk} u^j(\bs z^{-i}_{s'}, z^i_{s}) D_j u^i(\bs x^{-i}, y^i)(x^i-y^i)(x^k-y^k)\,\di s\,\di s' \\[-3pt]
+ \sum_{1 \leq i,k \leq N} \int_0^1 \int_0^1 \sum_{\substack{1 \leq j \leq N \\ j \neq i}} D^2_{ij} u^j(\bs y^{-i}, z^i_{s}) D^2_{jk} u^i(\bs z^{-i}_s, z^i_{1-s}) (x^i - y^i) (x^k - y^k)\, \di s\,\di s',
\end{multlined}\\[-\baselineskip]
\end{split}
\end{multlined}
\]
where products between tensors are understood in the sense of tensor contraction as discussed in the proof of \Cref{ENS_d2N}. By the Cauchy--Schwarz inequality and assumption~{\bfseries(\hyperlink{LS}{LS})},
\begin{multline}\label{ENS_aboveest}
(\de_t + L)w(\bs x, \bs y) \leq \Bigg( {\kappa_f} + \bigg( \sum_{1 \leq i \leq N} \bigg\lVert \sum_{\substack{1 \leq j,k \leq N \\ j,k \neq i}} \big| D^3_{ijk} u^j \big|^2 \bigg\rVert_\infty \bigg\rVert \sum_{\substack{1 \leq j \leq N \\ j \neq i}} \big| D_j u^i \big|^2 \bigg\rVert_\infty \bigg)^{\frac12} \\ 
+ \bigg( \sum_{1 \leq i \leq N} \bigg\lVert \sum_{\substack{1 \leq j \leq N \\ j \neq i}} \big| D^2_{ij} u^j \big|^2 \bigg\rVert_\infty \bigg\rVert \sum_{\substack{1 \leq j \leq N \\ j \neq i}} \big| D(D_j u^i) \big|^2 \bigg\rVert_\infty \bigg)^{\frac12} \Bigg) | \bs x - \bs y |^2.
\end{multline}
We can now proceed with the proof as we did for the one of \Cref{ENS_dscprop}. By formula~\eqref{ENS_itow}, estimate~\eqref{ENS_aboveest}, \Cref{ENS_propestLL} %assumption~{\bfseries(\hyperlink{LS}{LS})} 
and estimate~\eqref{ENS_knowest}, we obtain
\[
w(\tau, \bbs x, \bbs y) \geq - e^{2(\kappa^*+\gamma)T}\Bigl( \kappa^*_g + T \kappa^*_f + \frac{K}{\sqrt{N}\,}\Bigr) | \bbs x - \bbs y|^2,
\]
where the constant $K$ depend only on $\sigma$, $T$, $\kappa^* +\gamma$, $L_g$ and $L_f$. % (in fact, with the notation used in the proof of \Cref{ENS_dscprop}, $K_\kappa = C_{\kappa+\gamma}$ and $K_\kappa = C_{\kappa+\gamma}'$).
Then we conclude by arguing as in the proof of \Cref{ENS_dscprop}: pick $\kappa^* = (2T)^{-1}$, $\kappa^*_g$ and $\kappa^*_f$ small enough and $N_{*}'$ large enough so that
\[
\kappa^*_g + T \kappa^*_f + \frac{K}{\sqrt{N_{*}'}\,} \leq \frac{\kappa^*}{2 e^{2(\kappa^*+\gamma)T}} = \frac{1}{4T e^{2\gamma T + 1}}. \qedhere
\]
\end{proof}

\begin{prop}[Improvement of the diagonal lower bound] \label{ENS_cruboprop}
%Let $t \in [0,T)$ and suppose that there exists $\kappa > 0$ such that $\LLop[u] \geq -\kappa \QA$ on $(\tau,T] \times (\R^d)^N$. Then there exists a constant $\gamma^*$, depending only on $T$, $L_g$ and $L_f$ \ros{(and not on $\kappa$)}, and a natural number $N_*$, depending only on $T$, $\kappa$, $L_g$ and $L_f$, such that
Assume that {\bfseries(\hyperlink{MF}{MF})} holds. Let $\kappa > 0$. Then there exists a constant $\gamma^*$ depending only on $\sigma$, $T$, $L_g$ and $L_f$ (and not on $\kappa$) and a natural number $N_{*}''$ depending only on $\sigma$, $T$, $\kappa$, $L_g$ and $L_f$, such that  for any $\tau \in [0,T)$ one has
\[
\begin{rcases}
\LLop[u] \geq -\kappa \QA& \text{on}\ (\tau,T] \times (\R^d)^N \\
\bs\triangle_u \geq -\gamma^* \msfb I & \text{on}\ (\tau, T] \times (\R^d)^N \\
N \geq N_{*}''
\end{rcases}
\quad \implies \quad \bs\triangle_{u(\tau,\var)} \geq -\frac12\gamma^* \msfb I \ \ \text{on} \ (\R^d)^N.
\]
\end{prop}

The explicit constant $\gamma^*$ can be found below, see \eqref{Lambdastar}.

\begin{proof}[Proof of \Cref{ENS_cruboprop}] We are going to use an argument ``along optimal trajectories'', but without doubling variables/coupling.
Consider $v = D_i u^i$; it satisfies the $\R^d$-valued equation\footnote{We point out that we use the natural notation $\tr((\sigma\msfb I + \beta\msfb J)D^2 v)$ to denote the vector with coordinates
\[
\sum_{1 \leq j,k \leq N,\ 1 \leq a,b \leq d} (\sigma\msfb I + \beta\msfb J)_{dj+a-d,dk+b-d} D^3_{x^{kb}x^{ja}x^{i\ell}} v^i, \quad \ell = 1,\dots,d.
\]
}
\begin{equation} \label{ENS_eqDiui}
-\de_t v - \tr((\sigma\msfb I + \beta\msfb J)D^2 v) + \sum_{1 \leq j \leq N} D_j u^j \cdot D_j v = D_i f^i - \sum_{\substack{ 1 \leq j \leq N \\ j \neq i}} D^2_{ij} u^j D_j u^i.
\end{equation}
Fix $\bs x \in (\R^d)^N$, $s \in [\tau, T)$ and let $\bs X$ solve \eqref{ENS_NcSDEs}, with initial condition $\bs X_s = \bs x$ and optimal control $\alpha^i(t,\bs x) = -D_i u^i(t,\bs x)$. By formula~\eqref{ENS_itow}, 
\[
D_i u^i(s, \bs x) = \bb E \biggl[ D_i g^i(\bs X_T) + \int_s^T D_i f^i(\bs X) - \int_s^T \sum_{\substack{ 1 \leq j \leq N \\ j \neq i}} D^2_{ij} u^j(t,\bs X_t) D_j u^i(t,\bs X_t) \,\di t \biggr],
\]
hence by assumption~{\bfseries(\hyperlink{MF}{MF})}, the Cauchy--Schwarz inequality, estimates~\eqref{ENS_LLd1} and \eqref{ENS_LLd2'} and the arbitrariness of $(s,\bs x) \in [\tau, T) \times (\R^d)^N$,
\begin{equation} \label{ENS_estDiui}
\norm{D_i u^i(t,\cdot)}_\infty \leq L_g + TL_f + \frac KN \quad \forall\, t \in [\tau, T],
\end{equation}
where the constant $K$ depends only on $\sigma$, $T$, $\kappa$, $\gamma^*$ (that will be chosen below), $L_g$ and $L_f$.

Letting now $s = \tau$ and consider $w = \frac12|D_i u^i|^2$, which satisfies
\[ \begin{multlined}[.95\displaywidth]
-\de_t w - \tr((\sigma\msfb I + \beta \msfb J)D^2w) + \sum_{1 \leq j \leq N} D_j u^j \cdot D_j w + \sigma\! \sum_{1 \leq j \leq N} |D_{ij}^2 u^i|^2 + \beta\bigg| \sum_{1 \leq j \leq N} D_{ij}^2 u^i \bigg|^2 \\[2pt]
= D_i u^i \cdot D_i f^i - \sum_{\substack{ 1 \leq j \leq N \\ j \neq i}} D^2_{ij} u^j D_j u^i \cdot D_i u^i;
\end{multlined}\]
using again formula~\eqref{ENS_itow}, we have
\[
\begin{split}
&\frac12 |D_i u^i(\tau, \bs x)|^2 + \sigma \bb E \biggl[\, \int_\tau^T \sum_{1 \leq j \leq N}  |D^2_{ij} u^i(t,\bs X_t)|^2\,\di t \biggr] + \beta \bb E \biggl[\, \int_\tau^T \bigg| \sum_{1 \leq j \leq N} D_{ij}^2 u^i(t,\bs X_t) \bigg|^2 \di t \biggr]\\
&\qquad \begin{multlined}[.95\displaywidth]
= \bb E \biggl[\,  \frac12|D_i g^i(\bs X_T)|^2 + \int_\tau^T D_i u^i(t,\bs X_t) \cdot D_i f^i(\bs X_t)\,\di t 
\\[-5pt]
 - \int_\tau^T \sum_{\substack{ 1 \leq j \leq N \\ j \neq i}} D^2_{ij} u^j(t,\bs X_t) D_j u^i(t,\bs X_t) \cdot D_i u^i(t,\bs X_t)\, \di t \biggr].
\end{multlined}
\end{split}\]
Therefore, using assumption~{\bfseries(\hyperlink{MF}{MF})}, the Cauchy--Schwarz inequality and estimates~\eqref{ENS_LLd1}, \eqref{ENS_LLd2'} and \eqref{ENS_estDiui}, one obtains in particular
\begin{equation} \label{ENS_estdiiui}
\bb E \biggl[ \,\int_\tau^T |D^2_{ii} u^i(t, \bs X_t)|^2\,\di t \biggr] \leq \frac{L_g^2}{2\sigma} + \frac{TL_f}\sigma(L_g + TL_f) + \frac{K'}{N},
\end{equation}
where the constant $K'$ depends only on $\sigma$, $T$, $\kappa$, $\gamma^*$, $L_g$ and $L_f$. \
% ros{We now require that $N \geq  K'$, so that the previous estimate depends only on $T$, $L_g$ and $L_f$.}

As a last step, $\msf V = D_{ii}^2u^i$ solves the following $\Sym(d)$-valued equation
\begin{equation*} %\label{ENS_eqdii}
\begin{multlined}[t][.95\displaywidth]
-\de_t \msf V - \tr((\sigma\msfb I + \beta \msfb J)D^2 \msf V) + \sum_{1 \leq j \leq N} D_j u^j D_{j} \msf V + \msf V^2 \\
= D_{ii}^2 f^i - 2 \sum_{\substack{1 \leq j \leq N \\ j \neq i}} D_{ij}^2 u^i D_{ji}^2 u^j - \sum_{\substack{1 \leq j \leq N \\ j \neq i}} D_{iij}^3 u^j D_{j} u^i ;
\end{multlined}
\end{equation*}
thus formula~\eqref{ENS_itow} gives
\begin{equation*} %\label{ENS_dualDiiui}
\begin{multlined}[t][.95\displaywidth]
D_{ii}^2 u^i(\tau, \bs x) + \bb E \biggl[\, \int_\tau^T (D^2_{ii} u^i(t,\bs X_t))^2 \,\di t -  D_{ii} g^i(\bs X_T) - \int_\tau^T D_{ii}^2 f^i(\bs X) \biggr] \\
= \bb E\biggl[ \, \int_\tau^T \sum_{\substack{1 \leq j \leq N \\ j \neq i}} \!\big( 2 D_{ij}^2 u^i D_{ji}^2 u^j + D_{iij}^3 u^j D_{j} u^i \big)(t,\bs X_t) \,\di t \biggr].
\end{multlined}
\end{equation*}
By assumption~{\bfseries(\hyperlink{MF}{MF})}, the Cauchy--Schwarz inequality, estimates in \Cref{ENS_propestLL} and estimate~\eqref{ENS_estdiiui} we get
\[
D_{ii}^2 u^i(\tau, \bs x) \geq - \Bigl(L_g + L_f + \frac{L_g^2}{2\sigma} + \frac{TL_f}\sigma(L_g + TL_f) + \frac{K'}{N} + \frac{K''}{\sqrt{N}\,}\Bigr) \msf I_d,
\]
where the constant $K''$ depends only on $\sigma$, $T$, $\kappa$, $\gamma^*$, $L_g$ and $L_f$. At this point one sees that if 
\begin{equation}\label{Lambdastar}
\gamma^* = 4 \Bigl(L_g + L_f + \frac{L_g^2}{2\sigma} + \frac{TL_f}\sigma(L_g + TL_f) \Bigr)
\end{equation}
 then for any $N$ large enough so that $K' N^{-1} + {K''}N^{-\frac12} \leq \frac14 \gamma^*$ one has $D_{ii} u^i(\tau, \bs x) \geq -\frac12 \gamma^*\msf I_d$, uniformly with respect to $i \in \{1,\dots,N\}$. The conclusion follows from the arbitrariness of $\bs x  \in (\R^d)^N$.
\end{proof}

For the benefit of the reader, we make explicit how \Cref{ENS_simLLprop} can be now obtained.

\begin{proof}[Proof of \Cref{ENS_simLLprop}] Let $\kappa^*$ be as in \Cref{ENS_LLmprop} and $\gamma^*$ be as in \Cref{ENS_cruboprop} (note that these two do not depend on each other). Then,
\Cref{ENS_LLmprop} applies with $\gamma = \gamma^*$, yielding the improvement of \LLs semimonotonicity for $N \geq  N_{*}'$. On the other hand, \Cref{ENS_cruboprop} applies with $\kappa = \kappa^*$, so that the diagonal lower bound improves for $N \geq  N_{*}''$. Taking $N_* = \max\{ N_*' , N_{*}''\}$ gives the assertion.
\end{proof}

\begin{thm}[Estimates on the Nash system with \LLs semimonotone data] \label{ENS_mainthmLL}
Assume that $f$ and $g$ satisfy assumptions~{\bfseries(\hyperlink{MF}{MF})} and {\bfseries(\hyperlink{LS}{LS})}. Let $T>0$. There exist nonnegative constants $\kappa_f^*$ and $\kappa_g^*$ and a natural number $N^*$, all depending only on $\sigma$, $T$, $L_f$ and $L_g$ (in such a way that $\kappa_f^*,\kappa_g^* \to 0$ as $T \to +\infty$ and $\kappa_f^*,\kappa_g^* \to +\infty$ as $T \to 0$), such that if $\kappa_g \leq \kappa_g^*$, $\kappa_f \leq \kappa_f^*$ and $N \geq N^*$ then % there exists a unique 
any solution $u$ to~\eqref{ENS_NS} on $[0,T] \times (\R^d)^N$ satisfies
\[
\sup_{1 \leq i \leq N} \bigg( \sup_{\substack{1 \leq j \leq N \\ j \neq i}} \norm{D_j u^i(t,\cdot)}_\infty +  \bigg\lVert \sum_{\substack{1\leq j\leq N \\ j \neq i}} \big| D^2_{ij} u^j\big|^2 \bigg\rVert_\infty +  \sum_{\substack{1\leq j \leq N \\ j \neq i}}\norm{D(D_j u^i)}_\infty^2 \bigg) \lesssim \frac1N
\]
and
\[
\sup_{1 \leq i \leq N} \norm{D u^i}_\infty^2 + \sup_{1 \leq i \leq N} \sum_{1\leq j \leq N}\norm{D(D_i u^j)}_\infty^2 + \sum_{\substack{1\leq i,j \leq N \\ j\neq i}} \norm{D(D_{ij} u^j)}_\infty^2 \lesssim 1,
\]
where the implied constants depend only on $\sigma$, $T$, $L_f$ and $L_g$. In addition, there exists $\kappa^* > 0$ (depending only on $T$) such that $u$ is $\kappa^*$-\LLs semimonotone on $[0,T]$.
\end{thm}

\begin{proof} The argument is very much analogous to the one of \Cref{ENS_mainthmDS}, so we omit it. \end{proof}

%Let us just stress that it is crucial to use the previous two Proposition in the correct order. First, let $\gamma^*$ and $N_*$ be as in \Cref{ENS_cruboprop}, that applies with $\kappa = L_g + 1$; note that $\gamma^*$ depends on $T$, $L_g$, $L_f$ only. 
%
%by assumption~{\bfseries(\hyperlink{LS}{LS})}, $\LLop[u(T,\cdot)] + k^*\QA \geq \epsilon \QA$ on $(\R^d)^{2N}$
%
%
%\bigskip
%
%Let $M_f^*$, $M_g^*$, $M^*$ and $N^*$ be given by \Cref{ENS_dscprop}. By assumption~{\bfseries(\hyperlink{DS}{DS})}, $\Dop[u(\ros{T},\cdot)] + M^*\QA \geq \epsilon \QA$ on $(\R^d)^{2N}$, with $\epsilon = M^* - \ros{M^*_g} > 0$; that is, by \Cref{ENS_rmkDMeq}, $(D_{ij}^2 u^i(\ros{T},\cdot))_{1 \leq i,j \leq N} + M^* \msfb I \geq \epsilon \msfb I$ on $(\R^d)^N$. Since $D(D^2_{ij} u^i)$ is bounded for all $i,j$, by uniform continuity there exists $\tau \in [0,T)$ (a priori dependent of $N$) such that $(D_{ij}^2 u^i(t,\cdot))_{1 \leq i,j \leq N} + M^* \msfb I \geq \frac\epsilon2 \msfb I$ on $(\R^d)^N$ for all $t \in (\tau, T]$
%\end{proof}

% \part{Nonsymmetric Mean Field Games}

\section{Analysis of the \texorpdfstring{$N\to\infty$}{large population} limit} \label{sec_lpl}

This section is devoted to the analysis of the convergence problem for $u^i = u_N^i$ as $N\to\infty$, in the case of quadratic Hamiltonians. Our main goal is to prove \Cref{ENS_thmMFsumm}. We will suppose that the assumption \hyperlink{MF}{{\bfseries{(MF)}}} of Mean-Field-like interactions, and either \Ds semimonotonicity  \hyperlink{DS}{{\bfseries{(DS)}}} or \LLs semimonotonicity \hyperlink{LS}{{\bfseries{(LS)}}} hold in such a way to guarantee the estimates obtained in the first part of the work; for the benefit of the following discussion, we explicitly recall that the estimates on the Nash system we obtained imply 
\begin{subequations}
\begin{gather}
\sup_{1 \leq i \leq N} \norm{D_i u^i}_\infty \lesssim 1, \qquad \sup_{1 \leq i \leq N} \sup_{\substack{1 \leq j \leq N \\ j \neq i}} \norm{D_j u^i}_\infty \lesssim \frac1N, \label{ENS_recapE1}\\
\sup_{1 \leq i \leq N} \norm{D^2_{ii} u^i}_\infty \lesssim 1, \qquad \sup_{1 \leq i \leq N} \,\biggl\lVert \sum_{\substack{1\leq j \leq N \\ j \neq i}} |D_j D_i u^i|^2 \biggr\rVert_\infty \lesssim \frac1N. \label{ENS_recapE2}
\end{gather}
\end{subequations}

To be sure that, for any fixed $N$, solutions to the Nash system exist, we assume that $f^i$ and $g^i$ are of class $C^{2+\alpha}((\R^d)^N)$ and $C^{4+\alpha}((\R^d)^N)$ respectively, for every $i$ and $N$. By \Cref{ENS_thmexbdd}, solutions $u^i$ indeed exists and are unique, and belong to the class $C^{2+\alpha/2, 4+\alpha}$. We have then enough regularity to apply the a priori estimates obtained in first part.

In addition, we require now the symmetry assumption \hyperlink{S}{{\bfseries{(S)}}} on the data $f^i$, $g^i$, that need to be of the form $f_N(\lambda_N^i, \cdots)$, $g_N(\lambda^i_N, \cdots)$, where $f_N$, $g_N$ are Lipschitz continuous with respect to $\lambda \in \Lambda \defeq [0,1]$ (assumption \hyperlink{LP}{{\bfseries{(LP)}}}). In this framework, we may focus our attention on $\frk u_N(\bs\lambda,t,\bs x)= u^1_{\bs\lambda}(t,\bs x)$, namely on the first solution of the Nash system with parameter vector $\bs \lambda= (\lambda^1, \ldots, \lambda^N)$.

We begin by showing that assumption \hyperlink{LP}{{\bfseries(LP)}} reflects in a Lipschitz dependence of $\frk u^N$ and $D_x \frk u^N$ (recall that $D_x \frk u^N = D_{x^1} \frk u^N$ as we are writing $\bs x = (x, \hat{\bs x})$) with respect to $\bs\lambda = (\lambda,\hbs\lambda) \in \Lambda \times \Lambda^{N-1} = \Lambda^N$.

\begin{lem} \label{ENS_contlcontr}
For any $\bs\lambda,\bs\lambda' \in \Lambda^N$, one has
\[
\norm{\frk u_N({\bs\lambda},\var) - \frk u_N({\bs\lambda'},\var)}_{\infty} + \norm{D_x \frk u_N({\bs\lambda},\var) - D_x \frk u_N({\bs\lambda'},\var)}_{\infty} \lesssim |\lambda - \lambda'| + \frac{|\hbs\lambda-\hbs\lambda{}'|}N,
\]
where the implied constant depends only on $T$, $L_g$, $L_f$ and $L_\Lambda$.
\end{lem}

\begin{proof}
All implied constants appearing in this proof depend only on $T$, $L_g$ and $L_f$. Let $\bs\lambda, \bs\lambda' \in \Lambda^N$. Consider $w^i \defeq u^i_{\bs\lambda} - u^i_{\bs\lambda'}$ that solves the following equation in $(0,T) \times (\R^d)^N$:
\begin{equation} \label{ENS_eqwlambdadiff}
\begin{split}
-\de_t w^i - \Delta w^i + \sum_{1 \leq j \leq N} D_j u^j_{\bs\lambda} \cdot D_j w^i
 &= f_N\big|_{(\bs\lambda',\var)}^{(\bs\lambda,\var)} - \sum_{\substack{1 \leq j \leq N \\ j \neq i}} D_j w^j \cdot D_j u^i_{\bs\lambda'} + \frac12 | D_iw^i |^2 \\
&\lesssim |\lambda^i - \lambda'^i| + \frac{1}{N} \sum_{\substack{1 \leq j \leq N \\ j \neq i}} |D_jw^j| + |D_iw^i|,
\end{split}
\end{equation}
where the inequality is obtained by exploiting assumption~\hyperlink{LP}{\bfseries(LP)} and estimates~\eqref{ENS_recapE1}. By \Cref{ENS_estv} we deduce that, for any $\tau \in [0,T]$,
\begin{equation} \label{ENS_lemLPeq1}
\sup_{t \in [\tau,T]} \norm{D_i w^i(t,\cdot)}_\infty \lesssim (T-\tau)^{\frac12} \Bigl( |\lambda^i - \lambda'^i| + \frac1N \sum_{\substack{1 \leq j \leq N \\ j \neq i}} \sup_{t \in [\tau,T]} \norm{D_j w^j(t,\cdot)}_\infty + \sup_{t \in [\tau,T]} \norm{D_i w^i(t,\cdot)}_\infty \Bigr).
\end{equation}
Since by summing over $i$ we get
\[
\sum_{1 \leq i \leq N} \sup_{t \in [\tau,T]} \norm{D_i w^i(t,\cdot)}_\infty \lesssim (T-\tau)^{\frac12} \Bigl( |\bs\lambda - \bs\lambda'| +  \sum_{1 \leq i \leq N} \sup_{t \in [\tau,T]} \norm{D_i w^i(t,\cdot)}_\infty \Bigr),
\]
possibly by iterating this estimate a finite number of times we obtain
\[
\sum_{1 \leq i \leq N} \sup_{t \in [0,T]} \norm{D_i w^i(t,\cdot)}_\infty \lesssim |\bs\lambda - \bs\lambda'|;
\]
then plugging this back into \eqref{ENS_lemLPeq1} (and again possibly iterating the estimate) we have
\[
\sup_{t \in [0,T]} \norm{D_i w^i(t,\cdot)}_\infty \lesssim |\lambda^i - \lambda'^i| + \frac{|\bs\lambda - \bs\lambda'|}N,
\]
which is equivalent to $\norm{D_x \frk u_N({\bs\lambda},\var) - D_x \frk u_N({\bbs\lambda},\var)}_{\infty} \lesssim |\lambda - \lambda'| + N^{-1}|\bs\lambda-\hbs\lambda{}'|$. Now by the maximum principle one deduces that $\norm{w^1}_{\infty} \lesssim |\lambda - \lambda'| + N^{-1}|\bs\lambda-\hbs\lambda{}'|$ as well, thus concluding the proof.
\end{proof}

\begin{rmk} \label{ENS_deldiui}
As a consequence of \Cref{ENS_contlcontr}, by Rademacher's theorem $\bs\lambda \mapsto D_x \frk u^N(\bs\lambda,t,\bs x)$ is $\scr L^N$-a.e.\ differentiable on $\Lambda^N$ for each $(t,\bs x) \in [0,T] \times (\R^d)^N$, and we have $\sup_{1 < j \leq N} \norm{\de_{\lambda^j} D_x \frk u}_\infty \lesssim N^{-1}$,
where the implied constant depends only on $T$, $L_g$, $L_f$ and $L_\Lambda$. The same can be said for $\frk u$ (clearly, with differentiability everywhere).
\end{rmk}

The following result will be instrumental in our study of the limit of the (solution to the) Nash system as $N \to \infty$; it adapts \cite[Theorem~2.1]{cardanotes}. 

\begin{lem} \label{ENS_lemcarconv}
Let $(\Gamma,d)$ be a metric space and let $(E,\norm{\var})$ be a finite-dimensional normed vector space. Let $(h_N)_{N \in \N}$ be a sequence of functions such that $h_N \colon \Gamma \times E^N \to \R$ is symmetric on $E^N$ and there exist $C > 0$ and concave moduli of continuity $\omega_\Gamma$, $\omega_p$, all independent of $N$, such that $\norm{h_N}_\infty \leq C$ and, for some $p \geq 1$,
\begin{equation} \label{ENS_lipW1cond}
|h_N(\gamma, \bs v) - h_N(\gamma', \bs v')| \leq \omega_\Gamma(d(\gamma,\gamma'))+ \omega_p(W_p(m_{\bs v}, m_{\bs v'})) \quad \forall\, \gamma,\gamma' \in \Gamma, \ \forall\, \bs v, \bs v' \in E^N.
\end{equation}
Then there exists $h \in C^0(\Gamma \times \Pc_p(E))$ such that, up to a subsequence, for any compact sets $H \subset \Gamma$ and $\call K \subset \call P_p(E)$,
\[
\lim_{N \to \infty} \sup_{\gamma \in H} \sup_{\bs v \in E^N :\, m_{\bs v} \in \call K} |h_N(\gamma,\bs v) - h(\gamma, m_{\bs v})| = 0.
\]
Furthermore, $h$ has the same moduli of continuity with respect to $\gamma \in \Gamma$ and $m \in \call P_p(E)$.
\end{lem}

\begin{proof}
Define $\frk h_{N} \colon \Gamma \times \Pc_p(E) \to \R$ by
\[
\frk h_{N}(\gamma,m) \defeq \inf_{\bs v \in E^N} \big\{ h_N(\gamma, \bs v) + \omega_p(W_p(m_{\bs v},m)) \big\}.
\]
This is well-defined since $h_N$ is bounded. By \eqref{ENS_lipW1cond}, $\frk h_N(\gamma, m_{\bs v}) = h_N(\gamma,\bs v)$. Given $\gamma, \gamma' \in \Gamma$, $m,m' \in \Pc_p(E)$ and $\bs v \in E^N$ which is $\epsilon$-optimal in the definition of $\frk h_N(\gamma',m')$,
\[\begin{split}
\frk h_N(\gamma,m) &\leq h_N(\gamma, \bs v) + \omega_p(W_p(m_{\bs v},m)) \\ 
&\leq \frk h_N(\gamma', m') + \omega_\Gamma(d(\gamma, \gamma')) + \omega_p( W_p(m_{\bs v},m)) - \omega_p(W_p(m_{\bs v},m')) + \epsilon \\
&\leq \frk h_N(\gamma',m')  + \omega_\Gamma(d(\gamma, \gamma')) + \omega_p(W_p(m,m')) + \epsilon.
\end{split}
\]
We have used that $\omega_p$ is subadditive because it is concave and $\omega_p(0) = 0$.
Therefore, $(\frk h_N)_{N \in \N}$ is a sequence of equicontinuous and equibounded functions, so a generalisation of the Ascoli--Arzelà theorem guarantees the existence of $h \in C^0(\Gamma \times \Pc_p(E))$ such that, up to a subsequence, $\frk h_N \to h$ in the topology of compact convergence, that is our assertion.
% Finally, note that if $K \subset E$ is such that $\sup_{K} \norm{\var}^{r} = R$ for some integer $r > p$, then
%\[
%M_{r}(m_{\bs v}) \defeq \int_E \norm{w}^r \, m_{\bs v}(\di w) \leq R \quad \forall\,\bs v \in K^N;
%\]
%hence, for any compact $H \subset \Gamma$,
%\[
%\sup_{\gamma \in H,\, \bs v \in K^N} |h_N(\gamma,\bs v) - h(\gamma, m_{\bs v})| \leq \sup_{\gamma \in H,\, m \in R\call K_r} |\frk h_N(\gamma, m) - h(\gamma, m) | \to 0,
%\]
%using that $\call K_r \defeq \{ m \in \call P(E) :\ M_r(m) \leq 1 \}$ is compact in $\Pc_p(E)$; cf., e.g., \cite[Lemma~5.7]{cardanotes}.
\end{proof}

The importance of this lemma lies in the fact mentioned in \cite[Remark~2.3]{cardanotes} that if $E = \R^n$ and $\sup_{1 \leq i \leq n} \norm{D_i h}_\infty \lesssim N^{-1}$, then $\bs x \mapsto h(\gamma,\bs x)$ is Lipschitz continuous with respect to the $1$-Wasserstein distance of the empirical measures in the sense of \eqref{ENS_lipW1cond}. Hence we will be able to apply this lemma to $f_N$ and $g_N$ thanks to the symmetry \bf(\hyperlink{S}{S})} and assumptions~{\bf(\hyperlink{MF}{MF})}--{\bf(\hyperlink{LP}{LP})} as well as to $\frk u$ and $D_x \frk u$ thanks to \Cref{ENS_deldiui} and our estimates on the derivatives. 

We formalise a generalisation of the above-mentioned fact in the following lemma, where we denote by $|\cdot|_p$ the standard $p$-norm on $(\R^d)^N$.

\begin{lem} \label{ENS_DtoW}
Let $h \colon (\R^d)^N \to \R$ be symmetric and such that $\norm{|Dh|_q}_\infty \leq CN^{-\frac1p}$ with $\frac1p + \frac1q = 1$, for some constants $C>0$ and $q \in (1,\infty]$ independent of $N$.
Then $|h(\bs x) - h(\bs y)| \leq CW_p(m_{\bs x}, m_{\bs y})$ for all $\bs x, \bs y \in (\R^d)^N$.
\end{lem}

\begin{proof}
Recall that $W_p(m_{\bs x}, m_{\bs y}) = N^{-\frac1p} |\bs x - \bs y^\sigma|_p$ for some permutation $\sigma$ on $\{1,\dots,N\}$.
Then by the symmetry of $h$ and H\"older's inequality, $|h(\bs x) - h(\bs y)| \leq \norm{|Dh|_q}_\infty |\bs x - \bs y^\sigma|_p$. The conclusion follows.
\end{proof}

We are now ready to prove the existence of a limit of $\frk u_N$ as $N \to \infty$.

\begin{thm}\label{ENS_rmkconvdiuil}
Let assumptions \hyperlink{MF}{{\bfseries{(MF)}}}, \hyperlink{S}{{\bfseries{(S)}}} and \hyperlink{LP}{{\bfseries{(LP)}}} be in force. Assume also that one between \hyperlink{DS}{{\bfseries{(DS)}}} and \hyperlink{LS}{{\bfseries{(LS)}}} holds, with the corresponding semimonotonicity constants being such that, given $T>0$, the thesis of \Cref{ENS_mainthm} holds.

Then, there exists a map $U \colon \Lambda \times [0,T] \times \R^d \times \call P_2(\Lambda \times \R^d) \to \R$ with bounded derivative $D_x U$ such that $U$ and $D_x U$ are Lipschitz continuous on $\Lambda \times \R^d \times \Pc_2(\Lambda \times \R^d)$ and $\frac13$-H\"{o}lder continuous on $[0,T]$, and, up to a subsequence,
\begin{equation} \label{ENS_convD1}
\sup | D_x^k \frk u_N({\bs\lambda},t,\bs x) - D_x^k U(\lambda,t,x,m_{(\hbs\lambda:\hbs x)}) | \,\xrightarrow[N \to \infty]{}\, 0, \qquad k \in \{0,1\},
\end{equation}
whenever the supremum is taken over any set of the form
\[
\bigl\{ (\bs\lambda,t, \bs x) \in \Lambda^N \times [0,T] \times (\R^d)^N :\ |x| \leq R,\ m_{\hbs x} \in \call K \bigr\}
\]
for some $R > 0$ and some compact set $\call K \subset \Pc_2(\R^d)$.
\end{thm}

\begin{proof}
For the local uniform convergence of a subsequence $\frk u_N$, we apply \Cref{ENS_lemcarconv} with $\gamma = (t, \lambda, x)$ and $\bs v = (\hbs \lambda:\hbs x)$. Indeed, first note that the $\frk u_N$ are globally bounded, uniformly in $N$, by a standard comparison argument (recall that $f^i$ and $g^i$ are now assumed to be bounded uniformly). The second inequality in \eqref{ENS_recapE1} and the estimate of \Cref{ENS_contlcontr} yield uniform Lipschitz continuity with respect to $(\hbs \lambda:\hbs x)$ in $W_2$ distance, in the sense provided by \Cref{ENS_DtoW}. Lipschitz continuity in $(\lambda, x)$ is also a consequence of \Cref{ENS_contlcontr}, and the first estimate in \eqref{ENS_recapE1}. Finally, since the right-hand side of the  equation for $\frk u_N$
\[
-\de_t \frk u_N - \Delta \frk u_N = - \dfrac12 |D_x \frk u_N|^2 - \displaystyle\sum_{2 \leq j \leq N} D_j u^j \cdot D_j \frk u_N + f_N
\]
is uniformly bounded by \eqref{ENS_recapE1}, and $\frk u_N$ is uniformly Lipschitz  in the $\bs x$-variable with respect to the weighted norm $\|\cdot\|^1$, again by \eqref{ENS_recapE1}, the desired uniform H\"older regularity in $t$ follows by \Cref{lemmaholder}.

The local uniform convergence of $\partial_x \frk u_N$ is analogous. We argue as before, employing estimates in \eqref{ENS_recapE2} instead of \eqref{ENS_recapE1}. Time regularity is obtained from the equation
\[
-\de_t (D_x \frk u_N) - \Delta (D_x \frk u_N) = - \displaystyle\sum_{1 \leq j \leq N} D_j u^j \cdot D_j (D_x \frk u_N) - \displaystyle\sum_{2 \leq j \leq N} (D^2_{j1} u^j) \cdot D_j \frk u_N + D_x f_N,
\]
whose right-hand side can be controlled using \eqref{ENS_recapE1}, \eqref{ENS_recapE2}.
%Thanks to estimate~\eqref{ENS_recapE1} and \Cref{ENS_deldiui}, one can apply  and then  to $\frk u^N$; in this way one identifies the limit function $U$. Note that the Lipschitz continuity with respect to $x \in \R^d$ follows from \eqref{ENS_recapEX}, the one with respect to $\lambda \in \Lambda$ from \Cref{ENS_contlcontr}, while for the one with respect to $t \in [0,T]$ it suffices to note that estimates~\eqref{ENS_recapE1} and \eqref{ENS_recapE4} yield (using the Nash system) $\norm{\de_t \frk u^N}_\infty \lesssim 1$, with implied constant independent of $N$.
%
%Thanks to estimate~\eqref{ENS_recapE2} and again \Cref{ENS_deldiui}, the same lemmata above can be applied to $D_x \frk u^N$, and the corresponding limit function is easily shown to coincide with $D_x U$ by exploiting the fundamental theorem of calculus. The claimed Lipschitz continuity follows from \Cref{ENS_contlcontr} and estimate~\eqref{ENS_recapE4}. The H\"{o}lder continuity with respect to $t$ can be deduced from all the derivative estimates and the Krylov--Safonov theorem applied to the equation of $\frk u^N$ derived with respect to $x$.
\end{proof}

\begin{rmk} Regularity of $U$ is actually a bit better. First, \eqref{ENS_recapE1} yields Lipschitz continuity in the measure variable with respect to the $W_1$ distance. Moreover, looking at the equation for $\frk u_N$ one observes that $\norm{\de_t \frk u_N}_\infty \lesssim 1$, that gives uniform Lipschitz continuity in the variable $t$.
\end{rmk}

\begin{rmk} \label{ENS_convhN}
In \Cref{ENS_exmonL} we discussed the implications of our assumption when taking the data $h^N$ (that is, $f^N$ and $g^N$) as projections over empirical measures of given $\frk h_N : \Lambda \times \R^d \times \Pc_p(\R^d) \to \R$, and in \Cref{ENS_rmkfdp1} we argued that this can be considered a natural choice. 

Furthermore, our standing assumptions force in fact also the convergence of $h_N$ to functions defined on probability measures, at least up to subsequences. Indeed,
assumption~{\bf(\hyperlink{MF}{MF})} guarantees that $\norm{D_j h_N(\lambda,\var)}_\infty \leq L_h N^{-1}$ for all $j > 1$%, while given $\bs\lambda \in \Lambda^N$, $h_N(\lambda_i,\cdot)$ is actually independent of $\bs\lambda_{-i}$
. Then by \Cref{ENS_DtoW} we can apply \Cref{ENS_lemcarconv} and say that
there exists a map $h \in \Lip( \Lambda \times \R^d \times \Pc_1(\R^d) )$ such that, up to a subsequence,
\begin{equation} \label{ENS_convh}
\sup_{\lambda \in \Lambda,\, |x| \leq R,\, m_{\hbs x} \in \call K} |  D_x^k h_N(\lambda,\bs x) -  D_x^k h(\lambda,x,m_{\hbs x}) | \,\xrightarrow[N \to \infty]{}\, 0, % \quad k \in \{0,1\},
\end{equation}
for any $R > 0$ and $\call K \subset \Pc_1(\R^d)$ compact, and for $k$ as in \hyperlink{LP}{{\bfseries{(LP)}}}.
\end{rmk}

\subsection{Characterisation of the limit}

Henceforth, we will work with the three limit functions $f$, $g$ and $U$ defined in \Cref{ENS_convhN} and \Cref{ENS_rmkconvdiuil}. We will adopt the notation $f^\lambda=f(\lambda,\var)$ and analogously for $g^\lambda$ and $U^\lambda$.

Our aim is to give a characterisation of $U$. As a first step and a starter, the most basic result we can prove regards the standard symmetric setting of MFGs; that is, when $\bs\lambda = \lambda 1_N$ for some $\lambda \in \Lambda$, so that all $f^i$ (and all $g^i$) coincide up to a permutation of the coordinates of $(\R^d)^N$.

\begin{prop} \label{ENS_convfacile}
For all $\tau \in [0,T]$, $\frk m,\bar{\frk m} \in \Pc_2(\R^d)$ and $\lambda \in \Lambda$,
\begin{equation} \label{ENS_reprUl} \begin{multlined}[t][.9\displaywidth]
\int_{\R^d} U^\lambda(\tau,\var,\delta_\lambda\!\otimes \frk m) \,\di \bar{ \frk m }\vspace{-7pt}\\
= \int_\tau^T \int_{\R^d} \Big( \,\frac12\big|D_x U^\lambda(s,\var,\delta_\lambda\!\otimes m_s)\big|^2 + f^\lambda(\var, m_s) \Big) \di \bar m_s\di s + \int_{\R^d} g^\lambda(\var,  m_T) \,\di \bar m_T,
\end{multlined}
\end{equation}
where $m$ and $\bar m$ solve, respectively,
\begin{equation} \label{FPeq1}
\begin{cases}
\de_t m - \Delta m - \mathrm{div}(D_x U^\lambda(t,x,\delta_\lambda\!\otimes m)m) = 0 \\
m|_{t=\tau} = \frk m
\end{cases}
\end{equation}
and
\begin{equation} \label{FPeq2}
\begin{cases}
\de_t \bar m - \Delta \bar m - \mathrm{div}(D_x U^\lambda(t,x,\delta_\lambda\!\otimes m)\bar m) = 0  \\
\bar m|_{t=\tau} = \bar{\frk m}
\end{cases}
\end{equation}
for $t \in [\tau,T]$ and $x \in \R^d$.
\end{prop}

 In order to abridge the notation, given $\frk m \in \Pc_2(\R^d)$ we will write 
 \[
 U^\lambda(t,x,\frk m) \defeq U^\lambda(t,x,\delta_\lambda\!\otimes \frk m).
 \]
 Also, throughout the proofs we will imply all constants which depend only on $T$, $L_g$, $L_f$ and $L_\Lambda$. We state the basic propagation of chaos results needed for the proof in the following two lemmata.

\begin{lem} \label{ENS_propfacileL1}
Let $\lambda \in \Lambda$, $\bar m, m_0 \in \Pc_q(\R^d)$, $q > 4$, and let $\bs Z \sim \bar m \otimes m_0^{\otimes(N-1)}$; then consider $\bs X^N$ and $\bs Y^N$ solving, respectively,
\[
\begin{cases}
\di X^{N,i} = - D_x U^\lambda(t,X^{N,i}_t,m_{\bs X^{N}_t})\,\di t + \sqrt2\, \di B^i_t , & i \in \{1,\dots,N\}\\
\bs X_0^N = \bs Z
\end{cases}
\]
and
\[
\begin{cases}
\di Y^{N,i} = - D_x U^\lambda(t,Y^{N,i}_t,\Law_{Y^{N,i}_t})\,\di t + \sqrt2\, \di B^i_t, & i \in \{2,\dots,N\} \\
\di Y^{N,1} = - D_x U^\lambda(t,Y^{N,1}_t,\Law_{Y^{N,2}_t})\,\di t + \sqrt2\, \di B^1_t \\
\bs Y^N_0 = \bs Z.
\end{cases}
\]
Let $m_t \defeq \Law_{Y^{N,i}_t}$ if $i \geq 2$ (note that $\Law_{Y^{N,i}_t}$ is independent of $i \geq 2$ and $N$). Then
\begin{equation} \label{ENS_estT1.3}
\lim_{N \to \infty} \sup_{1 \leq i \leq N} \bb E\Big[ \sup_{t \in [0,T]} \big| X^{N,i}_t - Y^{N,i}_t \big|^2 \Big] = 0
\end{equation}
and
\begin{equation}\label{convservedopo}
\lim_{N \to \infty} \bb E \Bigl[\, \sup_{t \in [0,T]} W_2(m_{\bs X^{N}_t}, m_t) \Bigr] = 0.
\end{equation}
\end{lem}

\begin{proof}
Subtracting the equations of $\bs X^N$ and $\bs Y^N$ and using the Lipschitz continuity of $\partial_xU$, we have
\[
|X^{N,i}_t - Y^{N,i}_t| \lesssim \int_0^t \bigl( |X^{N,i} - Y^{N,i}| + W_2(m_{\bs X^{N}}, m_{\bs Y^{N}}) + W_2(m_{\bs Y^{N}},m) \bigr);
\]
therefore, since due to the boundedness of $D_x U$ we have\footnote{We are using the notation $m_{\bs Y^{N,-1}} = \frac1{N-1} \sum_{1 \leq j \leq N, \, j \neq 1} \delta_{Y^{N,j}}$, which is coherent to that given in Section \ref{sec_nota}; note that in the literature the same empirical measure is often denoted by $m^{N,1}_{\bs Y}$.} $\bb E \bigl[ \sup_{t \in [0,T]} W_2(m_{\bs Y^{N,-1}_t},m_{\bs Y^N_t})^2 \bigr] \lesssim N^{-1}$,
\begin{equation} \label{ENS_pc-1}
|X^{N,i}_t - Y^{N,i}_t|^2 \lesssim \int_0^t \Bigl( |X^{N,i} - Y^{N,i}|^2 + \frac1N \sum_{1 \leq j \leq N} |X^{N,j} - Y^{N,j}|^2 + W_2(m_{\bs Y^{N,-1}},m) \Bigr) + \frac1{N}.
\end{equation}
Averaging over $i$,
\[
\frac1{N} \sum_{1 \leq i \leq N} |X^{N,i}_t - Y^{N,i}_t|^2 \,\lesssim\, \frac1N + \int_0^t \Bigl( \frac1{N} \sum_{1 \leq i \leq N} |X^{N,i} - Y^{N,i}|^2 + W_2(m_{\bs Y^{N,-1}},m) \Bigr),
\]
thus by Gronwall's lemma
\[
\frac1{N} \sum_{1 \leq i \leq N} |X^{N,i}_t - Y^{N,i}_t|^2 \lesssim \frac1N + \int_0^t W_2(m_{\bs Y^{N,-1}},m)
\]
and plugging this back into \eqref{ENS_pc-1} one concludes that
\begin{equation*} %\label{ENS_estT1.3'}
\sup_{1 \leq i \leq N} \bb E\Big[ \sup_{t \in [0,T]} \big| X^{N,i}_t - Y^{N,i}_t \big|^2 \Big] \lesssim \frac1N + \int_0^T \bb E \big[ W_2(m_{\bs Y^{N,-1}},m)^2 \big],
\end{equation*}
where the right-hand side goes to $0$ as $N \to \infty$ by \cite[Theorem~1]{FG13} (see also \cite[Theorem~5.8]{cardelar}). Then note that by the triangle inequality one deduces \eqref{convservedopo} as well.
\end{proof}

\begin{lem} \label{ENS_propfacileL2}
Let $\bs X^N$ be as in \Cref{ENS_propfacileL1} and consider $\tbs X^N$ solving
\[
\begin{cases}
\di \tilde X^{N,i}_t = - D_x \frk u_N({\lambda1_N},t,\tilde X^{N,i}_t,\tbs X^{N,-i}_t)\,\di t + \sqrt2\, \di B^i_t, & i \in \{1,\dots,N\}\\
\tbs X_0^N = \bs Z.
\end{cases}
\]
Then
\begin{equation} \label{ENS_convtXX}
\lim_{N \to \infty} \sup_{1 \leq i \leq N} \bb E \Big[ \sup_{t \in [0,T]} |\tilde X^{N,i}_t - X^{N,i}_t|^2 \Big] = 0.
\end{equation}
\end{lem}

\begin{proof}
Note that, since $\bb E \bigl[ \sup_{t \in [0,T]} W_2(m_{\tbs X^{N,-i}_t},m_{\tbs X^N_t})^2 \bigr] \lesssim N^{-1}$,
\begin{equation} \label{ENS_esttXX} \begin{split}
|\tilde X^{N,i}_s - X^{N,i}_s|^2 &\lesssim \begin{multlined}[t][.7\displaywidth] \int_0^s \Big( \big| D_x \frk u_N({\lambda1_N},t,\tilde X^{N,i}_t,\tbs X^{N,-i}_t) - D_x \frk u_N({\lambda1_N},t, X^{N,i}_t,\tbs X^{N,-i}_t) \big|^2  \vspace{-7pt} \\
+ \big| D_x \frk u_N(\lambda1_N,t, X^{N,i}_t,\tbs X^{N,-i}_t) - D_x U^\lambda(t,X^{N,i}_t, m_{\tbs X^N_t}) \big|^2 \\
+ \big| D_x U^\lambda(t,X^{N,i}_t, m_{\tbs X^N_t}) - D_x U^\lambda(t,X^{N,i}_t, m_{\bs X^N_t}) \big|^2 \Big)\,\di t \vspace{-5pt}
\end{multlined} \\
& \lesssim \int_0^s \big( |\tilde X^{N,i}_t - X^{N,i}_t|^2 + \delta^{N,i}_t + W_2(m_{\tbs X^N_t}, m_{\bs X^N_t})^2 \big)\,\di t,
\end{split}
\end{equation}
where we have set
\[
\delta^{N,i}_t \defeq \big| D_x \frk u_N({\lambda1_N},t, X^{N,i}_t,\tbs X^{N,-i}_t) - D_x U^\lambda(t,X^{N,i}_t, m_{\tbs X^{N,-i}_t}) \big|^2 +  \frac1N.
\]
Taking the averages,
\[
\frac1N \sum_{1 \leq i \leq N} |\tilde X^{N,i}_s - X^{N,i}_s|^2 \lesssim \frac1N \sum_{1 \leq i \leq N} \int_0^T  \delta^{N,i}_t \,\di t + \int_0^s \frac1N \sum_{1 \leq i \leq N} |\tilde X^{N,i}_t - X^{N,i}_t|^2 \,\di t,
\]
so that Gronwall's lemma yields
\[
\sup_{t \in [0,T]} \frac1N \sum_{1 \leq i \leq N} |\tilde X^{N,i}_t - X^{N,i}_t|^2 \lesssim \frac1N \sum_{1 \leq i \leq N} \int_0^T  \delta^{N,i}_t \,\di t
\]
and, plugging this back into \eqref{ENS_esttXX},
\begin{equation} \label{ENS_convtXX'}
\sup_{1 \leq i \leq N} \bb E \Big[ \sup_{t \in [0,T]} |\tilde X^{N,i}_t - X^{N,i}_t|^2 \Big] \lesssim \frac1N \sum_{1 \leq i \leq N} \int_0^T  \bb E\, \delta^{N,i}_t \,\di t \,\xrightarrow[N \to \infty]{}\, 0;
\end{equation}
%the convergence to $0$ of the right-hand side of \eqref{ENS_convtXX} is due to \eqref{ENS_convD1},  and the dominated convergence theorem. 
in order to justify the convergence to $0$ of the right-hand side of \eqref{ENS_convtXX'}, we proceed as follows. First note that, since $D_x \frk u^N, D_x U$ are bounded uniformly and $\bar m, m_0 \in \Pc_q(\R^d)$, $q > 4$, it is standard that $\bb E |\tilde X^{N,i}_t|^q$ and $\bb E |X^{N,i}_t|^q$ are uniformly bounded (and thus, so is $N^{-1} \bb E \sum_{1 \leq j \leq N} |\tilde X^{N,j}_t|^q$). Then, for $\bs x \in (\R^d)^N$, let $M_q(\bs x) \defeq N^{-1} \sum_{1 \leq j \leq N} |x^j|^q$; since
\[
\int_0^T \bb E\, \delta^{N,i}_t \,\di t= \int_0^T  \bb E \bigl[ \bs 1_{\{M_q(\tbs X^N_t) + |X_t^{N,i}| \le R\}} \delta^{N,i}_t \bigr] \,\di t+ \int_0^T  \bb E \bigl[ \bs 1_{\{M_q(\tbs X^N_t) + |X_t^{N,i}| > R\}} \delta^{N,i}_t \bigr] \,\di t,
\]
the second term can be made arbitrarily small provided that $R$ is large, by the aforementioned bounds on $q$-th moments, while the first can be made small by choosing $N$ large enough ($\delta^{N,i}_t$ is bounded uniformly). Indeed, it is sufficient to apply the uniform convergence \eqref{ENS_convD1} on the set $\call K$ of all probability measures with $q$-th moment bounded by $R$,\footnote{Here we mean that $\call K \defeq \{ \nu \in \Pc_2(\R^d) :\ \int_{\R^d} |\bs x|^q \,\nu(\di\bs x) \leq R\}$.} which is compact in $\call P_2(\R^d)$. Note that the previous limit is uniform in $i$, since $K$ and the bounds on the moments are independent of $i$.
\end{proof}

\begin{proof}[Proof of \Cref{ENS_convfacile}]
We have
\[\begin{multlined}[.95\displaywidth]
\bigg| \big|D_x U^\lambda(t,\var)\big|^2\Big|_{(X^{N,i}_t,m_{\bs X^N_t})}^{(Y^{N,i}_t,m_t)} \bigg| + \bigg| f^\lambda \Big|_{(X^{N,i}_t, m_{\bs X^N_t})}^{(Y^{N,i}_t, m_t)} \bigg| + \bigg| g^\lambda \Big|_{(X^{N,i}_t, m_{\bs X^N_t})}^{(Y^{N,i}_t, m_t)} \bigg| \vspace{3pt}\\
\lesssim \big|X^{N,i}_t - Y^{N,i}_t\big| + W_2(m_{\bs X^N_t},m_t),
\end{multlined}\]
the expectation of the right-hand side going to $0$ as $N \to \infty$ (uniformly in $t \in [0,T]$) by \Cref{ENS_propfacileL1}. Then we deduce that
\begin{equation} \label{ENS_pclim1} \begin{multlined}[t][.9\displaywidth]
 \bb E\bigg[ \int_0^T \Big( \,\frac12\big|D_x U^\lambda(t,X^{N,1}_t,m_{\bs X^N_t})\big|^2 + f^\lambda(X^{N,1}_t, m_{\bs X^N_t}) \Big) \di t + g^\lambda(X^{N,1}_T, m_{\bs X^N_T}) \bigg]
\vspace{-7pt}\\ \xrightarrow[N \to \infty]{}\, \int_0^T \int_{\R^d} \Big( \,\frac12\big|D_x U^\lambda(t,\cdot,m_t)\big|^2 + f^\lambda(\cdot, m_t) \Big) \di \bar m_t\di t + \int_{\R^d} g^\lambda(\cdot, m_T) \,\di \bar m_T,
\end{multlined}
\end{equation}
where $\bar m_t \defeq \Law_{Y^{N,1}_t}$ (which is in fact independent of $N$).
Now, by the representation formula for the value functions solving of the Nash system (see, e.g., \cite[Section~2.1.4]{cardelar}),
\begin{equation} \label{ENS_rfu1pr}
\bb E \, \frk u_N(\lambda 1_N,0,\bs Z) = \bb E\biggl[\, \int_0^T \Big(\,\frac12\big| D_x \frk u_N(\lambda1_N,t,\tbs X^N_t) \big|^2 + f_N(\lambda,\tbs X^N_t)\Big)\di t + g_N(\lambda,\tbs X^N_T) %\ \Big\vert\ \tbs X^N_0 = \bs Z 
\biggr],
\end{equation}
where $\tbs X^N$ is the vector of optimal trajectories for the $N$-player game; that is, $\tbs X^N$ is as in \Cref{ENS_propfacileL2}. Then, combining \eqref{ENS_pclim1}, \eqref{ENS_rfu1pr} and \eqref{ENS_convtXX} one deduces that
\begin{equation} \label{ENS_pclim2}
\bb E \, \frk u_N({\lambda 1_N},0,\bs Z) \,\xrightarrow[N \to \infty]{}\, \int_0^T \int_{\R^d} \Big( \,\frac12\big|D_x U^\lambda(t,\var,m_t)\big|^2 + f^\lambda(\var, m_t) \Big) \di \bar m_t\di t + \int_{\R^d} g^\lambda(\var, m_T) \,\di \bar m_T.
\end{equation}
Finally, since $\bb E\,W_2(m_{\bs Z^{-1}}, m_0) \to 0$ by \cite[Theorem~1]{FG13}, from \eqref{ENS_convD1} and \eqref{ENS_pclim2} we obtain
\[
\int_{\R^d} U^\lambda(0,\var,m_0) \,\di \bar m = \int_0^T \int_{\R^d} \Big( \,\frac12\big|D_x U^\lambda(t,\var,m_t)\big|^2 + f^\lambda(\var, m_t) \Big) \di \bar m_t\di t + \int_{\R^d} g^\lambda(\var, m_T) \,\di \bar m_T.
\]
The fact that $m$ and $\bar m$ solve the Fokker--Planck equations \eqref{FPeq1} and \eqref{FPeq2}, respectively, is a standard application of It\^{o}'s formula. Finally, as $\bar m, m_0$ can be chosen in a dense subset of $\Pc_2(\R^d)$ and the choice of $\tau=0$ as initial time is arbitrary, \eqref{ENS_reprUl} follows by standard stability of the Fokker--Planck equation. 
\end{proof}
%
%\begin{rmk} We point out that the semimonotonicity of $u^N$, whose constant $M$ is stable with respect to $N$ is exploited in the previous section to derive the estimates, passes to the limit $N \to \infty$, and takes the form of the inequality
%\[
%\int_{\R^d \times \R^d} \big( D_x U^\lambda(\tau,x,\delta_\lambda\!\otimes \frk m_1) - D_x U^\lambda(\tau,y,\delta_\lambda\!\otimes \frk m_2) \big)\cdot(x-y) \, \frk m_1(\di x) \frk m_2(\di y) \geq  - M 
%\]
%\end{rmk}

%Thanks to the asymptotical continuity shown in \Cref{ENS_contlcontr}, it is then immediate to prove the following generalisation of \Cref{ENS_convfacile}. We will denote by $\bs 1$ the all-ones sequence in $\ell^\infty = \ell^\infty(\N;\Lambda)$ and embed $\Lambda^N \subset \ell^\infty$ via $\bs\lambda \mapsto (\bs\lambda,0,\dots)$.

%\begin{cor}
%Suppose that $\bs\lambda^N \to \lambda \bs 1$ coordinate-wise as $N \to \infty$, for some $\lambda \in \Lambda$. Then, up to a subsequence,
%\begin{equation} \label{ENS_convtol}
%| u^i_{\bs\lambda^N}(t,x,\bs y) - U^\lambda(t,x,m_{\bs y})| + | D_i u^i_{\bs\lambda^N}(t,x,\bs y) - D_x U^\lambda(t,x,m_{\bs y}) | \,\xrightarrow[N \to \infty]{} 0 \quad \forall\, i \in \N
%\end{equation}
%locally uniformly in $(t,x,\bs y) \in [0,T] \times \R^d \times (\R^d)^{N-1}$, where $U^\lambda$ is the map given in \Cref{ENS_convfacile}. Furthermore, if $\bs\lambda^N - \lambda 1_N \to 0$ in $\ell^\infty$, then the convergence \eqref{ENS_convtol} is uniform with respect to $i$.
%\end{cor}

We finally extend the characterisation \eqref{ENS_reprUl} to any measure $\bs\mu \in \Pc_2(\Lambda \times \R^d)$, with a continuous disintegration with respect to the projection $\pi_\Lambda \colon \Lambda \times \R^d \to \Lambda$. This will be the main result of this section. The extension is rather technical, but the main point is performing the following reduction: the continuity in $\lambda$ (of the disintegration and of the data) allows to approximate the problem with a simpler one, where the first marginal of $\mu$ is supported on a finite subset of $\Lambda$. This is a sort of problem with finitely many populations, where we apply the previous convergence argument; in other words, we are going to perform a sort of multiple ``localised'' propagation of chaos.

\begin{thm} \label{ENS_convfinita}
Let $\bs\mu \in \Pc_2(\Lambda \times \R^d)$ admit a disintegration $\frk m \colon K \to \Pc_2(\R^d)$ with respect to $\pi_\Lambda$ which is continuous on some compact subset $K \supseteq \spt(\pi_\Lambda\mpush\bs\mu)$. Let $\bbs{\frk m} \in \Pc_2(\R^d)^\Lambda$. Then for all $\tau \in [0,T]$ and $\lambda \in K$,
\begin{equation} \label{ENS_reprUlden} 
\begin{multlined}[t][.9\displaywidth]
\int_{\R^d} U^{\lambda}(\tau,\var,\bs\mu) \,\di \bar{\frk m}^\lambda
\vspace{-7pt}\\
= \int_\tau^T \int_{\R^d} \Big( \,\frac12\big|D_x U^{\lambda}(s,\var,\mu_s)\big|^2 + f^{\lambda}(\var, {\pi_{\R^d}}\mpush \mu_s) \Big) \,\di \bar m^{\lambda}_s\di s + \int_{\R^d} g^{\lambda}(\var, {\pi_{\R^d}}\mpush \mu_T) \,\di \bar m^{\lambda}_T,
\end{multlined}
\end{equation}
where $\mu \in C^0([\tau,T]; \Pc_2(\Lambda \times \R^d))$ and $\bar m \in C^0([\tau,T];\Pc_2(\R^d))^\Lambda$ solve, respectively,
\begin{equation} \label{ENS_FPLambdaC}
\begin{cases}
\de_t m^\lambda - \Delta m^\lambda - \mathrm{div}(D_x U^\lambda(t,x, \mu_t) m^\lambda) = 0 \\
\mu_t = \int_\Lambda m_t^\lambda\, \pi_\Lambda\mpush \bs\mu(\di\lambda)\\
m^\lambda_\tau = \frk m^\lambda
\end{cases}
\end{equation}
and
\begin{equation} \label{ENS_FPLambdaC'}
\begin{cases}
\de_t \bar m^\lambda - \Delta \bar m^\lambda - \mathrm{div}(D_x U^\lambda(t,x, \mu_t) \bar m^\lambda) = 0 \\
\bar m^\lambda_\tau = \bar{\frk m}^\lambda
\end{cases}
\end{equation}
for $t \in [\tau, T]$, $x \in \R^d$ and $\lambda \in \Lambda$.
\end{thm}

\begin{rmk}
Since $D_x U^\lambda$ and ${\rm div} (D_x U^\lambda)$ are globally bounded, $\bar m^\lambda$ is locally continuous as a function with values in $L^1({\R^d})$, it solves the Fokker-Planck equation in the strong sense, and it belongs locally to $W^{2,1}_p$ for every $p$; see for example \cite{BKR}. Regularity extends up to time $\tau$ provided that $\bar{\frk m}^\lambda$ is smooth enough.
\end{rmk}

Problem~\eqref{ENS_FPLambdaC'} is simply an uncoupled system of (possibly uncountably many) Fokker--Planck equations. On the other hand, before proceeding with the proof of this theorem we show that problem~\eqref{ENS_FPLambdaC} is well-posed, as a result of the following lemma. Also note that it can be written as
\begin{equation} \label{ENS_FPLambdaC1FP}
\begin{cases}
\de_t \mu - \Delta_x \mu - \mathrm{div}_x(D_x U^\lambda(t,x, \mu_t) \mu) = 0 & \text{on} \ [\tau,T] \times \Lambda \times \R^d \\
\mu_\tau = \bs \mu,
\end{cases}
\end{equation}
so that $\mu$ solves an equation of Fokker--Planck type with no derivatives with respect to $\lambda$.

\begin{lem} \label{ENS_lemFPL}
Let $B \colon \Lambda \times [0,T] \times \R^d \times \Pc_2(\Lambda \times \R^d) \to \R^d$ be bounded Lipschitz continuous on $\Lambda \times \R^d \times \Pc_2(\Lambda \times \R^d)$ uniformly on $[0,T]$. Let $\bs\mu \in \Pc_2(\Lambda \times \R^d)$ and let $\rho \defeq {\pi_\Lambda}\mpush \bs\mu \in \Pc(\Lambda)$ be its first marginal. Then problem\footnote{With an abuse of notation we write $\int_\Lambda m^\lambda_t \rho(\di\lambda)$ instead of $\int_\Lambda \hat m^\lambda_t \rho(\di\lambda)$ with $\hat m^\lambda_t(E) \defeq m^\lambda(\pi_{\R^d}(E \cap \pi_\Lambda^{-1}(\{\lambda\})))$ for all Borel sets $E \subseteq \Lambda \times \R^d$. We will continue to identify $\hat m^\lambda$ with $m^\lambda$.\label{ENS_fndis}}
\begin{equation} \label{ENS_FPLambda}
\begin{cases}
\de_t m^\lambda - \Delta m^\lambda + \mathrm{div}(B^\lambda(t,x, \mu_t) m^\lambda) = 0 \quad \text{on}\ [0,T] \times \R^d, \ \ \forall\, \lambda \in \Lambda\\
\mu_t = \int_\Lambda m_t^\lambda \rho(\di\lambda) \quad \forall \, t \in [0,T]\\
\mu_0 = \bs\mu
\end{cases}
\end{equation}
has a unique solution $\mu \in C^{\frac12}([0,T]; \Pc_2(\Lambda \times \R^d))$, whose disintegration $m$ with respect to $\rho$ is unique as well if one fixes the disintegration of $\bs\mu$.

Furthermore, if $\mu$ and $\tilde \mu$ are the solutions relative to data $\bs\mu$ and $\tbs\mu$ with respective disintegrations given by $\frk m$ and $\tilde{\frk m}$, respectively, one has
\begin{gather} \label{ENS_gronmll}
\sup_{t \in [0,T]}W_2(\mu_t, \tilde\mu_t) \,\lesssim\, W_2(\bs\mu,\tbs\mu),
\\ \label{ENS_estmutmu}
\sup_{t \in [0,T]} W_2(m^\lambda_t,\tilde m^{\lambda'}_t) \,\lesssim\, W_2(\frk m^\lambda,\tilde{\frk m}^{\lambda'}) + W_2(\bs\mu,\tbs\mu) \quad \forall\,\lambda,\lambda' \in \Lambda,
\end{gather}
where the implied constants depend only on $T$ and $\norm{B}_{\Lip}$.
\end{lem}

\begin{proof} For some $A > 0$ to be determined, let $\frk C$ be the set of all $\mu \in C^0([0,T];\Pc_2(\Lambda \times \R^d))$ such that $[\mu]_{0,\frac12} \defeq \sup_{s,t \in [0,T],\, s \neq t} |s-t|^{-\frac12}W_2(\mu_s,\mu_t) \leq A$.
Note that $\frk C$ a closed convex subset $C^0([0,T];\Pc_2(\Lambda \times \R^d))$. Given $\mu \in \frk C$ and $\lambda \in \Lambda$, let $m^\lambda \in C^0([0,T];\Pc_2(\R^d))$ be the unique solution to $\de_t m^\lambda - \Delta m^\lambda + \mathrm{div}(B^\lambda(t,x,\mu_t)m^\lambda) = 0$ with $m^\lambda_0 = \frk m^\lambda$, where $(\frk m^\lambda)_{\lambda \in \Lambda}$ is a fixed disintegration of $\bs\mu$ with respect to $\rho$. Let $\Phi(\mu) \defeq \int_\Lambda m^\lambda \rho(\di\lambda)$. It is standard (cf.~\cite[Lemma~1.6]{CPnotes}) to show that $[m^\lambda]_{0,\frac12}$ is bounded by a constant which increases with $\norm{B^\lambda}_\infty$. In addition, if $\pi_{s,t}^\lambda$ is an optimal coupling of $(m^\lambda_s,m^\lambda_t)$, then $\textstyle{\int_{\Lambda \times \Lambda}} \pi^\lambda_{s,t} (( \mathrm{id}_\Lambda, \mathrm{id}_\Lambda )\mpush \rho)(\di \lambda,\di \lambda')$
is a coupling of $(\tilde\mu_s,\tilde\mu_t)$, showing that $W_2(\Phi(\mu)_s, \Phi(\mu)_t)^2 \leq \int_\Lambda W_2(m^\lambda_s,m^\lambda_t)^2 \rho(\di\lambda)$. Therefore, since $\sup_{\lambda \in \Lambda} \norm{B^\lambda}_\infty$ is finite, if $A$ is sufficiently large then the map $\Phi \colon \frk C \to \frk C$ is well-defined.

The following estimates will prove that $\Phi$ is a contraction for $T$ small, and thus (by iteration) the existence of a unique solution to \eqref{ENS_FPLambda} with arbitrary $T$. Furthermore, a posteriori, they will also yield the desired estimates~\eqref{ENS_gronmll} and \eqref{ENS_estmutmu}.

Let $\mu,\tilde \mu \in \frk C$ and let two initial data $\bs\mu,\tbs\mu \in \Pc(\Lambda \times \R^d)$ be given; let $\rho = {\pi_\Lambda}\mpush \bs\mu$, $\tilde\rho = {\pi_\Lambda}\mpush \tbs\mu$ and $(\frk m^\lambda)_{\lambda \in \Lambda}$ (resp.\ $(\tilde{\frk m}^\lambda)_{\lambda \in \Lambda})$ be a disintegration of $\bs\mu$ (resp.\ $\tbs \mu$) with respect to $\rho$ (resp.~$\tilde\rho$). Let $\nu \in C^0([0,T]; \Pc_2((\Lambda \times \R^d)^2)$ satisfy
\[
\de_t \nu - \Delta_x \nu - \Delta_y \nu + 2\sum_{1 \leq j \leq d} \de^2_{x^j y^j} \nu + \mathrm{div}_x(B^\lambda(t,x,\mu_t)\nu) + \mathrm{div}_y(B^{\lambda'}(t,y,\tilde\mu_t)\nu) = 0,
\]
so that if $\nu_0 \in \Pi(\bs\mu,\tbs\mu)$, then $\nu_t$ is a coupling of $(\Phi(\mu)_t,\Phi(\tilde\mu)_t)$. Testing by $|\lambda-\lambda'|^2 + |x-y|^2$ and assuming that $\nu_0$ is optimal for the $2$-Wasserstein distance, we obtain
\[\begin{multlined}[.95\displaywidth]
\int_{(\Lambda \times \R^d)^2} (|\lambda-\lambda'|^2 + |x-y|^2)\, \nu_t(\di\lambda,\di x,\di\lambda',\di y) \vspace{-5pt}\\
= W_2(\bs\mu,\tbs\mu)^2 + 2 \int_0^t \int_{(\Lambda \times \R^d)^2} \big( B^\lambda(t,x,\mu_t) - B^{\lambda'}(t,y,\tilde\mu_t) \big) \cdot (x-y) \, \di \nu_s(\di\lambda,\di x, \di\lambda',\di y)\, \di s.
\end{multlined}\]
By the Cauchy--Schwarz and Young's inequalities,
\[\begin{multlined}[.95\displaywidth]
\int_{(\Lambda \times \R^d)^2} (|\lambda-\lambda'|^2 + |x-y|^2)\, \nu_t(\di\lambda,\di x,\di\lambda',\di y) \vspace{-10pt} \\
\leq W_2(\bs\mu,\tbs\mu)^2 + C\int_0^t W_2(\mu_s,\tilde\mu_s)^2\,\di s + C\int_0^t \int_{(\Lambda \times \R^d)^2} (|\lambda-\lambda'|^2 + |x-y|^2)\, \nu_s(\di\lambda,\di x,\di\lambda',\di y) \,\di s,
\end{multlined}\]
where $C$ depends only on $\norm{B}_{\Lip}$. Apply Gronwall's lemma to deduce
\begin{equation} \label{ENS_proofFPG1}
W_2(\Phi(\mu)_t, \Phi(\tilde\mu)_t)^2 \leq \biggl( W_2(\bs\mu,\tbs\mu)^2 + C\int_0^t W_2(\mu_s,\tilde\mu_s)^2\,\di s \biggr) e^{Ct},
\end{equation}
which implies that $\Phi$ is a contraction on $\frk C \subset C^0([0,T];\Pc_2(\Lambda \times \R^d))$ if $CTe^{CT} < 1$. Since this smallness condition on $T$ depends only on $\norm{B}_{\Lip}$, we can iterate finitely many times the short-time uniqueness provided by the contraction theorem in order to obtain uniqueness on an arbitrary horizon.

Now $\Phi(\mu) = \mu$ and $\Phi(\tilde\mu) = \tilde\mu$ in \eqref{ENS_proofFPG1} above, then we can apply Gronwall's lemma once again to get \eqref{ENS_gronmll}.
With a similar argument (which considers a coupling of $(m^\lambda_t,\tilde m^{\lambda'}_t)$, where $m^\lambda$ is a solution to \eqref{ENS_FPLambda} given $m^\lambda_0 = \frk m^\lambda$ and analogously for $\tilde m^{\lambda'}$) one can also deduce that
\[
W_2(m^\lambda_t,\tilde m^{\lambda'}_t)^2 \,\lesssim\, W_2(\frk m^\lambda,\tilde{\frk m}^{\lambda'})^2 + \int_0^t W_2(\mu_s,\tilde\mu_s)^2\,\di s.
\]
Then, using \eqref{ENS_gronmll}, one has \eqref{ENS_estmutmu}.
\end{proof}

For the benefit of the reader, we also collect the crucial steps of the proof of \Cref{ENS_convfinita} in some lemmata, in order to isolate the propagation of chaos and the approximation arguments. We begin with the latter.

\begin{lem} \label{ENS_convfinitaapprox}
Suppose that \Cref{ENS_convfinita} holds for a sequence $(\bs\mu_k)_{k\in\N} \subset \Pc_2(\Lambda \times \R^d)$ such that $\pi_\Lambda\mpush\bs\mu_k$ is finitely supported and weakly converges to $\rho \in \Pc(\Lambda)$. Then \Cref{ENS_convfinita} holds for any $\bs\mu$ as in its statement such that $\pi_\Lambda\mpush\bs\mu = \rho$.
\end{lem}

%We are going to show that the proof for general $\bs \mu$ can be reduced to the case where $\pi_\Lambda\mpush\bs\mu $ is finitely supported. The convergence in the latter case of ``finitely many'' populations of players will be discussed in the steps below.

%Fix arbitrary $\ell \in \N$, $(\theta_j)_{1 \leq j \leq \ell} \subset \bb Q_+$ and $(\lambda^j)_{1 \leq j \leq \ell} \subset \Lambda$. We prove that it suffices to establish the theorem for $\rho = \sum_{1 \leq j \leq \ell} \theta_j \delta_{\lambda^j} \in \Pc(\Lambda)$.

\begin{proof}
Given $\bs\mu \in \Pc_2(\Lambda \times \R^d)$ as in the statement of \Cref{ENS_convfinita} and $\rho \defeq {\pi_\Lambda}\mpush\bs\mu$, let $(\rho_k)_{k \in \N} \subset \Pc(K)$ be a sequence of finitely-supported probability measures such that $\rho_k \to \rho$ weakly. Since $W_2$ metrizes the weak convergence of probability measures on $K$ (see, e.g., \cite[Corollary~6.13]{villaniOTON}), the existence of such a sequence comes from the Banach--Alaoglu and the Krein--Milman theorems. Moreover, note that one can in particular choose $\rho_k$ supported only on rational numbers.

Consider the continuous disintegration $\frk m \colon K \to \Pc_2(\R^d)$ of $\bs\mu$ with respect to $\pi_\Lambda$ and define $\bs\mu_k \defeq \int_\Lambda \frk m^\lambda \rho_k(\di\lambda)$; then, let $m_k$ be the corresponding solution to \eqref{ENS_FPLambdaC} (note that $\pi_\Lambda\mpush (\mu_k)_t =\rho_k$ for all $t$). Thanks to the continuity of the disintegration $\frk m$, the map $\lambda \mapsto \int_{\R^d} F(\lambda,x)\frk m^\lambda(\di x)$ is continuous, so, for any $F \in C(\Lambda \times \R^d)$ with sub-quadratic growth with respect to $x \in \R^d$, by the weak convergence $\rho_k \to \rho$ we have
\[
\int_{\Lambda \times \R^d} F(\lambda,x)  \frk m^\lambda(\di x) \rho_k(\di \lambda) \to \int_{\Lambda \times \R^d} F(\lambda,x)  \frk m^\lambda(\di x) \rho(\di\lambda);
\]
that is, $\bs\mu_k \to \bs\mu$ in $\Pc_2(\Lambda \times \R^d)$. Therefore, by \eqref{ENS_gronmll}, $\mu_k \to \mu$ in $C^0([\tau,T];\Pc_2(\Lambda \times \R^d))$, which implies that $\pi_{\R^d}\mpush \mu_k \to \pi_{\R^d}\mpush \mu$ in $C^0([\tau,T];\Pc_2(\R^d))$ as well.

%and by \eqref{ENS_estmutmu}, $m_k \to m$ in $C^0(\call K \times [\tau,T]; \Pc_2(\R^d))$; \ver{indeed, \eqref{ENS_estmutmu} yields both
%\[
%\sup_{[0,T]} W_2(m^\lambda^k, m^{\lambda'}_k) \,\lesssim\, W_2(\frk m^\lambda,{\frk m}^{\lambda'}) \quad \text{and} \quad 
%\sup_{\Lambda \times [0,T]} W_2(m, m_k) \,\lesssim\, W_2(\bs\mu,\bs\mu_k).
%\]
Let now $\lambda \in K$ and assume that $\lambda \in \bigcap_{k \in \N} \spt(\rho_k)$.
%Since by the Portmanteau theroem one has $\spt(\rho) = \bigcap_{h \in \N} \bar{\bigcup_{k \geq h} \spt(\rho_k)}$, for any $\lambda \in \spt(\rho)$ there exists a sequence $(\lambda_{k(h)})_{h \in \N}$ with $\lambda_{k(h)} \in \spt(\rho_{k(h)})$ such that $\lambda_{k(h)} \to \lambda$.
Standard stability of the Fokker--Planck equation gives the convergence $\bar m^\lambda_k \to \bar m^\lambda$ in $C^0([\tau,T];\Pc(\R^d))$, so we have all the ingredients to pass the identity \eqref{ENS_reprUlden} (that we are assuming to hold when $\pi_\Lambda\mpush \bs \mu$ is finitely supported) to the limit $k \to \infty$. 
On the other hand, if $\lambda \notin \bigcap_{k \in \N} \spt(\rho_k)$ one can consider, in place of $\rho_k$, measures $\tilde\rho_k \defeq \frac{N-1}{N}\rho_k + \frac1N \delta_{\lambda}$, so that still $\tilde\rho_k \to \rho$ weakly and the above argument can be still performed.
\end{proof}

\begin{lem} \label{ENS_convfinitaproplem}
Let $(\theta_j)_{1 \leq j \leq \ell} \subset \bb R_+$ be such that $\sum_{1 \leq j \leq \ell} \theta_j = 1$, and define
\begin{equation} \label{ENS_notUlvm}
U^\lambda(t,x,\bs{\frk m}) \defeq U^\lambda\Bigl(t,x, \sum_{1 \leq j \leq \ell} \theta_j\delta_{\lambda^{j}} \!\otimes \frk m^j \Bigr).
\end{equation}
Let $\bs m_0 \in \Pc_q(\R^d)^\ell$, $q > 4$, and consider $\bs m^{\bs\lambda}$ solving
\begin{equation} \label{ENS_fokplfinit}
\begin{cases}
\de_t m^{\lambda^j} - \Delta m^{\lambda^j} - \mathrm{div}(D_x U^{\lambda^j}(t,x,\bs m^{\bs\lambda}_t)m^{\lambda^j}) = 0, &  j \in \{1,\dots,\ell\} \\
 \bs m^{\bs \lambda}|_{t=\tau} = \bs m_0,
 \end{cases}
\end{equation}
where $\bs m^{\bs\lambda} \defeq (m^{\lambda^j})_{1 \leq j \leq \ell}$. %\footnote{For the sake of clarity, we remark that, according to notation~\eqref{ENS_notUlvm},
%\[
%D_x U^{\lambda^j}(t,x,\bs m^{\bs\lambda}_t) = D_x U^{\lambda^j}\Bigl(t,x,\sum_{1 \leq k \leq \ell} \theta_k\delta_{\lambda^{k}} \!\otimes m^{\lambda^k}_t \Bigr).
%\]}
Fix $\bbs{\frk m} \in \Pc_2(\R^d)^\ell$. Let $(I_j)_{1 \leq j \leq \ell}$ be a partition of $\N$ and let $\bs Z$ be an $(\R^d)^N$-valued random variable such that
\[
Z^i \sim \begin{cases}
\bar{\frk m}^j & \text{if} \ i = \min I_j^N \ \text{(for some $j \in \{1,\dots,\ell\}$)} \\
m_0^j & \text{if} \ i \in I_j^N \setminus \bigcup_{1 \leq k \leq \ell} \{ \min I_k^N \},
\end{cases}
\]
where $I^N_j \defeq I_j \cap \{1,\dots,N\}$.
Let $\bs X^{I_\cdot^N}$ and $\tbs X^N$ be the solutions of
\[
\di X^{I_j^N,i} = - D_x U^{\lambda^j}\bigl(t,X^{I_j^N,i},(\bs m_t^{\bs\lambda^{-j}},m_{\bs X^{I^N_j}_t})\bigr)\,\di t + \sqrt2\,\di B_t^i,
\qquad i \in I^N_j, \ \ j \in \{1,\dots,\ell\},\footnotemark
\]
\footnotetext{For the sake of clarity, we remark that, according to the notation introduced in \Cref{sec_nota}, the vector $(\bs m^{\bs \lambda^{-j}},m)$ is that with all coordinates given by $m^{\lambda^i}$ but the $j$-th one, which is $m$, hence (recall also notation~\eqref{ENS_notUlvm})
\[
D_x U^{\lambda^j}(t,x,(\bs m_t^{\bs\lambda^{-j}},m)) = D_x U^{\lambda^j}\Bigl(t,x, \theta_j \delta_{\lambda^j} \! \otimes m + \sum_{\substack{1 \leq k \leq \ell \\ k \neq j}} \theta_k \delta_{\lambda^k}\!\otimes m^{\lambda^k}_t \Bigr). \vspace{-15pt}
\]}and
\[
\di \tilde X^{N,i} = - D_x \frk u_N(\lambda_N^i,\bs\lambda_N^{-i},t,\tilde X^{N,i}_t,\tbs X^{N,-i}_t)\,\di t + \sqrt2\,\di B_t^i, \qquad i \in \{1,\dots,N\},
\]
respectively, with $\bs X^{I_\cdot^N}_0 = \bs Z = \tbs X^N_0$. If
\[
\lim_{N \to \infty} \frac{\#I^N_j}{N} = \theta_j \quad \forall\, j \in \{1,\dots,\ell\},
\]
then 
\begin{equation} \label{ENS_PCfinit}
\lim_{N \to \infty} \sup_{1 \leq j \leq \ell} \sup_{i \in I_j^N} \bb E \Bigl[\, \sup_{t \in [0,T]} \big|\tilde X^{N,i}_t - X^{I^N_j,i}_t \big|^2 \Bigr] = 0.
\end{equation}
\end{lem}

\begin{proof}
We first note that, since $\bs m_t^{\bs\lambda}$ is a given flow of measures, each agent $i \in I^N_j$ is driven by the same drift \[
V^j\bigl(t,X^{I_j^N,i},m_{\bs X^{I^N_j}_t}) \defeq - D_x U^{\lambda^j}\bigl(t,X^{I_j^N,i},(\bs m_t^{\bs\lambda^{-j}},m_{\bs X^{I^N_j}_t})\bigr),
\]
we can invoke \Cref{ENS_propfacileL1} to have, in particular,
\begin{equation} \label{ENS_alsohavefromstep1}
\lim_{N \to \infty} \bb E \Bigl[\, \sup_{t \in [0,T]} W_2(m_{\bs X^{I_j^N}_t}, m^{\lambda^j}_t) \Bigr] = 0;
\end{equation}
this will be needed afterwards.

Now, if $i \in I^N_j$ we have
\begin{equation} \label{ENS_XNXINj}
\begin{multlined}[t][.9\displaywidth] 
\big|\tilde X^{N,i}_s - X^{I^N_j,i}_s\big|^2 \\
\leq 2 \int_0^s \Big( \big| D_x \frk u^N(\lambda^j,\bs\lambda_N^{-i},t,\tilde X^{N,i}_t,\tbs X^{N,-i}_t) - D_x \frk u^{N}(\lambda^j,\bs\lambda_N^{-i},t, X^{I^N_j,i}_t,\tbs X^{N,-i}_t) \big|^2  \vspace{-7pt} \\
\hspace{45.5pt} + \big| D_x \frk u^{N}(\lambda^j,\bs\lambda_N^{-i},t, X^{I^N_j,i}_t,\tbs X^{N,-i}_t) - D_x U(\lambda^j,t,X^{I^N_j,i}_t, m_{(\bs\lambda_N:\tbs X^N_t)}) \big|^2 \\
+ \big| D_x U(\lambda^j,t,X^{I^N_j,i}_t, m_{(\bs\lambda_N:\tbs X^N_t)}) - D_x U^{\lambda^j}\bigl(t,X^{I^N_j,i}_t,(\bs m_t^{\bs\lambda^{-j}},m_{\bs X^{I^N_j}_t})\bigr) \big|^2 \Big)\,\di t;
\end{multlined}
\end{equation}
note that, according to the notation introduced in \Cref{sec_nota},
\[
m_{(\bs\lambda_N:\tbs X^N_t)} = \frac1N \sum_{1 \leq k \leq N} \delta_{\lambda^k_N}\!\otimes \delta_{\tilde X^{N,k}_t}.
\]
By the Lipschitz continuity of $D_x U$ (throughout the proof we will imply constants if they depend only on $T$, $L_g$, $L_f$ and $L_\Lambda$),
\begin{equation} \label{ENS_estBfinit} \begin{multlined}[t][.9\displaywidth]
\big|D_x U(\lambda^j,t,X^{I^N_j,i}_t, m_{(\bs\lambda_N:\tbs X^N_t)}) - D_x U^{\lambda^j}\bigl(t,X^{I^N_j,i}_t,(\bs m_t^{\bs\lambda^{-j}},m_{\bs X^{I^N_j}_t})\bigr) \big|^2 \\
\lesssim W_2\Big(m_{(\bs\lambda_N:\tbs X^N_t)}, \ \theta_j \delta_{\lambda^j} \!\otimes m_{\bs X^{I^N_j}_t} + \sum_{\substack{1 \leq k \leq \ell \\ k \neq j}} \theta_k \delta_{\lambda^k} \!\otimes m_t^{\lambda^k} \Big)^2;
\end{multlined}
\end{equation}
on the other hand, recalling the definition of $\bs\lambda_N$ and then controlling the Wasserstein distance with the total variation (cf., e.g., \cite[Theorem~6.15]{villaniOTON}), one easily gets
\[
W_2\Big(m_{(\bs\lambda_N:\tbs X^N_t)},\, \sum_{1 \leq k \leq \ell} \theta_k \delta_{\lambda^k} \!\otimes \frac1{\#I^N_k} \sum_{i \in I_k^N} \delta_{\tilde X^{N,i}_t} \Big)^2 \leq 2 \sup_{1 \leq i \leq N} | \tilde X^{N,i}_t |^2  \sum_{1 \leq k \leq \ell} \Big\lvert \frac{\#I^N_k}{N} - \theta_k \Big\rvert \eqdef \gamma^N_t,
\]
where $\bb E\bigl[ \sup_{[0,T]} \gamma^N \bigr] \to 0$ as $N \to \infty$ thanks to the uniform boundedness of $D_x \frk u_N$.
Then applying the triangle inequality to the right-hand side of \eqref{ENS_estBfinit} we see that it is controlled by
\begin{equation} \label{ENS_prfinitabove}
W_2\Big(\sum_{1 \leq k \leq \ell} \theta_k\delta_{\lambda^k} \!\otimes  \frac1{\#I^N_k} \sum_{i \in I_k^N} \delta_{\tilde X^{N,i}_t}, \ \theta_j \delta_{\lambda^j} \!\otimes m_{\bs X^{I^N_j}_t} + \sum_{\substack{1 \leq k \leq \ell \\ k \neq j}} \theta_k \delta_{\lambda^k} \!\otimes m_t^{\lambda^k} \Big)^2 + \gamma^N_t.
\end{equation} 
Now, by the convexity of $W_2^2$ (cf., e.g., \cite[Theorem~4.8]{villaniOTON}), %,\footnote{We are using the following property: given $\mu_j,\nu_j \in \Pc_p(\call X)$ for $j = 1,2$, optimal couplings $\pi_j$ of $(\mu_j,\nu_j)$ and $s \in (0,1)$, then $\varpi \defeq s\pi_1 + (1-s)\pi_2$ is a coupling of $(s\mu_1 + (1-s)\mu_2, s\nu_1 + (1-s)\nu_2)$, thus
%\[
%W_p(s\mu_1 + (1-s)\mu_2, s\nu_1 + (1-s)\nu_2)^p \leq \int_{\call X \times \call X} |x-y|^p \,\varpi(\di x,\di y) = sW_p(\mu_1,\nu_1)^p + (1-s)W_p(\mu_2,\nu_2)^p.
%\]}
the first term in \eqref{ENS_prfinitabove} is in turn controlled by
\[ \begin{multlined}[.95\displaywidth]
\theta_j W_2\Big( \frac1{\#I^N_j} \sum_{i \in I_j^N} \delta_{\tilde X^{N,i}_t}, \ m_{\bs X^{I^N_j}_t} \Big)^2 + \sum_{\substack{1 \leq k \leq \ell \\ k \neq j}} \theta_k W_2 \Big(\, \frac1{\#I^N_k} \sum_{i \in I_k^N} \delta_{\tilde X^{N,i}_t},\, m^{\lambda^k}_t \Big)^2
\\
\leq 2 \sum_{1 \leq k \leq \ell} \frac{1}{\#I_k^N} \sum_{i \in I_k^N} \big|\tilde X^{N,i}_t - X^{I^N_k,i}_t\big|^2 + 2 \sum_{1 \leq k \leq \ell} W_2(m_{\bs X^{I_k^N}_t}, m^{\lambda^k}_t)^2
\end{multlined}\]
Therefore, from \eqref{ENS_XNXINj} we have
\begin{equation} \label{ENS_firstpin}
\big|\tilde X^{N,i}_s - X^{I^N_j,i}_s\big|^2
\lesssim \int_0^s \bigg( \sum_{1 \leq k \leq \ell} \frac{1}{\#I_k^N} \sum_{i \in I_k^N} \big|\tilde X^{N,i}_t - X^{I^N_k,i}_t\big|^2 + \delta^{N,j}_t \bigg)\, \di t,
\end{equation}
where
\[ \begin{multlined}[.95\displaywidth]
\delta^{N,j}_t \defeq \big| D_x \frk u_N(\lambda^j,{\bs\lambda_N^{-i}},t, X^{I^N_j,i}_t,\tbs X^{N,-i}_t) - D_x U(\lambda^j,t,X^{I^N_j,i}_t, m_{(\bs\lambda^N,\tbs X^N_t)}) \big|^2 + \gamma^N_t \\
+ \sum_{1 \leq k \leq \ell} W_2(m_{\bs X^{I_k^N}_t}, m^{\lambda^k}_t)^2
\end{multlined} \]
is such that $\bb E \big[ \sup_{[0,T]} \delta^{N,j}_t \big] \to 0$ as $N \to \infty$ (cf.~\Cref{ENS_rmkconvdiuil}, property~\eqref{ENS_alsohavefromstep1} and see the last argument in the proof \Cref{ENS_propfacileL1}).
Averaging \eqref{ENS_firstpin} over $i \in I_j$ and summing over $j \in \{1,\dots,\ell\}$ we have
\[
\sum_{1 \leq j \leq \ell} \frac{1}{\#I_j^N} \sum_{i \in I_j^N} \big|\tilde X^{N,i}_s - X^{I^N_j,i}_s\big|^2
\lesssim \int_0^s \bigg( \ell \sum_{1 \leq j \leq \ell} \frac{1}{\#I_j^N} \sum_{i \in I_k^N} \big|\tilde X^{N,i}_t - X^{I^N_j,i}_t\big|^2 + \sum_{1 \leq j \leq \ell} \delta^{N,j}_t \bigg)\, \di t,
\]
whence
\[
\sum_{1 \leq j \leq \ell} \frac{1}{\#I_j^N} \sum_{i \in I_j^N} \bb E \Bigl[\, \sup_{t \in [0,T]} \big|\tilde X^{N,i}_t - X^{I^N_j,i}_t \big|^2 \Bigr] \to 0 \quad \text{as} \ N \to \infty.
\]
by Gronwall's lemma; then, plugging this back into \eqref{ENS_firstpin}, we obtain \eqref{ENS_PCfinit}.
\end{proof}

\begin{proof}[Proof of \Cref{ENS_convfinita}]
By \Cref{ENS_convfinitaapprox} it suffices to prove the theorem by assuming that the first marginal is of the form $\rho = \sum_{1 \leq j \leq \ell} \theta_j \delta_{\lambda^j} \in \Pc(\Lambda)$ for some (fixed) $\ell \in \N$, $\theta_j \in \bb Q_+$ and $\lambda^j \in \Lambda$.

We now want to construct $N$-player games, where each player is suitably labelled with $\lambda^j$, $j \in \{1,\dots,\ell\}$. The distribution of labels in the limit $N \to \infty$ must be $\rho$. It is elementary to construct a partition $(I_j)_{1 \leq j \leq \ell}$ of $\N$ such that $\frk d(I_j) = \theta_j$, where $\frk d$ denotes the natural density;\footnote{Recall that the natural density of $S \subset \N$ is defined as $\frk d(S) \defeq \lim_{N \to \infty} N^{-1} \#(S \cap \{1,\dots,N\})$, if the limit exists.} indeed, it suffices to consider the minimum natural number $D$ such that $D \theta_j \in \N$ for all $j \in \{1,\dots,\ell\}$, take a partition $\{\Theta_j\}_{1 \leq j \leq \ell}$ of $\{1,\dots,D\}$ such that $\#\Theta_j = D\theta_j$ and set $I_j \defeq \Theta_j + D \N$. So we can define $\bbs \lambda \in \ell^\infty(\N;\Lambda)$ by letting $\bar\lambda^i \defeq \lambda^j$ if $i \in I_j$ and then $\bs\lambda_{N} \in \Lambda^N$ for any $N \in \N$ by taking the first $N$ coordinates of $\bbs\lambda$.

With the notation of \Cref{ENS_convfinitaproplem}, we want argue as in the proof of \Cref{ENS_convfacile}, with the replacements $\{1,\dots,N\} \rightsquigarrow I^N_j $ and $D_x U^\lambda(t,x,m) \rightsquigarrow D_x U^{\lambda^j}(t,x,(\bs m_t^{\bs\lambda^{-j}},m))$. This crucial intermediate step is an ``intra-population'' propagation of chaos. As noted at the beginning of the proof of \Cref{ENS_convfinitaproplem}, we can exploit \Cref{ENS_propfacileL1} and argue exactly as in the proof of \Cref{ENS_convfacile} (note that $\#I_j^N \to \infty$ as $N \to \infty$) to obtain, for $i = \min I^N_j$,
\begin{equation} \label{ENS_pclim1finit} \begin{multlined}[t][.9\displaywidth]  \begin{multlined}
 \bb E\bigg[ \int_0^T \Big( \,\frac12\big|D_x U^{\lambda^j}\big(t,X^{I_j^N,i}_t,(\bs m_t^{\bs\lambda^{-j}},m_{\bs X^{I_j^N}_t})\big)\big|^2 + f^{\lambda^j}\big( X^{I_j^N,i}_t,(\bs m_t^{\bs\lambda^{-j}},m_{\bs X^{I_j^N}_t})\big) \Big) \di t
 \vspace{-7pt}\\
 + g^{\lambda^j}\big(X^{I_j^N,i}_T,(\bs m_T^{\bs\lambda^{-j}},m_{\bs X^{I_j^N}_T})\big) \bigg]
\end{multlined}
\\ \xrightarrow[N \to \infty]{}\, \int_0^T \int_{\R^d} \Big( \,\frac12\big|D_x U^{\lambda^j}(t,\cdot,\bs m_t^{\bs\lambda})\big|^2 + f^{\lambda^j}(\cdot,\bs m_t^{\bs\lambda}) \Big) \di \bar m_t^{\lambda^j}\di t + \int_{\R^d} g^{\lambda^j}(\cdot, \bs m_T^{\bs\lambda}) \,\di \bar m^{\lambda^j}_T,
\end{multlined}
\end{equation}
where $\bs m^{\bs\lambda}$ and $\bbs m^{\bs\lambda}$ solve, respectively, \eqref{ENS_fokplfinit} and
\begin{equation*} %\label{ENS_fokplfinit'}
\begin{cases}
\de_t \bar m^{\lambda^j} - \Delta \bar m^{\lambda^j} - \mathrm{div}(D_x U^{\lambda^j}(t,x,\bs m^{\bs\lambda}_t) \bar m^{\lambda^j}) = 0, &  j \in \{1,\dots,\ell\} \\
\bbs m^{\bs\lambda}|_{t=\tau} = \bbs{\frk m},
 \end{cases}
\end{equation*}
both with initial time $\tau = 0$.

At this point consider the representation formula for the value functions solving the Nash system constructed above,\begin{equation} \label{ENS_preprffinit}
\bb E \, \frk u_N({\bs\lambda_N},0,\bs Z) = \bb E\biggl[\, \int_0^T \Big(\,\frac12\big| D_x \frk u_N({\bs\lambda_N},t,\tbs X^N_t) \big|^2 + f_N(\bar\lambda^1,\tbs X^N_t)\Big)\di t + g_N(\bar\lambda^1,\tbs X^N_T) \biggr],
\end{equation}
where $\tbs X^N$ is as in \Cref{ENS_convfinitaproplem}; in order to pass to the limit this identity, we need to compare $\tbs X^N$ and $\bs X^{I^N_\cdot}$, which we have thus already done in \Cref{ENS_convfinitaproplem}.
Having fixed $j \in \{1,\dots,\ell\}$, note that we can assume without loss of generality that $1 \in I_j$ (otherwise, just take a different partition $\{\Theta_j\}_{1 \leq j \leq \ell}$ at the beginning of Step 2); therefore, $\bar\lambda^1 = \lambda^j$, and one deduces \eqref{ENS_reprUlden} from \eqref{ENS_pclim1finit}, passing \eqref{ENS_preprffinit} to the limit. In order to motivate this last claim, for the sake of completeness, we prove that, for $h \in \{f,g\}$, one has
\[
\lim_{N \to \infty} \bb E \Big[ \sup_{t \in [0,T]} \big| h^{\lambda^j}\big( X^{I_j^N,i}_t, (\bs m_t^{\bs\lambda^{-j}},m_{\bs X^{I_j^N}_t})\big) - h_N(\lambda^j,\tbs X^N_t) \big| \Big] = 0
\]
(the other analogous convergence that we will need, involving $D_x \frk u_N$ and $U$, can be proved in a similar manner). Under our latest assumption, $X^{I^N_j,i} = X^{I^N_j,1}$, so, by the Lipschitz continuity of $h$ and the triangle inequality, 
\[ \begin{multlined}[.95\displaywidth]
 \big| h^{\lambda^j}\big( X^{I_j^N,i}_t,(\bs m_t^{\bs\lambda^{-j}},m_{\bs X^{I_j^N}_t})\big) - h_N(\lambda^j,\tbs X^N_t) \big|^2 \\
 \begin{multlined}[.95\displaywidth]
\lesssim \bigl|X^{I^N_j,1}_t - \tilde X^{N,1}_t\bigr|^2 + W_2\Bigl(\theta_j m_{\bs X^{I_j^N}_t} + \sum_{\substack{1 \leq k \leq \ell \\ k \neq j}} \theta_k m^{\lambda^k}_t, \, m_{\tbs X^{N}_t}\Bigr)^2 \\
+ W_2(m_{\tbs X^{N}_t}, m_{\tbs X^{N,-1}_t})^2 + \bigl| h^{\lambda^j}(\tilde X^{N,1}_t, m_{\tbs X^{N,-1}_t}) - h_N(\lambda^j,\tbs X^N_t) \bigr|^2;
\end{multlined}
\end{multlined}\]
it is straightforward to show that, after taking the supremum over $t \in [0,T]$ and then the expectation, the third term in the left-hand side vanished as $N \to \infty$, while the first one and the last one go to $0$ as well thanks to \eqref{ENS_PCfinit} and \Cref{ENS_convhN}, respectively. As for the second term, by the triangle inequality and convexity of $W_2^2$,
\[ \begin{multlined}[.95\displaywidth]
W_1\Bigl(\theta_j m_{\bs X^{I_j^N}_t} + \sum_{\substack{1 \leq k \leq \ell \\ k \neq j}} \theta_k m^{\lambda^k}_t, \, m_{\tbs X^{N}_t}\Bigr) \\
\begin{multlined}[.95\displaywidth]
\leq \theta_j W_2\Bigl(m_{\bs X^{I_j^N}_t}, \frac1{\#I^N_j} \sum_{i \in I^N_j} \delta_{\tilde X^{N,i}_t} \Bigr)^2 + \sum_{\substack{1 \leq k \leq \ell \\ k \neq j}} \theta_k W_2\Bigl( m^{\lambda^k}_t, \, \frac1{\#I^N_k} \sum_{i \in I^N_k} \delta_{\tilde X^{N,i}_t} \Bigr)^2 \\[-3pt]
+ W_2\Bigl( \sum_{1 \leq k \leq \ell} \frac{\theta_k}{\#I^N_k} \sum_{i \in I^N_k} \delta_{\tilde X^{N,i}_t}, \, m_{\tbs X^{N}_t}\Bigr)^2;
\end{multlined}
\end{multlined}
\]
by using the estimates obtained in the previous lemmata, the expectation of the supremum over $t \in [0,T]$ of the right-hand side can be shown to vanish as $N \to \infty$.
\end{proof}

\subsection{A continuum of MFG systems}

Some considerations about the characterisation provided by \eqref{ENS_reprUlden}--\eqref{ENS_FPLambdaC}--\eqref{ENS_FPLambdaC'} are now in order. 

\begin{rmk}
If $U^\lambda$ is sufficiently regular on $[0,T] \times \R^d \times \Pc_2(\Lambda \times \R^d)$ (e.g., according to Definition~2.4.1 in \cite{CDLL} of a classical solution to the first-order Master Equation, so, in particular it is differentiable with respect to $\mu$), one can differentiate $t \mapsto \int_{\R^d} U^\lambda(t,\var,\mu_t)\,\di m^\lambda_t$ with $m_\tau^\lambda = \delta_x$ and use \eqref{ENS_reprUlden}, \eqref{ENS_FPLambdaC}, \eqref{ENS_FPLambdaC'} to get, after standard computations,
\[ \begin{multlined}[.95\displaywidth]
-\de_t U^\lambda(t,x,\bs\mu) - \Delta_x U^\lambda(t,x,\bs\mu) - \int_{\Lambda \times \R^d} \Delta_y \frac{\delta}{\delta \bs\mu} U^\lambda(t,x,\bs\mu,\lambda',y)\,\bs\mu(\di\lambda',\di y) \\
+ \frac12\bigl| D_x U^\lambda(t,x,\bs\mu)\bigr|^2 
+ \int_{\Lambda \times \R^d} D_y \frac{\delta}{\delta \bs\mu} U^\lambda(t,x,\bs\mu,\lambda',y) \cdot \de_y U^{\lambda'}(t,y,\bs\mu) \, \bs\mu(\di \lambda',\di y) = f^\lambda(x,\pi_{\R^d}\mpush\bs\mu).
\end{multlined}\]
This is exactly the form one would expect a Master Equation associated to our non-symmetric game to exhibit. In this sense, the limit game \eqref{ENS_reprUlden}--\eqref{ENS_FPLambdaC}--\eqref{ENS_FPLambdaC'} can be considered as a \emph{weak formulation of the Master Equation for a nonsymmetric game} parametrised by $\lambda \in \Lambda$. Since $U^\lambda$ is Lipschitz continuous with respect to the measure variable, this formulation is very close to the one proposed in \cite{bertuccio}. Moreover, it brings naturally a continuum of standard MFG systems, as described below.
\end{rmk}

\begin{thm} \label{ENS_MFGthm}
With the notation of \Cref{ENS_convfinita}, define $u^\lambda(t,x) \defeq U^\lambda(t,x,\mu_t)$. Then $u$ solves the following (generalised) Mean Field system on $(0,T) \times \Lambda \times \R^d$:
\begin{equation} \label{ENS_MFS}
\begin{dcases}
- \de_t u^\lambda - \Delta_x u^\lambda + \frac12 |D_xu^\lambda|^2 = f^\lambda(x,\pi_{\R^d}\mpush\mu_t) \\
\de_t \mu - \Delta_x \mu - \mathrm{div}_x (Du^\lambda \mu) = 0 \\
u^\lambda(T,\cdot) = g^\lambda(\cdot,\pi_{\R^d}\mpush\mu_T), \quad \mu_0 = \bs\mu,
\end{dcases}
\end{equation}
where both the Hamilton--Jacobi and the Fokker--Planck equations are satisfied in the classical sense.
\end{thm}

\begin{proof}
Fix $\lambda \in \Lambda$. Note that, by \eqref{ENS_reprUlden}, for every $\tau$ and smooth $\bar{\frk m}^\lambda$,
\begin{equation}\label{ufrkmeq}
\int_{\R^d} u^\lambda(\tau,\cdot)  \,\di \bar{\frk m}^\lambda
= \int_\tau^T \int_{\R^d} \Big( \,\frac12\big|D_x u^\lambda(s,\cdot)\big|^2 + f^{\lambda}(\var, {\pi_{\R^d}}\mpush \mu_s) \Big) \,\di \bar m^{\lambda}_s\di s + \int_{\R^d} g^{\lambda}(\var, {\pi_{\R^d}}\mpush \mu_T) \,\di \bar m^{\lambda}_T.
\end{equation}
Let $v$ be solution of the linear backward Cauchy problem
\[
\begin{dcases}
- \de_t v - \Delta_x v + D_xu^\lambda \cdot D_x v = f^\lambda(x,\pi_{\R^d}\mpush\mu_t) + \frac12\big|D_x u^\lambda \big|^2\\
v(T,\cdot) = g^\lambda(\cdot,\pi_{\R^d}\mpush\mu_T).
\end{dcases}
\]
By standard linear theory \cite{LSU}, $v$ solves the equation a.e., and since coefficients are in $L^\infty$, it is globally bounded and locally in $W^{2,1}_p$ for every $p \in (1,\infty)$. By duality between the equation for $v$ and the equation for $\bar m^\lambda$, for every $\tau \in [0,T]$ we have
\[
\int_{\R^d} v(\tau,\cdot)  \,\di \bar{\frk m}^\lambda
= \int_\tau^T \int_{\R^d} \Big( \,\frac12\big|D_x u^\lambda(s,\cdot)\big|^2 + f^{\lambda}(\var, {\pi_{\R^d}}\mpush \mu_s) \Big) \,\di \bar m^{\lambda}_s\di s + \int_{\R^d} g^{\lambda}(\var, {\pi_{\R^d}}\mpush \mu_T) \,\di \bar m^{\lambda}_T.
\]
Since the previous equation has the same right-hand side of \eqref{ufrkmeq}, by the arbitrariness of $\bar{\frk m}^\lambda$ and $\tau$, we obtain $v \equiv u^\lambda$. Plugging back $D_x v  \equiv D_x u^\lambda$ into the equation for $v$ we obtain the assertion. That $u^\lambda$ is a classical solution follows by standard parabolic regularity.
%By \eqref{ENS_reprUlden}, the continuity and boundedness of $D_x U$ and the continuity of $\mu$ and $\bar m^\lambda$ with respect to $t$, one deduces that $U$ (i.e., $u$) is differentiable with respect to $t$. Then, choosing $\bar m^\lambda_\tau = \delta_x$, by \eqref{ENS_reprUlden}
%\[
%\frac{\di}{\di t}\biggr|_{t=\tau} \int_{\R^d} u^\lambda(t,\cdot) \,\bar m^\lambda_t = - \frac12 |Du^\lambda(\tau,x)|^2 - f^\lambda(x,\pi_{\R^d}\mpush\mu_\tau),
%\]
%while by \eqref{ENS_FPLambdaC'}
%\[
%\frac{\di}{\di t}\biggr|_{t=\tau} \int_{\R^d} u^\lambda(t,\cdot) \,\bar m^\lambda_t = \de_t u^\lambda(\tau,x) + \Delta u^\lambda(\tau,x) - |Du^\lambda(\tau,x)|^2.
%\]
%Note that Laplacian is a priori well-defined for almost every $x$ by the Lipschitz continuity of $Du^\lambda$, yet by comparing with the other terms in the identities above one can conclude that it is in fact everywhere defined. Along with \eqref{ENS_FPLambdaC} (with $\tau = 0$ and integrated with respect to $\rho$) we obtain system~\eqref{ENS_MFS}.
\end{proof}

\begin{rmk}
If $f$ and $g$ are constant in $\lambda$, then \eqref{ENS_MFS} is the classic Mean Field system (this case corresponds to the situation in \Cref{ENS_convfacile}). We also mention that, as a byproduct, under assumptions of monotonicity, one can then deduce uniqueness of the limit function $U$ by classic uniqueness results for the Mean Field system.
\end{rmk}

\section{A joint vanishing viscosity and large population limit} \label{sec_vvl}

As a last application of our a priori estimates, we discuss, in the \Ds semimonotone case, the problem of describing the limit when $N \to \infty$ and, \emph{simultaneously}, the idiosyncratic diffusion $\sigma = \sigma_N$ vanishes (while $\beta=0$, as in the previous section).

For the sake of simplicity, and brevity, we are going to discuss such a problem in the illustrative case of purely quadratic Hamiltonians, that is in the setting of symmetric (classical) Mean Field interactions of \Cref{rmkonmfg}; nevertheless, our results can be extended to the more general ``multi-population'' case discussed in the previous section, or to more general Hamiltonians in light of the results presented in \Cref{sec_genH} below.%and following the additional steps presented in the previous section to prove \Cref{ENS_convfinita}.

The first important point to note to deal with vanishing viscosity comes from the proof of \Cref{ENS_dscprop}, where one can see that we still have propagation of the semimonotonicity constant $M^*$ (which is in fact independent of $\sigma$) provided that, in equation \eqref{stareq}, the term
\[
\frac{\bar C}{\sqrt N} = \frac{\bar C(M^*,\sigma)}{\sqrt N} 
\]
vanishes as $N \to \infty$. In \Cref{sec_esti,sec_Ds}, we carefully keep track of the dependence on $\sigma$ of all the relevant constants contributing to $\bar C$, showing that
\[
\bar C =e^{e^{c(1+\sigma^{-1})}},
\]
with $c$ independent of $\sigma$, so this gives us a threshold on how fast we can allow $\sigma$ to deteriorate in the large population limit; namely, we require that
\begin{equation} \label{sNthreshold}
 \liminf_{N \to \infty} \sigma_N \log\log N \geq \theta_* > 0
\end{equation}
for some $\theta_* = \theta_*(L_g,L_f,T)$ sufficiently large (cf.~\Cref{rmksusn}).

Consequently, in such a regime, we have the important estimates \eqref{ENS_recapE1}--\eqref{ENS_recapE2} recalled at the beginning of \Cref{sec_lpl}. The second important point to note is that the implied constant in the first estimate of \eqref{ENS_recapE1} does not in fact depend on $\sigma$. On the other hand, the other bounds in \eqref{ENS_recapE1}--\eqref{ENS_recapE2} may a priori deteriorate as $\sigma$ vanishes (as they are obtained using heavily \Cref{ENS_estv}).
%so as $\sigma \to 0$ we lose the sharp rate of $1/N$ needed to apply \Cref{ENS_DtoW}, and thus \Cref{ENS_lemcarconv}, in order to obtain convergence of the derivatives of $D_i u^i$ to the derivative $D_x U_0$ of a prospective limit value function $U_0$ (as well as the Lipschitz continuity of $(x,m) \mapsto D_x U(t,x,m)$). 
 
To recover the uniform regularity in $\sigma$ a posteriori that guarantees the full  Lipschitz continuity of the limit $(x,m) \mapsto U_0(t,x,m), D_x  U_0(t,x,m)$, we need in particular to apply the weighted Lipschitz estimate of \Cref{ENS_estd1} to $D_i u^i$. To perform this step, we will use not only the fact that the standing assumptions on $F$ and $G$ guarantee that the {\bfseries(\hyperlink{MF}{MF})} assumption holds, but also that the following properties hold true:
%On the other hand, bounds like \eqref{ENS_recapE1}--\eqref{ENS_recapE2}, uniform in $\sigma$, are the only piece of information we need in order to have a convergence result analogous to \Cref{ENS_rmkconvdiuil} when $\sigma \to 0$ according to the regime \eqref{sNthreshold}. Indeed, note that the only additional result we use in the proof is \Cref{lemmaholder}, which yields a constant that does not deteriorate as $\sigma$ vanishes.
%
%
%Even though we will thoroughly discuss the vanishing viscosity limit only in the classical Mean Field setting, it is worth mentioning that, in a Mean-Field-like regime, the price we pay to have such a sharper control is to ask for additional decay estimates for the data (all reasonable in a Mean-Field-like setting, and trivial in the classic Mean Field one). More precisely, the resulting strengthened Mean-Field-like assumption reads as follows:
%\begin{itemize}[align=left,topsep=7pt]
%\item[]
for $h \in \{f,g\}$, there exists $L_h > 0$, independent of $N$, such that
\begin{gather} \label{extraassMF0}
\sup_{1 \leq i \leq N} \sum_{\substack{1 \leq k \leq N \\ k \neq i}} \norm{D^2_{ki} f^i}_\infty^2 \leq \frac{L_f^2}{N},
\\ \label{extraassMF0'}
 \sup_{\substack{1 \leq i,k \leq N \\ k \neq i}} \norm{D_k(Dg^i)}_\infty \leq \frac{L_g}{N}, \qquad \sup_{1 \leq i \leq N} \sum_{\substack{1 \leq k \leq N \\ k \neq i}} \norm{D(D^2_{ki} g^i)}_\infty^2 \leq \frac{L_g^2}{N}.
\end{gather}
%\end{itemize}
% Clearly these bounds imply some of the others in assumption {\bfseries(\hyperlink{MF}{MF})}, but we leave to the reader to write down the minimal collection of bounds needed to have {\bfseries(\hyperlink{MF'}{$\bs{\overline{\text{MF}}}$})} satisfied.

We can now obtain the estimates we need to pass to the limit.

%\footnote{This smallness condition comes from the propagation of semimonotonicity; for instance, we showed that $TM_g + T^2M_f < \frac1{6e}$ works.}
\begin{prop} \label{prop_stabestwrts}
Let the semimonotonicity constants of $f$ and $g$ be small enough, and $N$ large as in Proposition \ref{ENS_dscprop}. Then, there exists $\theta_*>0$ (depending only on $L_g$, $L_f$, and $T$) such that, if condition~\eqref{sNthreshold} holds, then estimates \eqref{ENS_recapE1}--\eqref{ENS_recapE2} hold with implied constants independent of $\sigma = \sigma_N$.
\end{prop}

\begin{proof}
We have already discussed the fact that, due to \Cref{rmksusn}, the condition \eqref{sNthreshold} (with $\theta_*$ large enough) allows us to control the semimonotonicity of $u$ uniformly in $\sigma$, and that the first implied constant in \eqref{ENS_recapE1} is also independent of $\sigma$ (By \Cref{ENS_estd1}). Hence we only need to prove that the implied constants in the second estimate of \eqref{ENS_recapE1} and in those of \eqref{ENS_recapE2} do not depend on $\sigma$ as well. To this aim, it is sufficient to show that we have a weighted Lipschitz bound of the form
\begin{equation} \label{lipboundondiui}
|D_iu^i(t,\bs x) - D_iu^i(t,\bs y)| \leq \bar c_1 \norm{\bs x - \bs y}^i,
\end{equation}
with $\bar c_1$ independent of $\sigma$ and $N$. Indeed, if \eqref{lipboundondiui} holds, then, with implied constants independent of $\sigma$, we immediately obtain 
\[
\norm{D^2_{ii}u^i}_\infty + N\, \biggl\lVert\, \sum_{j \neq i} |D^2_{ji} u^i|^2 \biggr\rVert_\infty \lesssim 1
\]
(by arguing as in the proof of \Cref{ENS_uijest}), as well as
\[
\sup_{j \neq i} \, \norm{D_j u^i}_\infty \lesssim \frac1N
\] 
by following the proof of \Cref{ENS_uijestp}.

On the other hand, \eqref{lipboundondiui} directly follows from \Cref{ENS_estd1} applied to $D_i u^i$ (component-wise) in place of $u^i$; therefore, it remains to show that the hypotheses of \Cref{ENS_estd1} holds for the equation for $D_i u^i$. Among them, the only one to be checked is that \eqref{fgtildeass} holds for 
\[
h^i \defeq D_i f^i - \sum_{j \neq i} D^2_{ij} u^j D_j u^i;
\]
so, in light of the assumption on $F$ (which guarantees that \eqref{extraassMF0} holds) we only need to prove that
\begin{equation} \label{toproveforinds}
\biggl\lvert \, \sum_{j \neq i} \bigl( D^2_{ij} u^j D_j u^i \bigr)\Bigr|^{\bs x}_{\bs y} \biggr\rvert \lesssim \norm{\bs x - \bs y}^i,
\end{equation}
uniformly in $N$ (with $\sigma = \sigma_N$). By the fundamental theorem of calculus and the Cauchy--Schwarz inequality,
\begin{equation} \label{rhsabove--}
\begin{multlined}[t][.9\displaywidth]
\biggl\lvert \, \sum_{j \neq i} \bigl( D^2_{ij} u^j D_j u^i \bigr)\Bigr|^{\bs x}_{\bs y} \biggr\rvert^2 
\\
 \begin{split}
&\leq \biggl\lVert \sum_{j \neq i} |D_j u^i|^2 \biggr\rVert_\infty \biggl(\, \biggl\lVert \sum_{j \neq i} |D^3_{iij} u^j|^2 \biggr\rVert_\infty |x^i - y^i|^2 + \biggl\lVert \sum_{\substack{jk \\ j \neq i \neq k}} |D^3_{ikj} u^j|^2 \biggr\rVert_\infty \sum_{k \neq j} |x^k - y^k|^2 \biggr)
\\
&\quad + \biggl\lVert \sum_{j \neq i} |D^2_{ij} u^j|^2 \biggr\rVert_\infty \biggl(\, \biggl\lVert \sum_{j \neq i}\, |D^2_{ji} u^i|^2 \biggr\rVert_\infty |x^i - y^i|^2 + \biggl\lVert \sum_{\substack{jk \\ j \neq i \neq k}} |D^2_{jk} u^i|^2 \biggr\rVert_\infty \sum_{k \neq j} |x^k - y^k|^2 \biggr).
\end{split} 
\end{multlined}
\end{equation}
By \Cref{ENS_d2N,ENS_d2N'} the last line above is controlled by 
\[
\frac{C(\sigma)}{N} \bigl( \norm{\bs x - \bs y}^i \bigr)^2, \qquad \text{with} \ \ \log \log C(\sigma) \lesssim 1 + \sigma^{-1},
\]
so \eqref{sNthreshold} ensures that $\limsup_{N \to \infty} C(\sigma_N)/N < \infty$. Therefore, in order to have \eqref{toproveforinds}, it remains to show that the first part of the right-hand side of \eqref{rhsabove--} is bounded in a similar fashion; also, in light of the estimate of \Cref{ENS_uijest}, this boils down to proving that
\begin{equation} \label{lastdesttoprove}
\sup_h \sum_{i \neq h}\, \norm{D(D^2_{hi} u^i)}_\infty^2 \leq \frac{\hat C(\sigma)}{N}, \qquad \text{with} \ \ \log \log \hat C(\sigma) \lesssim 1+\sigma^{-1}.
\end{equation}
It is now sufficient to recall that \eqref{lastdesttoprove} has been proven to hold under the assumptions in \eqref{extraassMF0'} in \Cref{ENS_d3N}.
%We proceed as in the proof of \Cref{ENS_d3b}, but we only sum over $i \neq h$, thus getting, for $s \in [0,T)$,
%\[
%\sup_{t \in [T-s,T]} \sum_{i \neq h} \,\norm{D(D_{hi}^2 u^i(t,\var)}_\infty^2 
%\lesssim_{\,\sigma} \frac1N + \sum_{i \neq h} \,\norm{D_{hi}^2 u^i}_\infty^2 + \frac{s}{N} \sup_{t \in [T-s,T]} \sum_{\substack{ij \\ j \neq i \neq h}} \norm{D^3_{hij} u^j(t,\var)}_\infty^2,
%\]
%using \Cref{ENS_d2N,ENS_d2N'} as well as assumption {\bfseries(\hyperlink{MF'}{$\bs{\overline{\text{MF}}}$})}.\noticina{This needs to be changed as there is no longer assumption $\overline{\text{MF}}$. Also below.} Now we note that, also using \Cref{ENS_d3b},
%\[
%\frac{1}{N} \sum_{\substack{ij \\ j \neq i \neq h}} \norm{D^3_{hij} u^j(t,\var)}_\infty^2 \lesssim_{\,\sigma} \frac1N + \sum_{i \neq h} \,\norm{D(D_{hi}^2 u^i(t,\var)}_\infty^2.
%\]
%On the other hand, arguing as in the proof of \Cref{ENS_d2N'}, but without summing over $k \neq i$, one immediately obtains (under assumption {\bfseries(\hyperlink{MF'}{$\bs{\overline{\text{MF}}}$})})
%\[
%\sup_{\substack{ik \\ i \neq k}} \,\norm{D_k(Du^i)}_\infty \lesssim_{\,\sigma} \frac1N,
%\]
%whence in particular
%\[
%\sum_{i \neq h} \,\norm{D_{hi}^2 u^i}_\infty^2 \lesssim_{\,\sigma} \frac1N.
%\]
%We can now argue as in the end of the proof of \Cref{ENS_d3b} to obtain \eqref{lastdesttoprove}.
\end{proof}

We can now proceed with the proof of \Cref{ENS_rmkconvdiuilVV}.
\begin{proof}[Proof of \Cref{ENS_rmkconvdiuilVV}] The proof of the first part, on the existence of $U_0$, is exactly the same as the one of \Cref{ENS_rmkconvdiuil}, in light of the estimates of \Cref{prop_stabestwrts}. 

The proof of the second part, on the characterization of $U_0$, is very similar to the one of \Cref{ENS_convfacile}; one has just to be a bit careful with the comparison of optimal trajectories in a situation where $\sigma$ vanishes. With $\bs Z$ as in \Cref{ENS_propfacileL1}, let $\bs{\hat Y}^N$ and $\bs Y^N$ solve, respectively,
\[
\begin{cases}
\di \hat Y^{N,i} = - D_x U_0(t,\hat Y^{N,i}_t,\Law_{\hat Y^{N,i}_t})\,\di t + \sqrt{2\sigma_N}\, \di B^i_t, & i \in \{2,\dots,N\} \\
\di \hat Y^{N,1} = - D_x U_0(t,\hat Y^{N,1}_t,\Law_{\hat Y^{N,2}_t})\,\di t + \sqrt{2\sigma_N}\, \di B^1_t \\
\bs{\hat Y}^N_0 = \bs Z
\end{cases}
\]
and
\[
\begin{cases}
\di Y^{N,i} = - D_x U_0(t,Y^{N,i}_t,\Law_{Y^{N,i}_t})\,\di t , & i \in \{2,\dots,N\} \\
\di Y^{N,1} = - D_x U_0(t,Y^{N,1}_t,\Law_{Y^{N,2}_t})\,\di t \\
\bs Y^N_0 = \bs Z.
\end{cases}
\]
For $i \geq 2$, let $\hat m_t \defeq \Law_{\hat Y^{N,i}_t}$ and $m_t \defeq \Law_{Y^{N,i}_t}$ (note that they are in fact the same for all $i \geq 2$ and $N$). We only need to prove that
\begin{equation} \label{ENS_estT1.3VV}
\lim_{N \to \infty} \sup_{1 \leq i \leq N} \bb E\Big[ \sup_{t \in [0,T]} \big| \hat Y^{N,i}_t -  Y^{N,i}_t \big|^2 \Big] = 0,
\end{equation}
which also yields, in particular,
\begin{equation*}%\label{convservedopoVV}
\lim_{N \to \infty} \bb E \Bigl[\, \sup_{t \in [0,T]} W_2(\hat m_t,  m_t) \Bigr] = 0.
\end{equation*}
Indeed, we can estimate
\[\begin{multlined}[.95\displaywidth]
\bigg| \big|D_x U_0(t,\var)\big|^2\Big|_{(X^{N,i}_t,m_{\bs X^N_t})}^{( Y^{N,i}_t, m_t)} \bigg| + \bigg| F \Big|_{(X^{N,i}_t, m_{\bs X^N_t})}^{( Y^{N,i}_t,  m_t)} \bigg| + \bigg| G \Big|_{(X^{N,i}_t, m_{\bs X^N_t})}^{( Y^{N,i}_t, \hat m_t)} \bigg| \vspace{3pt}\\
\lesssim \big|X^{N,i}_t - \hat Y^{N,i}_t\big| + \big|\hat Y^{N,i}_t -  Y^{N,i}_t\big| + W_2(m_{\bs X^N_t},\hat m_t) + W_2(\hat m_t, m_t),
\end{multlined}\]
and proceed as in the proof of \Cref{ENS_convfacile}, with $\bar m \defeq \Law_{Y^{N,1}}$, also noting that \cite[Theorem~1]{FG13} provides a vanishing rate that is uniform for $\sigma_N \to 0$.

To prove \eqref{ENS_estT1.3VV}, note that by H\"{o}lder's, Young's, and Doob's inequalities, and the Lipschitz continuity of $D_x U_0$, one has (with implied constants independent of $N$), for any $\tau \in (0,T]$,
\[\begin{split}
\bb E \Bigl[ \,\sup_{t \in [0,\tau]} \bigl| Y^{N,i}_t - \hat Y^{N,i}_t \bigr|^2 \Bigr] &\lesssim  \bb E \biggl[\, \int_0^\tau \bigl|Y^{N,i}_s - \hat Y^{N,i}_s \bigr|^2 \,\di s \biggr] + \int_0^\tau W_2(\Law_{Y^{N,i}_s}, \Law_{\hat Y^{N,i}_s})^2 \,\di s + \sigma_N
\\
&\lesssim \int_0^\tau  \bb E \Bigl[\, \sup_{t \in [0,s]} \bigl|Y^{N,i}_t - \hat Y^{N,i}_t \bigr|^2 \Bigr] \,\di s + \sigma_N,
\end{split}
\]
so Gronwall's lemma, along with the fact that $\sigma_N \to 0$, gives \eqref{ENS_estT1.3VV}.
\end{proof}

%\end{proof}
%%% prima parte di thm 3.17

%\begin{rmk}
%We called this a \emph{vanishing viscosity version} since it is clearly relevant compared to \Cref{ENS_rmkconvdiuil} when $\sigma_N \to 0$ as $N \to \infty$. Nevertheless it is worth noting that this is naturally not required as an hypothesis, since it is not necessary to have \eqref{sNthreshold}.
%\end{rmk}

%Finally, we are able to characterise such a limit function $U_0$, as done in \Cref{ENS_convfacile}. The limiting functions $F,G$ that appear in below are (according to whether $f^i,g^i$ are defined as in \Cref{whenlconst} or as in \Cref{rmkonmfg}) either those obtained as in \Cref{ENS_convhN} or directly those given in \Cref{rmkonmfg}.

%% seconda parte thm 3.17

%\begin{proof}

\section{Estimates in the case of more general Hamiltonians} \label{sec_genH}

We finally discuss how to extend the estimates obtained in the previous sections to Hamiltonians of the form $H^i(x^i, p)$, and provide a proof of Theorem \ref{ENS_mainthm} under the general standing assumptions on $H^i$. We will focus on the estimates only, though their application to the large population limit developed in Sections \ref{sec_lpl} and \ref{sec_vvl} follows similar lines.

We start with the extension of the a priori estimates of \Cref{sec_esti}, for which only assumption {\bfseries(\hyperlink{H0}{H0})} is required. As in the previous sections, the solution $u$ to the Nash system is assumed to be classical, $C^1$ in time and $C^4$ in space, with bounded derivatives. Moreover, we require second order derivatives in space to be uniformly continuous.

\begin{prop}[Estimates on the derivatives, assuming one-sided Lipschitz bound on the drift] \label{generealAPesti} 
Assume that {\bfseries(\hyperlink{MF}{MF})} and {\bfseries(\hyperlink{H0}{H0})} hold.
Let $\tau \in [0,T)$ and assume that \eqref{generalDmon} holds for $\bs x,\bs y \in (\R^d)^N$ and $t \in (\tau, T]$, with $\tilde M$ independent of $N$. Then the following estimates hold for all $t \in [\tau,T]$, with implied constants depending on $\tilde M$ but independent of $N$:
\begingroup
\allowdisplaybreaks
\begin{gather}
\label{generald1}
\sup_{\substack{1\leq i,j \leq N \\ j \neq i}} \norm{D_j u^i(t,\var)}_\infty \lesssim \frac{1}{N}, \\
\label{generald2}
\sup_{1\leq i \leq N} \sum_{1\leq j \leq N}\norm{D(D_i u^j)(t,\var)}_\infty^2 \lesssim 1, \\
\label{generald2'}
\sup_{1\leq i \leq N} \bigg\lVert \sum_{\substack{1\leq j\leq N \\ j \neq i}} \big| D^2_{ij} u^j(t,\var)\big|^2 \bigg\rVert_\infty \lesssim \frac{1}{N} , \\
\label{generald2''}
\sup_{1\leq i \leq N} \sum_{\substack{1\leq j \leq N \\ j \neq i}}\norm{D(D_j u^i)(t,\var)}_\infty^2 \lesssim \frac{1}{N} ,\\
\label{generald3}
\sum_{\substack{1\leq i,j \leq N \\ j\neq i}} \norm{D(D_{ij} u^j)(t,\var)}_\infty^2 \lesssim 1.
\end{gather}
\endgroup
\end{prop}

\begin{proof}
We need to prove that the estimates contained in \Cref{ENS_estd1,ENS_d2b,ENS_d2N,ENS_uijestp,ENS_d2N',ENS_d3b} all hold with a more general $H^i$ satisfying {\bfseries(\hyperlink{H0}{H0})}. The corresponding proofs proceed essentially in the same way as those for the purely quadratic Hamiltonian (with the resulting constants also depending on $C_H$), so we only underline the main differences and explaining how to handle them. 

We start by noticing that the main difference that is common to all the proofs will be that all sums involving the drift terms $D_j u^j$ now present the drift $D_p H^j(x^j,D_j u^j)$ instead; nevertheless, such terms are dealt with in the same way as before, just by invoking the assumption \eqref{generalDmon} instead of the \Ds semimonotonicity of $u$. We now go through all the proofs and explain the remaining changes that need to be performed.

\emph{Extending \Cref{ENS_estd1}.}
In the proof, replace $\frac12|D_iu^i(\bar t,\var)|^2 \bigr|^{\bs x}_{\bs y}$ with $H^i(\var^i,D_iu^i(\bar t,\var))\bigr|^{\bs x}_{\bs y}$; using \eqref{ENS_Dx0}--\eqref{ENS_Dy0} and {\bfseries(\hyperlink{H0}{H0})} one has
\[
\Bigl| H^i(\var^i,D_iu^i(\bar t,\var))\bigr|^{\bs x}_{\bs y} \Bigr| \lesssim |x^i-y^i| + \epsilon(1 + \psi(\bar t) + \epsilon) \phi(\bar t,\bar{\bs x},\bar{\bs y}),
\]
with implied constant independent of $N$. The rest of the proof is the same. Note that this also gives the extension to more general Hamiltonians of \Cref{ENS_uijest}, as that is an immediate consequence of \Cref{ENS_estd1}. 

\emph{Extending \Cref{ENS_d2b}.}
Use \Cref{ENS_estv} with
\[
F = D_{x^{k\ell}}f^i - \delta_{ki} D_{x^\ell} H^i(x^i,D_i u^i) - \sum_{j \neq i} D_{pp}^2 H^j(x^j,D_ju^j) D_j(D_{x^{k\ell}} u^j) \cdot D_j u^i
\]
and note that, using {\bfseries(\hyperlink{H0}{H0})} and \Cref{ENS_uijest} (just extended),
\[
\norm{F}_\infty^2 \lesssim \norm{D_k f^i}_\infty^2 + \delta_{ki} C_H^2 + \frac{C_H^2c_1^2}{N} \biggl\lVert\, \sum_{j \neq i} |D^2_{kj} u^j|^2 \biggr\rVert,
\]
with implied constant independent of $N$. The rest of the proof is the same, so we have \eqref{generald2}.

\emph{Extending \Cref{ENS_d2N}.}
We can identify three types of terms in the counterpart of the equation \eqref{eqforwi} for $w^i$. Terms of the first type are the same as in \eqref{eqforwi}, and correspond to the first two lines and the last one of the equation (after the aforementioned general replacement of $D_ju^j$ with $D_p H^j(x^j,D_j u^j)$). Terms of the second type correspond to the remaining (third and fourth) lines of \eqref{eqforwi} has they have the same structure, with the only difference being an additional term in each product, of the form $D^2_{pp} H^\ell(x^\ell, D_\ell u^\ell)$ (in the third line) or $D^2_{pp} H^i(x^i,D_iu^i)$ (in the fourth line). Terms of the third types are new ones, coming from the derivatives of the Hamiltonians with respect to $x$ and/or of order $3$, and they read as follows:
\begin{equation} \label{extratermsB}
\begin{gathered}
\frac1N D^2_{xx} H^i D^2_{ii} u^i \cdot D^2_{ii} u^i + \sum_{\ell \neq i} (D^2_{px} H^\ell + D^2_{px} H^i)  D^2_{i\ell} u^\ell \cdot D^2_{i\ell} u^\ell 
\\
{}+ \frac2N  D^2_{px} H^i D^2_{ii}u^i \cdot D^2_{ii}u^i + \sum_{\ell \neq i} D_{ppx}^3 H^i D^2_{\ell i} u^i D_i u^\ell \cdot D^2_{i\ell} u^\ell 
\\
{}+ \sum_{\ell \neq i} \sum_{j \neq \ell} D^3_{ppp} H^j D^2_{ij} u^j D^2_{\ell j} u^j D_j u^\ell \cdot D^2_{i\ell} u^\ell + \frac1N \sum_{j \neq i} D^3_{ppp} H^j D^2_{ij} u^j D^2_{i j} u^j D_j u^i \cdot D^2_{ii} u^i,
\end{gathered}
\end{equation} 
where all the derivatives of $H^j$ are understood to be computed at $(x^j,D_ju^j)$.

We deal with the terms of second type just as in the proof of \Cref{ENS_d2N}, with only one exception, when it comes to control the only one for which we need to prove a lower bound provided by the semimonotonicity assumption (cf.~estimate \eqref{onlylowbound}). In this case, using that \eqref{generalDmon} and {\bfseries(\hyperlink{H0}{H0})} yield (with $\call H_0$ and $\call H_1$ defined as in \eqref{defnH01})
\[
\call H_0 (D^2_{ji} u^i)_{1 \leq i,j \leq N} \geq - \tilde M \msfb I - \call H_1 \geq -(\tilde M + dC_H) \msfb I,
\]
we estimate
\[
\sum_{\ell \neq i} \sum_j D^2_{i\ell} u^\ell D^2_{pp} H^j D^2_{ij} u^j D_{j\ell} u^\ell \geq -4(\tilde M + dC_H) w^i.
\]
Finally, we claim that the terms of the third type are bounded by $\tilde C(w^i+\frac1N)$, with $\tilde C$ depending only on $c_1$, $C_2$, and $C_H$. Using \Cref{ENS_uijest,ENS_d2b} (which we have just extended to general $H^i$) and {\bfseries(\hyperlink{H0}{H0})}, it is easy to prove that the claim is true for all the terms in \eqref{extratermsB} but the penultimate one, so we only show how to bound that more delicate term: we have, exploiting the Cauchy--Schwarz and Young's inequalities,
\[
\begin{multlined}[.9\displaywidth]
\biggl| \sum_{\ell \neq i} \sum_{j \neq \ell} D^3_{ppp} H^j D^2_{ij} u^j D^2_{\ell j} u^j D_j u^\ell \cdot D^2_{i\ell} u^\ell \biggr| \\
\begin{split}
&\leq C_H\biggl( 2\biggl(\, \sum_{\substack{j,\ell \\ j \neq \ell}} |D^2_{\ell j} u^j D_j u^\ell|^2 \biggr)^\frac12 + \norm{D^2_{ii} u^i}_\infty^2 \biggr) w^i + \frac{C_H}2 \sum_{\ell \neq i} |D^2_{\ell i}u^i D_i u^\ell|^2 
\\[-3pt]
&\leq C_H \biggl( 2\,\sup_{\substack{j,\ell \\ j \neq \ell}} |D_j u^\ell| \biggl( N \sup_{j} \sum_{\ell \neq j} |D^2_{\ell j} u^j|^2 \biggr)^{\frac12} + C_2 \biggr) w^i + \frac{C_HC_2}{2N}
\\
&\leq C_H\Bigl( \bigl( 2c_1 C_2^{\frac12} + C_2 \bigr) w^i + \frac{C_2}{2N} \Bigr).
\end{split}
\end{multlined}
\]
The rest of the proof proceed as in that of \Cref{ENS_d2N}, hence we have \eqref{generald2'}.

\emph{Extending \Cref{ENS_uijestp,ENS_d2N',ENS_d3b}.} The arguments used in the proofs of these propositions are very similar to those used to prove \Cref{ENS_d2b}. In fact, to prove the corresponding extensions to general Hamiltonians, it suffices to adapt the proofs as we have already discussed for that of \Cref{ENS_d2b}. We conclude that the remaining estimates \eqref{generald1}, \eqref{generald2''}, and \eqref{generald3} hold as well.
\end{proof}

We now proceed with the second step of our road map, by showing that we indeed have propagation of a Lipschitz bound for $(D_iu^i)_{1 \leq i \leq N}$. In the proof, we are also going to extend (also using assumption {\bfseries(\hyperlink{H1}{H1})} or {\bfseries(\hyperlink{H1'}{H1$\bs{'}$})}) the arguments employed to prove the improvement of semimonotonicity, as an intermediate step to obtain the desired Lipschitz control.

Actually, the following result is stated in terms of property \eqref{generalDmon} rather than the aforementioned Lipschitz bound, since the former is our ultimate goal to reach; nevertheless, in light of the equivalent formulation \eqref{genoslequiv} of \eqref{generalDmon}, the two propagations are in fact equivalent. 

\begin{prop}[Propagation of the one-sided Lipschitz bound on the drift] \label{ENS_lipbprop} Assume {\bfseries(\hyperlink{MF}{MF})} and {\bfseries(\hyperlink{H0}{H0})}. Also, assume either {\bfseries(\hyperlink{DS}{DS})} and {\bfseries(\hyperlink{H1}{H1})}, or {\bfseries(\hyperlink{LS}{LS})} and {\bfseries(\hyperlink{H1'}{H1$\bs'$})}. Then, if the semimonotonicity constants of $f$ and $g$ are sufficiently small with respect to $T$ (or vice versa), and $N$ is sufficiently large, then there exists a constant $\tilde M$ (independent of $N$) such that, for any $\tau \in [0,T)$, if
\begin{equation} \label{genoslipconthm}
\sum_{1 \leq i \leq N} D_p H^i(\var^i,D_iu^i(t,\var))\bigr|^{\bs x}_{\bs y} \cdot (x^i-y^i) \geq -\tilde M |\bs x - \bs y|^2 \qquad \forall\, t \in (\tau,T], \ \bs x,\bs y \in (\R^d)^N
\end{equation}
then
\[
\sum_{1 \leq i \leq N} D_p H^i(\var^i,D_iu^i(\tau,\var))\bigr|^{\bs x}_{\bs y} \cdot (x^i-y^i) \geq -\frac12 \tilde M |\bs x - \bs y|^2  \qquad \forall\, \bs x,\bs y \in (\R^d)^{N}.
\]
\end{prop}

\begin{rmk} \label{rmk.vicevgen}
Like in the case of propagation of semimonotonicity with purely quadratic Hamiltonians, throughout the proof we are going to obtain explicit relationships between the semimonotonicity constants of the data and the horizon. In this more general case as well, the quantities needed to be small will be $TM_g + T^2M_f$ (in the \Ds semimonotone case) and $(T\kappa_g + T^2\kappa_f)e^{\gamma T}$ (in the \LLs semimonotone case). This also explains the meaning of the \emph{vice versa} part in the statement above.
\end{rmk}

\begin{proof}[Proof of \Cref{ENS_lipbprop}]
First of all, we claim that it suffices to prove that, if the semimonotonicity constants of $f$ and $g$ are sufficiently small with respect to $T$ (or vice versa, in the sense discussed in \Cref{rmk.vicevgen}), then there exists a constant $L$, independent of $N$ and $\tilde M$ such that, if \eqref{genoslipconthm} holds, then
\begin{equation} \label{genproptoprove}
\sum_{1 \leq i \leq N} |D_i u^i(\tau,\bs x) - D_i u^i(\tau,\bs y)|^2 \leq L^2 |\bs x - \bs y|^2 \qquad \forall\, \bs x, \bs y \in (\R^d)^N.
\end{equation}
Indeed, this is equivalent to $-L\msfb I \leq (D_{ji} u^i)_{1 \leq i,j \leq N} \leq L\msfb I$, and thus, recalling that \eqref{generalDmon} is equivalent to \eqref{genoslequiv}, it implies that
\[
\sum_{1 \leq i \leq N} D_p H^i(\var^i,D_iu^i(t,\var))\bigr|^{\bs x}_{\bs y} \cdot (x^i-y^i) \geq -C_H(d+L) |\bs x - \bs y|^2.
\]
Therefore, if \eqref{genproptoprove} holds, then we have the thesis proved upon taking $\tilde M \geq 2C_H(d+L)$.

With the goal of proving \eqref{genproptoprove}, we start by studying the evolution along optimal trajectories of $z^i \defeq \frac12 \bigl|D_iu^i \bigr|^{\bs x}_{\bs y} \bigr|^2$. Applying Dynkin's formula \eqref{ENS_itow} with the synchronous coupling \eqref{ENS_Lsync} we obtain
\[
\begin{multlined}[.9\displaywidth]
\bb E\bigl[ z^i(T,\bs X_T, \bs Y_t) \bigr] 
\\
\begin{split}
&\geq z^i(\tau,\bar{\bs x},\bar{\bs y}) +  \bb E \biggl[\,\int_\tau^T D_x H^i(\var^i,D_i u^i) \bigr|^{\bs X_t}_{\bs Y_t} \cdot D_i u^i(t,\var) \bigr|^{\bs X_t}_{\bs Y_t}\,\di t \biggr] -  \bb E \biggl[\,\int_\tau^T D_i f^i \bigr|^{\bs X_t}_{\bs Y_t} \cdot D_i u^i(t,\var) \bigr|^{\bs X_t}_{\bs Y_t}\,\di t \biggr]
\\
&\quad + \sum_{j \neq i} \, \bb E \biggl[\, \int_\tau^T \bigl( D_{pp}^2 H^j(\var^j,D_j u^j) D^2_{ij} u^j(t,\var) D_j u^i(t,\var)\bigr)\bigr|^{\bs X_t}_{\bs Y_t} \cdot D_i u^i(t,\var) \bigr|^{\bs X_t}_{\bs Y_t}\,\di t \biggr].
\end{split}
\end{multlined}
\]
Let us denote the last sum by $E^i$; using the fundamental theorem of calculus and then the Cauchy--Schwarz inequality several times (as well as assumption {\bfseries(\hyperlink{H0}{H0})}) show that $\bigl|\sum_i E^i\bigr|$ is bounded by $C_H$ times
\[
\begin{split}
&\sqrt{N} \Biggl( \sup_{\substack{ij \\ j \neq i}}\, \norm{D_j u^i}_\infty \biggl( \biggl(\, \sum_{\substack{ij \\ j\neq i}} \norm{D(D_{ij} u^j)}_\infty^2 \biggr)^{\frac12} + \sqrt{N}  \sup_i \bigg\lVert \sum_{j \neq i} \big| D^2_{ij} u^j\big|^2 \bigg\rVert_\infty^{\frac12} \biggr)
\\
&\ \ + \Bigl( 1+ \sup_{\substack{ij \\ j \neq i}} \, \norm{D_j u^i}_\infty  \Bigr) \biggl(\,\sup_i \bigg\lVert \sum_{j \neq i} \big| D^2_{ij} u^j\big|^2  \bigg\rVert_\infty \sup_i\sum_{\substack{1\leq j \leq N \\ j \neq i}}\norm{D(D_j u^i)}_\infty^2 \biggr)^{\frac12} \Biggr) \sup_{t \in [\tau,T]} \bb E\bigl[ \bigl| \bs X_t - \bs Y_t \bigr|^2 \bigr],
\end{split}
\]
where the $L^\infty$ norms are understood with respect to the domain $[\tau,T] \times (\R^d)^N$. Therefore, using the estimates in \Cref{generealAPesti},
\[
\Bigl|\, \sum_i E^i \Bigr| \leq \frac{C(\tilde M)}{\sqrt N} \sup_{t \in [\tau,T]} \bb E\bigl[ \bigl| \bs X_t - \bs Y_t \bigr|^2 \bigr],
\]
where the constant $C(\tilde M)$ depends in particular on $\tilde M$ but not on $N$, so that, letting $z \defeq \sum_i z^i$, we can deduce that, for $N$ large enough,
\begin{equation} \label{estonz1}
z(\tau,\bbs x, \bbs y) \leq \bigl( C_g^2 + T(C_f^2 + C_H^2) \bigr) \!\sup_{t \in [\tau,T]} \!\bb E\bigl[ \bigl| \bs X_t - \bs Y_t \bigr|^2 \bigr] + 2(1+C_H) \int_\tau^T \bb E\bigl[ z(\var,\bs X, \bs Y) \bigr],
\end{equation}
where
\[
C_g^2 \defeq \sup_{\bs x \neq \bs y} \frac{\sum_i \bigl| D_i g^i \bigr|^{\bs x}_{\bs y} \bigr|^2}{|\bs x - \bs y|^2}
\]
and $C_f$ is defined likewise; note that we have the bounds $C_g^2 \leq 2L_g^2$ and $C_f \leq 2L_f^2$, which are independent of $N$, because, using {\bfseries(\hyperlink{MF}{MF})},
\[
\sum_i \bigl| D_i g^i \bigr|^{\bs x}_{\bs y} \bigr|^2 \leq \sum_{ij}\, \norm{D^2_{ji} g^i}^2_\infty |\bs x - \bs y|^2 \leq 2L_g^2 |\bs x - \bs y|^2,
\]
and likewise for $f$.

We now proceed in two slightly different ways, depending on the kind of semimonotonicity that holds. In both case, we are going to use the (generalizations of the) estimates developed for the improvement of semimonotonicity in order to obtain a useful bound on $\sup_{[\tau, T]} \bb E \bigl[ |\bs X - \bs Y|^2 \bigr]$.

\emph{\Ds semimonotone case.}
Follow the proof of \Cref{ENS_dscprop} to obtain, with the notation used therein,
\[
\begin{multlined}[.95\displaywidth]
\bb E\bigl[w^i(T,\bs X_T,\bs Y_T)\bigr] 
\\
\begin{split}
&\leq w^i(\tau,\bbs x, \bbs y) - \int_\tau^T \bb E\Big[ D_i f^i\big|_{\bs Y}^{\bs X} \cdot (X^i - Y^i) \Big] + \int_\tau^T \bb E\, \tildee{\call E}^i_t \,\di t
\\
&\quad + \int_\tau^T \bb E \biggl[ D_x H^i\Bigr|^{(X^i_t,D_iu^i(t,\bs X_t))}_{(Y^i_t,D_iu^i(t,\bs Y_t))} \cdot (X^i_t - Y^i_t) - D_p H^i\Bigr|^{(X^i_t,D_iu^i(t,\bs X_t))}_{(Y^i_t,D_iu^i(t,\bs Y_t))} \cdot D_i u^i(t,\var)\bigr|^{\bs X_t}_{\bs Y_t} \biggr]\,\di t,
\end{split}
\end{multlined}
\]
where $\tildee{\call E^i}$ is an error term for which, arguing as done to control $\sum_i \call E^i$ and using \Cref{generealAPesti}, we have the following bound:
\[
 \int_\tau^T \bb E \biggl[\, \sum_{i} \tildee{ \call E}_t^i\, \biggr]  \leq \frac{\overline C(\tilde M)}{\sqrt{N}}\,\sup_{t \in [\tau,T]} \bb E\big[ |\bs X_t - \bs Y_t|^2 \big],
\]
with $\overline C(\tilde M)$ independent of $N$. Therefore, using assumptions {\bfseries(\hyperlink{H1}{H1})} and {\bfseries(\hyperlink{DS}{DS})}, and the fact that, by the respective definitions and Young's inequality, $w(t,\bs x,\bs y) \leq \epsilon z(t,\bs x,\bs y) + \frac1{2\epsilon} |\bs x- \bs y|^2$ for any $\epsilon > 0$, we have
\[
\begin{multlined}[.9\displaywidth]
\lambda_H \sum_i \int_\tau^T \bb E \Bigl[ \Bigl| D_p H^i\bigr|^{(X^i_t,D_iu^i(t,\bs X_t))}_{(Y^i_t,D_iu^i(t,\bs Y_t))} \Bigr|^2 \Bigr] \,\di t 
\\
\leq \Bigl(M_g + TM_f + \frac{\overline C(\tilde M)}{\sqrt{N}}\Bigr) \sup_{t \in [\tau,T]} \bb E\big[ |\bs X_t - \bs Y_t|^2 \big] + \epsilon z(\tau,\bbs x,\bbs y) + \frac1{2\epsilon} |\bbs x- \bbs y|^2.
\end{multlined}
\]

We now estimate, using the equation of $\bs X-\bs Y$,
\begin{equation} \label{estonXY2}
|\bs X_s - \bs Y_s|^2 \leq 2|\bbs x - \bbs y|^2 + 2T \sum_i \int_\tau^s \Bigl| D_p H^i\bigr|^{(X^i_t,D_iu^i(t,\bs X_t))}_{(Y^i_t,D_iu^i(t,\bs Y_t))} \Bigr|^2\,\di t \qquad \forall s \in [\tau,T],
\end{equation}
whence we obtain
\[
\begin{multlined}[.9\displaywidth]
\lambda_H \sum_i \int_\tau^T \bb E \Bigl[ \Bigl| D_p H^i\bigr|^{(X^i_t,D_iu^i(t,\bs X_t))}_{(Y^i_t,D_iu^i(t,\bs Y_t))} \Bigr|^2 \Bigr] \,\di t 
\\
\begin{split}
&\leq 2\Bigl(TM_g + T^2M_f + \frac{T\overline C(\tilde M)}{\sqrt{N}}\Bigr)\sum_i \int_\tau^T \bb E \Bigl[ \Bigl| D_p H^i\bigr|^{(X^i_t,D_iu^i(t,\bs X_t))}_{(Y^i_t,D_iu^i(t,\bs Y_t))} \Bigr|^2 \Bigr] \,\di t 
\\
&\quad + \epsilon z(\tau,\bbs x,\bbs y) + \Bigl(2M_g + 2TM_f + \frac{2\overline C(\tilde M)}{\sqrt{N}} + \frac1{2\epsilon} \Bigr) |\bbs x- \bbs y|^2.
\end{split}
\end{multlined}
\]
If $TM_g$ and $T^2M_f$ are small enough, and $N$ is large enough, one can then deduce
\[
\sum_i \int_\tau^T \bb E \Bigl[ \Bigl| D_p H^i\bigr|^{(X^i_t,D_iu^i(t,\bs X_t))}_{(Y^i_t,D_iu^i(t,\bs Y_t))} \Bigr|^2 \Bigr] \,\di t 
\leq \frac{2\epsilon}{\lambda_H} z(\tau,\bbs x,\bbs y) + \Bigl( \frac1T + \frac1{\epsilon\lambda_H} \Bigr) |\bbs x- \bbs y|^2
\]
which, plugged back in to \eqref{estonXY2}, gives
\[
\sup_{t \in [\tau, T]} \bb E \bigl[ |\bs X_t - \bs Y_t|^2 \bigr] \leq 2 \Bigl( 1 + \frac{T}{\epsilon\lambda_H} \Bigr)|\bbs x - \bbs y|^2 + \frac{4\epsilon T}{\lambda_H} z(\tau,\bbs x,\bbs y).
\]

Using this bound in the right-hand side of \eqref{estonz1} (with $\epsilon$ small enough with respect to $T$, $C_g$, $C_f$, and $C_H$), we obtain
\[
z(\tau,\bbs x, \bbs y) \leq C |\bbs x - \bbs y|^2 + \frac12 z(\tau, \bbs x, \bbs y) + 2(1+C_H) \int_\tau^T \bb E\bigl[ z(\var,\bs X, \bs Y) \bigr],
\]
with $C$ independent of $N$ (and $\tilde M$). The desired control \eqref{genproptoprove} (with $L^2 = 2Ce^{2(1+C_H)T}$) now follows by invoking Gronwall's inequality.

\emph{\LLs semimonotone case.} First we note that the arguments in the proof of \Cref{ENS_cruboprop} also provide an upper bound for $\msf V \defeq D_{ii}^2 u^i$. For general Hamiltonian, the first part of the proof goes along the same lines, with only minor changes; in particular, one obtains, for $N$ large enough,
\begin{equation} \label{estonVgen}
\bb E \biggl[ \,\int_\tau^T |\msf V(t, \bs X_t)|^2\,\di t \biggr] \leq K,
\end{equation}
for some constant $K$ independent of $\tilde M$ and $N$.
We are only going to discuss the main differences in the last part of the proof. The equation of $\msf V$ is
\[
\begin{multlined}[t][.95\displaywidth]
-\de_t \msf V - \tr((\sigma\msfb I + \beta \msfb J)D^2 \msf V) + \sum_{1 \leq j \leq N} \! D_j u^j D_{j} \msf V + D^2_{xx} H^i + D_{xp}^2 H^i \msf V + \msf V D_{px}^2 H^i + \msf V D_{pp}^2 H^i \msf V \\
= D_{ii}^2 f^i - 2 \sum_{\substack{1 \leq j \leq N \\ j \neq i}}\! D^2_{pp} H^j D_{ij}^2 u^i D_{ji}^2 u^j - \sum_{\substack{1 \leq j \leq N \\ j \neq i}}\! D_{pp}^2 H^j D_{iij}^3 u^j D_{j} u^i - \sum_{\substack{1 \leq j \leq N \\ j \neq i}}\! D_{ppp}^3 H^j (D_{ij}^2 u^j)^2 D_{j} u^i ,
\end{multlined}
\]
where the derivatives of $H^i$ are understood to be computed at $(x^i, D_i u^i)$. The first two sums on the right-hand side can be estimated as done in the case of purely quadratic Hamiltonian, while the third one is easily seen to be bounded by $K(\tilde M)/N$ (using {\bfseries(\hyperlink{H0}{H0})}, \eqref{generald1} and \eqref{generald2'}), with $K(\tilde M)$ independent of $N$. On the other hand, using {\bfseries(\hyperlink{H0}{H0})} and {\bfseries(\hyperlink{H1'}{H1$\bs'$})},
\[
-\bigl(C_H(1 + C_H) + |\msf V|^2\bigr) \msf I_d \leq D^2_{xx} H^i + D_{xp}^2 H^i \msf V + \msf V D_{px}^2 H^i + \msf V D_{pp}^2 H^i \msf V \leq -C_H\bigl((1 + C_H) + |\msf V|^2\bigr)\msf I_d.
\]
Therefore, arguing as in the proof of \Cref{ENS_cruboprop} and using \eqref{estonVgen}, we conclude that, for $N$ large enough, there exists $\bar K > 0$ (independent of $\tilde M$ and $N$) such that
\begin{equation} \label{D2bddgen}
-\bar K\msf I_d \leq D_{ii}^2 u^i \leq \bar K \msf I_d.
\end{equation}

Follow now the proof of \Cref{ENS_LLmprop} and observe that the crucial difference consists in the non-negative term
\[
\sum_{i} \big| D_i u^i(\bs x) - D_i u^i(\bs y^{-i}, x^i) \big|^2
\]
appearing in \eqref{eqwnnterms} being replaced by
\[
\sum_i \Bigl( H^i(\var^i,D_i u^i)\bigr|^{\bs x}_{(\bs y^{-i}, x^i)} - D_p H^i(x^i,D_iu^i(\bs y^{-i}, x^i)) \cdot D_i u^i \bigr|^{\bs x}_{(\bs y^{-i}, x^i)} \Bigr),
\]
and the same for the other term obtained by swapping $\bs x$ and $\bs y$. By assumption {\bfseries(\hyperlink{H1'}{H1$\bs{'}$})}, such a term is bounded from below by
\[
\Lambda_H \sum_i \Bigl|D_i u^i \bigr|^{\bs x}_{(\bs y^{-i}, x^i)} \Bigr|^2 ;
\]
therefore, by Dynkin's formula with synchronous coupling, we obtain the estimate
\begin{equation} \label{estLLdisgen}
\Lambda_H \sum_i \int_\tau^T \bb E \Bigl[ \Bigl|D_i u^i \bigr|^{\bs X}_{(\bs Y^{-i}, X^i)} \Bigr|^2 \Bigr] \leq \Bigl(\kappa_g + T\kappa_f + \frac{\overline C'(\tilde M)}{\sqrt{N}}\Bigr) \sup_{t \in [\tau,T]} \bb E\big[ |\bs X_t - \bs Y_t|^2 \big] + \LLop[u(\tau,\var)](\bbs x, \bbs y).
\end{equation}

Use now \eqref{estonXY2} and \eqref{D2bddgen} to deduce that, for all $s \in [\tau, T]$, 
\[
|\bs X_s - \bs Y_s|^2 \leq 2|\bbs x - \bbs y|^2 + 6TC_H^2 \int_\tau^s \Bigl( \bigl(1+\bar K^2\bigr) |\bs X_t - \bs Y_t|^2 + \sum_i \Bigl|D_i u^i \bigr|^{\bs X}_{(\bs Y^{-i}, X^i)} \Bigr|^2 \Bigr),
\]
so that Gronwall's inequality gives
\[
\sup_{t \in [\tau, T]}\bb E\bigl[ |\bs X_t - \bs Y_t|^2 \bigr] \leq 2|\bbs x - \bbs y|^2 + 6TC_H^2e^{6TC_H^2(1+\bar K^2)} \sum_i \int_\tau^T \bb E \Bigl[ \Bigl|D_i u^i \bigr|^{\bs X}_{(\bs Y^{-i}, X^i)} \Bigr|^2 \Bigr].
\]
Combine this with \eqref{estLLdisgen} to have, for $(T\kappa_g + T^2\kappa_f)e^{6TC_H^2(1+\bar K)}$ small and $N$ large,
\[
\sup_{t \in [\tau,T]} \bb E\bigl[ |\bs X_t - \bs Y_t|^2 \bigr] \leq C'|\bbs x - \bbs y|^2 + C''\LLop[u(\tau,\var)](\bbs x, \bbs y),
\]
with $C'$ and $C''$ independent of $N$ and $\tilde M$.

At this point we go back to \eqref{estonz1} and use Gronwall's inequality to obtain
\begin{equation} \label{estforzLL}
z(\tau,\bbs x, \bbs y) \leq \tilde C'|\bbs x - \bbs y|^2 + \tilde C''\LLop[u(\tau,\var)](\bbs x, \bbs y),
\end{equation}
for some $\tilde C'$ and $\tilde C''$ independent of $N$ and $\tilde M$. By using (twice) the fundamental theorem of calculus, then the control \eqref{D2bddgen}, and finally Young's inequality, we have
\begin{equation*} %\label{LtozLL}
\begin{split}
\LLop[u(\tau,\var)](\bbs x, \bbs y) &= \sum_{i} \int_0^1 D_{i} u^i(\tau,\var)\bigr|^{(\bbs x^{-i},sx^i + (1-s)y^i)}_{(\bbs y^{-i},sx^i + (1-s)y^i)} \cdot (\bar x^i - \bar y^i)\,\di s  
\\
&\leq \sum_{i} \int_0^1 D_{i} u^i(\tau,\var)\bigr|^{\bbs x}_{\bbs y} \cdot (\bar x^i - \bar y^i)\,\di s + 2\bar K |\bbs x - \bbs y|^2
%\\
%&\quad - \int_0^1 \int_0^1 (1-s)\Bigl( D_{ii}^2 u^i\bigl(\tau,\bbs x^{-i},(1-ss')x^i + ss'y^i\bigr) + D_{ii}^2 u^i\bigl(\tau,\bbs y^{-i},ss'x^i + (1-ss')y^i\bigr) \bigr)(x^i-y^i) \cdot (\bar x^i - \bar y^i)\,\di s'\di s 
\\
&\leq \frac1{2\tilde C''}\, z(\tau, \bbs x, \bbs y) + (\tilde C'' + 2\bar K) |\bbs x - \bbs y|^2;
\end{split}
\end{equation*}
therefore, coming back to \eqref{estforzLL} we obtain
\[
z(\tau,\bbs x, \bbs y) \leq 2(\tilde C' + \tilde C'' + 2\bar K) |\bbs x - \bbs y|^2,
\]
which is \eqref{genproptoprove} with $L^2 = 2(\tilde C' + \tilde C'' + 2\bar K)$ by the arbitrariness of $\bbs x, \bbs y \in (\R^d)^N$.
\end{proof}

At this stage, that is with the estimates on the derivatives assuming one-sided Lipschitz bound on the drift, and the propagation of the one-sided Lipschitz bound on the drift, we can conclude the proof of Theorem \ref{ENS_mainthm} in the case of general $H^i$, arguing as in the proof of Theorem \ref{ENS_mainthmDS}.

\appendix

\section{An existence and uniqueness theorem for the Nash system}

We present an existence theorem for \eqref{ENS_NS}. As discussed in the Introduction, existence of solutions to the Nash system is rather delicate if the domain is unbounded and the growth with respect to the gradient is quadratic, as in our case. We will make use of the a priori estimates provided by \Cref{ENS_mainthm}, together with a (standard) continuation argument.

\begin{thm} \label{ENS_thmexbdd}
Let $f^i \in C^{\gamma+\alpha}((\R^d)^N)$ and $g^i \in C^{2+\gamma+\alpha}((\R^d)^N)$, for some integer $\gamma > 1$ and $\alpha \in (0,1)$. Let $H^i \in C^{1+\gamma+\alpha}(\R^d \times \Omega)$ for each $\Omega \ssubset \R^d$; suppose that $H^i$ are bounded from below and satisfy {\bfseries(\hyperlink{H0}{H0})}. Assume that condition {\bf(\hyperlink{MF}{MF})} also holds, as well as one of the following two pairs of assumptions: either {\bf(\hyperlink{DS}{DS})}--{\bf(\hyperlink{H1}{H1})} or {\bf(\hyperlink{LS}{LS})}--{\bf(\hyperlink{H1'}{H1$\bs'$})}. Denote by $c_f$ and $c_g$ the corresponding semimonotonicity constants of $f$ and $g$, respectively.

For any $T>0$, there are positive constants $c_f^*$ and $c_g^*$, and $N_* \in \N$ such that, if $c_f \leq c_f^*$, $c_g \leq c_g^*$ and $N \geq N_*$, then there exists a unique (bounded) solution to the Nash system~\eqref{ENS_NS} which is of class $C^{1+\frac\gamma2+\frac\alpha2, 2+\gamma+\alpha}$.
\end{thm}

\begin{proof} \emph{Step 1.}
For $\tau \in [0,T)$, let $\call X \defeq \call X_\tau = C([\tau,T]; C_b^1 ((\R^d)^N) \, )^N$, with norm $\norm{v}_{\call X} \defeq \sup_{1 \leq i \leq N} ( \norm{v^i}_\infty + \norm{Dv^i}_\infty )$. For $v \in \call X$, let $u^i = \call S(v)^i$ be the solution of the parabolic equation
\[
\begin{cases}
-\de_t u^i - \tr((\sigma\msfb I + \beta \msfb J)D^2 u^i)  = f^i(\bs x) - w^i(t,\bs x) \\
 u^i|_{t=T} = g^i,
\end{cases}\]
with
\[
w^i(t, \bs x) =  H^i(x^i,D_i v^i(t,\bs x)) + \sum_{\substack{1 \leq j \leq N \\ j \neq i}} D_pH^j(x^j,D_j v^j(t,\bs x)) \cdot D_j v^i(t,\bs x).
\]
By \cite[Theorem 5.1.2]{Lunardi}, $\call S(v) \in \call X$ (solutions actually enjoy further H\"older regularity). We have the following standard representation formula: let $\Phi$ be the fundamental solution to the diffusion equation $\de_t u - \tr((\sigma\msfb I + \beta \msfb J)D^2u) = 0$ on $[0,T] \times (\R^d)^N$; that is,
\[
\Phi(t,\bs x) \defeq \frac1{\sqrt{\det(4\pi(\sigma\msfb I + \beta \msfb J)t)}\,} \,e^{-\tfrac{(\sigma\msfb I + \beta \msfb J)^{-1} \bs x \cdot \bs x}{4t}} = \frac1{\sqrt{1+Nd\beta/\sigma}\,} \frac{1}{(4\pi \sigma t)^{\frac{Nd}2}} e^{-\tfrac{\left(\msfb I - \frac{\beta}{\sigma+Nd\beta} \msfb J\right) \bs x \cdot \bs x}{4\sigma t}} .
\]
Then for $v \in \call X$ and $i \in \{1,\dots,N\}$, we have 
\[
\call S(v)^i(t,\bs x) \defeq (\Phi(T-t,\var) \star g^i)(\bs x) + \int_t^T (\Phi(T-s,\var) \star (f^i-w^i(s-t,\var))(\bs x) \,\di s,
\]
where $\star$ denotes the convolution operator over $(\R^d)^N$. Then
\[
\norm{\call S(v)^i}_\infty \lesssim \norm{g}_{\call X} + (T-\tau)(\norm{f}_{\call X} + \norm{v}_{\call X}^2),
\]
where the implied constant depends only on $N$, $C_H$ and $d$. Also, for $j \in \{1,\dots,N\}$ and $k \in \{1,\dots,d\}$
\[\begin{split}
D_{x^{jk}} \call S(v)^i(t,\bs x) &= (\Phi(T-t,\var) \star D_{x^{jk}}g^i)(\bs x) + \int_t^T (\Phi(T-s,\var) \star D_{x^{jk}} f^i)(\bs x) \,\di s \\[-3pt]
&\quad - \int_t^T (D_{x^{jk}}\Phi(T-s,\var) \star w^i(s-t,\var))(\bs x) \,\di s;
\end{split}\]
then, using that $\int_{(\R^d)^N} \lvert D_{x^{jk}} \Phi(t,\bs y)\vert\,\di \bs y \leq Nd/{\sqrt{\pi \sigma t}\,}$, one has
\[
\norm{D\call S(v)^i}_\infty \lesssim \norm{g}_{\call X} + (T-\tau)\norm{f}_{\call X} + \sqrt {T-\tau}\, \norm{v}^2_{\call X},
\]
the implied constant depending only on $\sigma$, $N$, $C_H$ and $d$. %This proves that $v \mapsto \call S(v)$ well defines a map $\call S \colon \call X \to \call X$. 
Furthermore, given $v,\bar v \in \call X$ similar computations show that
\[
\norm{\call S(v) - \call S(\bar v)}_{\call X} \lesssim (1+\sqrt{T-\tau}\,)\sqrt{T-\tau}\max\{\norm{v}_{\call X},\norm{\bar v}_{\call X}\} \norm{v-\bar v}_{\call X},
\]
with implied constant depending only on $C_H, \sigma$, $N$ and $d$. Letting $\call B_c \defeq \{ v \in \call X :\ \norm{v}_{\call X} \leq c \}$, we deduce that $\call S \colon \call B_c \to \call B_c$ is a contraction for any $c = c(\norm{g}_{\call X},\sigma, C_H, N,d)$ sufficiently large and $\tau = \tau(c,\sigma,N,d, \beta, C_H, \norm{f}_{\call X})$ such that $T-\tau$ is sufficiently small. It follows by the Banach--Caccioppoli theorem that $\call S$ has a unique fixed point $u \in \call B_c$, which by the definition of $\call S$ solves system \eqref{ENS_NS} in the mild sense. Then note that by standard parabolic Schauder estimates we can bootstrap up to the desired regularity.

\emph{Step 2.} In order to conclude, it is sufficient to suppose that $u$ be a solution on $[\bar\tau, T] \times (\R^d)^N$ (with arbitrary $\bar\tau \in [0,T)$) and show that
\begin{equation} \label{todostep2}
\norm{u}_{\call X_{\bar\tau}} \leq C
\end{equation}
for some constant $C$ independent of $\bar\tau$. Indeed, in such a case, one proceeds as follows. Define
\[
\tau_* \defeq \inf\{ \tau \in [0,T] :\ \text{\eqref{ENS_NS} has a unique solution belonging to $\call X_\tau$} \}
\]
and notice that we proved that $\tau^* < T$; seeking for a contradiction, suppose that $\tau_* > 0$. Let $\epsilon \in (0,T-\tau_*)$ to be announced and consider a solution $u \in \call X_{\tau_* + \epsilon}$. Knowing that $\norm{u}_{\call X_{\tau_* + \epsilon}} \leq C$, one can redo the fixed point argument in Step $1$ with $u(\tau_*+\epsilon,\cdot)$ in lieu of $g$ and conclude that $u$ extends to the horizon $[\tau',T]$, provided that $\tau_* + \epsilon - \tau' \leq \delta$ for some $\delta > 0$ which is independent of $\tau_*$ and $\epsilon$. So one can fix $\delta_0 = \frac{\tau_*}2 \wedge \delta$, $\epsilon = \frac12({\delta_0} \wedge (T-\tau_*))$ and $\tau' = \tau_* - \frac{\delta_0}2 \in (0, \tau_*)$, thus contradicting the definition of $\tau_*$.

The uniform control on $|Du^i|$ required by \eqref{todostep2}, directly comes from \Cref{ENS_mainthm}, provided that we choose $c_f^*$ and $c_g^*$ small enough and $N_*$ large enough. Thus, we only need to bound $|u^i|$; to do so, recall the representation formula for the solution of the Nash system: for any $(s,\bs z) \in [\bar\tau,T] \times (\R^d)^N$,
\begin{equation} \label{ENS_reprAP}
\begin{split}
u^i(s,\bs z) &= \int_s^T \int_{(\R^d)^N} \Bigl(\, H^i(x^i,D_i u^i(t,\bs x)) - D_p H^i(x^i,D_i u^i(t,\bs x)) \cdot D_i u^i(t,\bs x) + f^i(\bs x) \Bigr) \rho(t,\di \bs x)\di t 
\\
&\quad + \int_{(\R^d)^N} g^i(\bs x)\,\rho(T,\di\bs x)\,,
\end{split}
\end{equation}
where $\rho = \rho^{s, \bs z}$ solves
\[
\begin{dcases}
\de_t \rho - \tr((\sigma\msfb I + \beta \msfb J)D^2 \rho) + \sum_{1 \leq j \leq N} \mathrm{div}_{x^j} (D_pH^j(x^j,D_j u^j) \rho) = 0 & \text{on} \ \ [s,T] \times (\R^d)^N \\[-3pt]
\rho|_{t=s} = \delta_{\bs z}.
\end{dcases}
\]
This is easily obtained by testing the equation of $\rho$ by $u^i$. In particular, by the convexity of $p \mapsto H^i(x,p)$,
\[
u^i(s,\bs z) \geq \int_s^T \int_{(\R^d)^N} \bigl( H^i(x^i,0) + f^i(\bs x) \bigr) \, \rho(t,\di \bs x)\di t + \int_{(\R^d)^N} g^i(\bs x)\,\rho(T,\di\bs x).
\]
On the other hand, %since
%\[
%-\de_t u^i - \tr((\sigma\msfb I + \beta \msfb J)D^2 u^i) + \sum_{\substack{1 \leq j \leq N \\ j \neq i}} D_pH^j(x^j,D_j u^j) \cdot D_j u^i \leq f^i + L^i(x^i,0),
%\]
if $\hat\rho^i$ solves instead
\[
\begin{dcases}
\de_t \hat\rho^i - \tr((\sigma\msfb I + \beta \msfb J)D^2 \hat\rho^i) + \sum_{\substack{1 \leq j \leq N \\ j \neq i}} \mathrm{div}_{x^j} (D_jH^j(x^j,D_j u^j) \hat\rho^i) = 0 & \text{on} \ \ [s,T] \times (\R^d)^N \\[-7pt]
\hat\rho^i|_{t=s} = \delta_{\bs z},
\end{dcases}
\]
then
\[
u^i(s,\bs z) \leq \int_s^T \int_{(\R^d)^N} \bigl( f^i(\bs x) - \inf_{p \in \R^d} H^i(x^i,p) \bigr) \, \hat\rho^i(t,\di \bs x)\di t + \int_{(\R^d)^N} g^i(\bs x)\,\hat\rho^i(T,\di\bs x).
\]
This proves that $\sup_{t \in [\bar\tau,T]} \norm{u^i(t,\cdot)}_\infty \leq (T+1)\max\{\norm{f^i}_\infty + (\inf H^i)_-, \norm{g^i}_\infty\}$, thus concluding the proof.
\end{proof}

\section{An estimate on the space-time regularity for the heat equation}

The Lemma below shows how to transfer weighted Lipschitz regularity in space of solution to heat equations to H\"older regularity in time, following a standard idea proposed in \cite{Kr69}. We state it for forward equations for simplicity. Note again that H\"{o}lder continuity is independent of $N$.

\begin{lem}\label{lemmaholder} Let $f: [0,T] \times (\R^d)^N \to \R$, $\sigma > 0$, $\beta \geq 0$, and $u$ be a classical solution of
\[
\de_t u - \tr( (\sigma \msfb I + \beta \msfb J) D^2 u) = f \qquad \text{in $(0,T) \times (\R^d)^N$ }.
\]
Assume also that 
\[
|u(t,\bs x) - u(t,\bs y)| \leq c_1 \norm{\bs x - \bs y}^i \qquad \forall\, t \in [0, T], \ \bs x, \bs y \in (\R^d)^N,
\]
and $\|f\|_\infty, \|u\|_\infty \le c_2$
for some $c_1, c_2 > 0$ and $i = 1, \ldots, N$. Then,
\[
|u(\tau,\bs z) - u(s,\bs z)| \leq (c_1 + 4(\sigma+\beta) d c_2) |\tau-s|^{\frac13} + c_2  |\tau-s| \qquad \forall\, \tau, s \in [0, T], \ \bs z \in (\R^d)^N.
\]
\end{lem}

\begin{proof} Assume first that $s = 0$, $\bs z = 0$ and $u(0,0) = 0$. Let $\rho > 0$ and
\[
v(t,\bs x) = c_1 \rho + \frac{4(\sigma+\beta) d c_2}{\rho^2}\, t + c_2 t + \frac{c_2}{\rho^2} (\|x\|^i)^2.
\]
Since
\[
\de_t v - \tr(  (\sigma \msfb I + \beta \msfb J) D^2 v) - f = \frac{4(\sigma+\beta) d c_2}{\rho^2}  + c_2 - \frac{(\sigma+\beta) d c_2}{\rho^2} \Bigl(4 - \frac2N\Bigr) - f \ge 0
\]
everywhere, and
\[
\begin{split}
& v(0, \bs x) \ge c_1 \rho  \ge c_1 \|\bs x\|^i \ge u(0, \bs x) \qquad \text{on $\{ \|\bs x\|^i \le \rho\}$}, \\
& v(t,\bs x) \ge \frac{c_2}{\rho^2} (\|x\|^i)^2 \ge u(t, \bs x) \qquad \text{on $[0,T] \times \{ \|\bs x\|^i = \rho\}$},
\end{split}
\]
by the comparison principle we get $u(t,\bs x) \le v(t, \bs x)$ on $[0,T] \times \{ \|\bs x\|^i \le \rho\}$. For $\tau > 0$, pick then $\rho = \tau^{\frac13}$, so that the previous inequality yields
\[
u(\tau,0)\le c_1  \tau^{\frac13} + 4(\sigma+\beta) d c_2  \tau^{\frac13} + c_2 \tau.
\]
Using $-v$ as a subsolution yields the analogous bound from below. It is now straightforward to conclude the estimate for general $s, \bs z$.
\end{proof}

\bibliography{ENSbib}
\bibliographystyle{abbrv}

\end{document}